\documentclass[a4paper]{amsart}

\usepackage{amsmath,amsthm,amssymb,enumerate}
\usepackage{epsfig}
\usepackage{amssymb}
\usepackage{amsmath}
\usepackage{amssymb}
\usepackage{amsmath,amsthm}
\usepackage[latin1]{inputenc}
\usepackage[T1]{fontenc}
\usepackage{path}
\usepackage{ae,aecompl}
\usepackage{amsfonts}
\usepackage{amsxtra}
\usepackage{euscript,mathrsfs}
\usepackage{color}
\usepackage[left=3.3cm,right=3.3cm,top=4cm,bottom=3cm]{geometry}
\usepackage[citecolor=blue,colorlinks=true]{hyperref}
\allowdisplaybreaks

\usepackage{enumitem}
\setenumerate{label={\rm (\alph{*})}}

\usepackage{amsfonts}
\usepackage{amsxtra}

\usepackage[notcite,notref]{showkeys}

\numberwithin{equation}{section}

\newcommand{\bfu}{\vc{u}}

\newcommand{\bfq}{\vc{q}}

\newcommand{\bfF}{\vc{F}}

\newcommand{\bfphi}{\boldsymbol{\varphi}}
\newcommand{\bfvarphi}{\boldsymbol{\varphi}}

\newcommand{\pas}{\prst\mbox{-a.s.}}

\newcommand{\D}{{\rm d}}
\newcommand{\W}{W}
\newcommand{\intQ}[1]{\int_{Q} #1 \ \dx}
\newcommand{\Q}{\mathbb{T}^3}
\newcommand{\tvre}{\tilde{\vr}_\ep}
\newcommand{\tvte}{\tilde{\vt}_\ep}
\newcommand{\tvue}{\tilde{\vu}_\ep}
\newcommand{\tvrd}{\tilde{\vr}_\delta}
\newcommand{\tvtd}{\tilde{\vt}_\delta}
\newcommand{\tvud}{\tilde{\vu}_\delta}
\newcommand{\vr}{\varrho}
\newcommand{\vu}{\vc{u}}
\newcommand{\vc}[1]{{\bf #1}}
\newcommand{\tor}{\mathbb{T}^3}
\newcommand{\expe}[1]{ \mathbb{E} \left[ #1 \right] }

\newcommand{\Div}{\divergence}
\newcommand{\ep}{\varepsilon}
\newcommand{\R}{\mathbb R}

\newcommand{\E}{\mathbb E}
\newcommand{\p}{\mathbb P}

\newcommand{\dd}{\mathrm{d}}
\newcommand{\dx}{\,\mathrm{d}x}
\newcommand{\dt}{\,\mathrm{d}t}

\newcommand{\8}{\infty}
\newcommand{\tens}{\tn}
\newcommand{\dif}{\mathrm{d}}

\newcommand{\mf}{\mathfrak{F}}

\newcommand{\prst}{\mathbb{P}}

\newcommand{\mn}{\mathbb{N}}

\newcommand{\mt}{\mathbb{T}^3}

\DeclareMathOperator{\diver}{div}
\DeclareMathOperator{\divergence}{div}

\newcommand{\bu}{\mathbf u}

\newcommand{\bfG}{\mathbf G}

\newcommand{\bFormula}[1]{
\begin{equation} \label{#1}}
\newcommand{\eF}{\end{equation}}

\newcommand{\Ov}[1]{\overline{#1}}

\newcommand{\DC}{C^\infty_c}

\newcommand{\aleq}{\stackrel{<}{\sim}}

\newcommand{\vre}{\vr_\ep}

\newcommand{\vue}{\vu_\ep}
\newcommand{\vtd}{\vartheta_\delta}

\newcommand{\cal}{\mathcal}

\newcommand{\vt}{\vartheta}
\newcommand{\vte}{\vt_\ep}

\newcommand{\bfS}{\mathbb S}
\newcommand{\Grad}{\nabla}

\newcommand{\tn}[1]{\mathbb{ #1 }}

\newcommand{\dxdt}{\dx \ \dt}
\newcommand{\intO}[1]{\int_{\mt} #1 \ \dx}

\newcommand{\intTor}[1]{\int_{\tor} #1 \ \dx}

\newcommand{\Del}{\Delta}
\newcommand{\TN}{{\mathbb{T}^3}}

\newtheorem{theorem}{Theorem}[section]
\newtheorem{Theorem}{Theorem}[section]
\newtheorem{proposition}{Proposition}[section]
\newtheorem{Proposition}{Proposition}[section]

\newtheorem{Lemma}{Lemma}[section]

\newtheorem{Remark}{Remark}[section]
\newtheorem{definition}{Definition}[section]

\allowdisplaybreaks

\begin{document}

%\begin{frontmatter}

%\title{Martingale solutions to the full Navier--Stokes--Fourier system with stochastic forcing}
\title{Stochastic Navier--Stokes--Fourier equations}

\author{Dominic Breit}
\address[D. Breit]{
Department of Mathematics, Heriot-Watt University, Riccarton Edinburgh EH14 4AS, UK}
\email{d.breit@hw.ac.uk}

\author{Eduard Feireisl}
\address[E. Feireisl]{Institute of Mathematics of the Academy of Sciences of the Czech Republic, \v Zitn\' a 25, CZ-115 67 Praha 1, Czech Republic}
\email{feireisl@math.cas.cz}

%\author{Martina Hofmanov\'a}
%\address[M. Hofmanov\'a]{Technical University Berlin, Institute of Mathematics, Stra\ss e des 17. Juni 136, 10623 Berlin, Germany}
%\email{hofmanov@math.tu-berlin.de}

\begin{abstract}
We study the full Navier--Stokes--Fourier system governing the motion of a general viscous, heat-conducting, and compressible fluid subject to
stochastic perturbation. Stochastic effects are implemented through (i) random initial data, (ii) a forcing term in the momentum equation represented by a
multiplicative white noise, (iii) random heat source in the internal energy balance.
We establish existence of a weak martingale solution under physically grounded structural assumptions.
As a byproduct of our theory we can show that stationary martingale solutions only exist in certain trivial cases.
\end{abstract}


\date{\today}

\maketitle

%\end{frontmatter}

\section{Introduction}

There is an abundant amount of literature concerning notably the stochastic perturbations of the Navier--Stokes system in the context of incompressible
fluid flows starting with \cite{BeTe}. The existence of weak martingale solutions -- these solutions are weak in the analytical sense (derivatives only exist in the sense of distributions) and weak in the probabilistic sense (the underlying probability space is part of the solution) --
was first shown in \cite{franco}.
Definitely much less is known if compressibility of the
fluid is taken into account. Similarly to \cite{BeTe}, first existence results were based on a suitable transformation formula.
It allows to reduce the problem to a random system of PDEs, where the stochastic integral does no longer appear. Existence of solutions in this semi-deterministic setting have been shown in \cite{feireisl2} (see also \cite{MR1760377} for the 1D case in the Lagrange coordinates and \cite{MR1807944} for a rather artificial periodic 2D case).
The first ``truly'' stochastic existence result for the compressible Navier--Stokes system perturbed by a general nonlinear multiplicative noise was obtained by Breit and Hofmanov\'a \cite{BrHo} for periodic boundary conditions (see \cite{2015arXiv150400951S} for the extension to the homogeneous Dirichlet boundary conditions and \cite{romeo}, where the system is studied on the whole space). It is based on the concept of finite energy weak
martingale solutions. As in \cite{franco} these solutions are weak in the analytical sense and weak in the probabilistic sense. Moreover, the time-evolution of their energy can be controlled in terms of the initial energy. This allows to study some asymptotic properties of the system, see \cite{BFHa}, \cite{BFH} and \cite{romeo}.\\
All these results are concerned with the barotropic case.
The natural next step is to take additionally into account the transfer of heat.
The aim of this paper is to develop a consistent mathematical theory of viscous, compressible, and  heat-conducting
fluid flows driven
by stochastic external forces.
\\
The motion of the fluid is described by the standard field variables: the mass density $\varrho=\varrho(t,x)$; the absolute temperature $\vartheta(t,x)$; the velocity field $\bfu=\bfu(t,x)$
evaluated at the time $t$ and the spatial position $x$ belonging to the reference physical domain $Q \subset\mathbb R^3$. The time evolution of the fluid
is governed by the full Navier--Stokes--Fourier system with stochastic forcing:
\begin{subequations}\label{eq:1}
\begin{align}
\dd\varrho+\Div (\varrho\bfu)\dt&=0 ,\label{eq:12}\\
\dd(\varrho\bfu)+ \left[ \Div(\rho\bfu\otimes\bfu)  + \Grad p(\vr, \vt) \right] \dt &= \Div \bfS (\vt, \Grad \vu) \dt+\varrho{\vc{F}}(\varrho,\vt,\bfu)\,\dd W,\,\,\label{eq:11}\\
\dd(\varrho e(\vr, \vt))&+\big[\Div(\varrho e(\vr, \vt) \bfu)+\Div\bfq (\vt, \Grad \vt) \big]\dt\label{eq:13}\\&=\big[\bfS(\vt, \Grad \vu) :\nabla\bfu-p(\vr, \vt) \Div\bfu\big]\dt + \vr H \dt.
 \nonumber
\end{align}
\end{subequations}
\\
The field equations (\ref{eq:1})
describe the balance of mass, momentum and internal energy.
They must be supplemented with a set of
constitutive relations characterizing the material properties of a concrete fluid.
In particular, the viscous stress tensor $\bfS$, the internal energy flux $\bfq$ as well as
the thermodynamic functions $p$ (pressure) and $e$ (specific internal energy) must be determined in terms of the
independent state variables $(\varrho,\vartheta,\bfu)$. For the viscous stress tensor we suppose
Newton's rheological law
\begin{align}\label{eq:nr}
\bfS=\bfS(\vartheta,\nabla \bfu)=\mu(\vartheta)\Big(\nabla\bfu+\nabla\bfu^T-\frac{2}{3}\Div\bfu\,\mathbb I\Big)+\eta(\vartheta)\Div\bfu\,\mathbb I.
\end{align}
The internal energy (heat) flux is determined by Fourier's law
\begin{align}\label{eq:fl}
\bfq=\bfq(\vartheta,\nabla\vartheta)=-\kappa(\vartheta)\nabla\vartheta=-\nabla\mathcal K(\vartheta),\quad \mathcal K(\vartheta)=\int_0^\vartheta\kappa(z)\,\dd z.
\end{align}
The thermodynamic functions $p$ and $e$ are related to the (specific) entropy $s = s(\vr, \vt)$ through {Gibbs' equation}
\bFormula{m97} \vartheta D s (\varrho, \vartheta) = D e (\varrho,
\vartheta) + p (\varrho, \vartheta) D \Big( \frac{1}{\varrho} \Big)
\ \mbox{for all} \ \varrho, \vartheta > 0.
\eF
\\
Randomness in the time evolution of the system is enforced in {three} ways: {\bf(i)} the initial state (data) is random;
{\bf (ii)} the heat source $H$ is a random variable that may depend on time, {\bf(iii)}
the driving process $W$ is a cylindrical Wiener process defined on some probability space, with the diffusion coefficient $\vr \vc{F}$ that may depend on the spatial variable $x$ as well as on the state variables $\varrho$, $\vartheta$ and $\bfu$.
The precise description of the problem setting will be given in  Section \ref{sec:framework}.\\
Our main result, stated in Theorem \ref{thm:main} below, asserts the existence of a \emph{martingale solution} to
a suitable weak formulation \eqref{eq:1}--\eqref{m97} with respect to conservative boundary conditions
\begin{equation} \label{beq:1}
\vu|_{\partial Q} = 0,\ \Grad \vt \cdot \vc{n}|_{\partial Q} = 0.
\end{equation}
We combine the deterministic approach for the Navier--Stokes--Fourier system developed in \cite{F} with the stochastic theory \cite{BrHo} for the barotropic system.
\\
In contrast with the earlier {approach} proposed in \cite{fei3}, the existence theory in \cite{F} is built up around the Second law of thermodynamics. In view of Gibb's relation (\ref{m97}),
the internal energy equation \eqref{eq:13} can be rewritten in the form of the entropy balance
\begin{align}\label{eq:13'}
\dd(\varrho s)+\Big[\Div(\varrho s\bfu)+\Div\Big(\frac{\bfq}{\vartheta}\Big)\Big]\dt=\sigma\dt + \vr \frac{{H}}{\vt} \dt
\end{align}
with the entropy production rate
\begin{align}\label{eq:0202}
\sigma=\frac{1}{\vartheta}\Big(\bfS:\nabla\bfu-\frac{\bfq\cdot\nabla\vartheta}{\vartheta}\Big).
\end{align}
In view of possible, but in {the case of viscous fluids} still only hypothetical singularities, it is convenient
to relax the equality sign in \eqref{eq:0202} to the inequality
\begin{align}\label{eq:0202bis}
\sigma \geq \frac{1}{\vartheta}\Big(\bfS:\nabla\bfu-\frac{\bfq\cdot\nabla\vartheta}{\vartheta}\Big).
\end{align}
Well posedness of the problem is then formally guaranteed by augmenting the system by the total energy balance
\begin{equation} \label{eq0202bisE}
\dd \int_{Q} \left[ \frac{1}{2} \vr |\vu|^2 + \vr e \right] \dx  = \int_{Q} \vr \vc{F} \cdot \vu \ \dd W + \int_{Q} \frac{1}{2} \vr |\vc{F}|^2 \dt
+ \int_{Q} \vr H \dt,
\end{equation}
cf. \cite[Chapter 2]{F}.
Note that validity of (\ref{eq0202bisE}) requires the system to be energetically insulated - the total energy flux through the boundary must vanish at any time
in accordance with (\ref{beq:1}).
{In comparison with \cite{F}, equality (\ref{eq0202bisE}) contains the contribution of the stochastic driving force here interpreted in It\^{o}'s sense.}
\\
The main advantage of the entropy formulation is the possibility to deduce a relative energy inequality derived for the deterministic
Navier--Stokes--Fourier system in \cite{FeiNov10},
and adapted to the barotropic \emph{stochastic} Navier--Stokes system in \cite{BFH}. In particular, we may expect the strong solutions to be
stable in the larger class of weak solutions
(weak--strong uniqueness). Besides, there are other interesting properties derived for simpler systems in \cite{BFH} and \cite{FeiNov10} that are likely to extend to the full thermodynamic framework. Note that a weak formulation based on the internal energy balance in the spirit of \cite{fei3} apparently does not enjoy these properties unless
the weak solution is quite ``regular'', see \cite{FeSu2015_N1}. We remark that a stochastic version of \cite{fei3} recently appeared in \cite{SmTr}.
\\
The total energy balance (\ref{eq0202bisE}) is considered as an integral part of the definition of weak solutions. This excludes
any kind of semi-deterministic approach in the spirit of \cite{BeTe} or \cite{feireisl2}. Instead, we adapt
the multi--layer approximation scheme developed in \cite[Chapter 3]{F} combined with a refined stochastic compactness method based on the Jakubowski-Skorokhod representation theorem, cf. \cite{jakubow}. Although a similar idea has been applied to the barotropic Navier--Stokes system in \cite{BrHo}, the explicit
dependence of the diffusion coefficient $\vc{F}$ on the temperature along with the total energy balance appended to the problem give rise to rather
challenging new difficulties pertinent to the complete, meaning energetically closed, fluid system. {One of the most
subtle among them is the necessity to perform the change of probability space via Skorokhod--Jakubowski representation theorem
\emph{before} showing compactness of the arguments in the diffusion coefficients $\vc{F}$. This is in sharp contrast with the method developed in \cite{BrHo},
where the new probability space emerged in a natural way only at the end of each approximate step of the construction of weak solutions.}
\\
Following \cite[Chapter 3]{F}
we { start with the original (internal energy) formulation
\eqref{eq:1} regularized via artificial viscosity ($\varepsilon$-layer), artificial pressure ($\delta$-layer),
as well as other stabilizing terms,
see Section \ref {sec:galerkin}.} The so-obtained system is then solved by a Galerkin approximation. Here, the momentum equation is solved in a finite-dimensional subspace
whereas we keep continuity and internal energy equation in a continuous framework ($m$-layer). This has the advantage that the maximum principle applies and both density and temperature remain positive. We follow the approach from \cite{BFHbook}:
Instead of introducing a stopping-time as in \cite{BrHo} we cut several nonlinearities ($R$-layer) and gain the existence of a strong solution to the so-obtained system \eqref{wW}, see {Theorem \ref{wWP1}}. Also differently from \cite{BrHo} we do not apply a fixed-point argument to get a solution to \eqref{wW}. Instead we apply a simple time-step by means of a modification of the Cauchy collocation method, see \eqref{wWS14}--\eqref{wWS16}. The latter one can be solved immediately. Eventually, we pass to the limit in the time-step in Section \ref{wWS1S2}. {At this stage, it is important to keep the temperature strictly positive in order to
divide finally equation (\ref{eq:13}) obtaining the entropy formulation, cf. Section \ref{EGA}}.\\
{The remaining part of the existence proof leans on the entropy formulation, with (\ref{eq:13}) replaced by (\ref{eq:13'}) and (\ref{eq0202bisE}). In Section
\ref{EGA} we perform the limit in the Galerkin approximation scheme (limit $m \to \infty$) obtaining the artificial viscosity approximation. In Section \ref{sec:vanishingviscosity},
we get rid of the artificial viscosity and related stabilizing terms (limit $\ep \to 0$). Finally, in Section \ref{sec:vanishingpressure},
we remove the remaining artificial terms recovering a weak solution of the original system (limit $\delta \to 0$).}

\section{Mathematical framework and the main result}
\label{sec:framework}

Due to the lack of regularity of the unknown fields, in particular with respect to the time variable, we follow the ``weak'' approach developed in \cite{BFHbook}.

\subsection{Random variables (distributions)}
\label{sec:prelimsstoch}

Let $Q_T = (0,T) \times Q$. Let $(\Omega,\mf,(\mf_t)_{t\geq0},\prst)$ be a complete stochastic basis with a Borel probability measure $\prst$ and a right-continuous filtration
$(\mf_t)$.
The majority of the random fields we deal with are vector valued functions $\vc{U} \in L^1(Q_T; R^M)$ depending on the random parameter $\omega \in \Omega$. We say that
$\vc{U}$ is a random variable, if all functions $\omega \mapsto \int_{Q_T} \vc{U} \cdot \bfphi \dxdt$ are $\prst-$measurable for any $\bfphi \in \DC(Q_T; R^M)$.
We also introduce a \emph{natural filtration or history} of $\vc{U}$ up to a time $\tau$,
\[
\sigma_\tau[ \vc{U} ] = \sigma \left\{  \left\{ \int_{Q_T} \vc{U} \cdot \bfphi \dxdt < a \right\};\ a \in R,\ \bfphi \in \DC(Q_\tau; R^M) \right\}.
\]
\\
We say that $\vc{U}$ is \emph{progressively $(\mf_t)$-measurable} if $\sigma_\tau [\vc{U}] \subset \mf_\tau$ for any $\tau \geq 0$.
\\
It is convenient to consider a random variable $\vc{U} \in L^1(Q_T; R^M)$ as a distribution ranging in a larger space $\vc{U} \in W^{-k,2}(Q_T; R^M)$, where the latter
is a separable Hilbert space; whence Polish. Note that, by virtue of the standard embedding
$W^{k,2}_0 \hookrightarrow\hookrightarrow C$ as soon as $k > \frac{5}{2}$, we have $L^1(Q_T) \hookrightarrow\hookrightarrow W^{-k,2}(Q_T)$. Accordingly, any such
$\vc{U}$ may be viewed as a Borel random variable on the Polish space $W^{-k,2}(Q_T; R^M)$.

\subsubsection{Initial data}

In the context of martingale solutions, the initial state of the system is prescribed in terms of a law $\Lambda$ - a Borel measure defined on a suitable
function space. 
%In view of the hypothetical possibility of vacuum regions, it is convenient 
We consider $(\vr_0, \vt_0, \vu_0)$ to describe the initial state.
Accordingly, we consider $\Lambda$ defined on $L^1(Q) \times L^1(Q) \times L^1(Q; R^3)$ satisfying
\begin{equation} \label{Ida}
\begin{split}
\Lambda &\left\{ \vr_0 \geq 0,\ \vt_0 > 0, \  0 < \underline{\vr} < (\vr_0)_Q \equiv \frac{1}{|Q|} \intQ{ \vr_0 } < \Ov{\vr} \right\} = 1,\\
\int_{L^1 \times L^1 \times L^1} &\left\| \vr_0 |\vu_0|^2 + \vr_0 e(\vr_0, \vt_0) + \vr_0 |s(\vr_0, \vt_0)| \right\|_{L^1(Q)}^r \ {\rm d}\Lambda
\leq c(r),\ \mbox{for all}\ r \geq 1.
\end{split}
\end{equation}
Here $\underline{\vr}$, $\Ov{\vr}$ are two deterministic constants.

\subsubsection{Mechanical bulk force}

The process $W$ is a cylindrical Wiener process, that is, $W(t)=\sum_{k\geq1}\beta_k(t) e_k$ with $(\beta_k)_{k\geq1}$ being mutually independent real-valued standard Wiener processes relative to $(\mf_t)_{t\geq0}$. Here $(e_k)_{k\geq1}$ denotes a complete orthonormal system in a sepa\-rable Hilbert space $\mathfrak{U}$. In addition, we introduce an auxiliary space $\mathfrak{U}_0\supset\mathfrak{U}$ via
$$\mathfrak{U}_0=\bigg\{v=\sum_{k\geq1}\alpha_k e_k;\;\sum_{k\geq1}\frac{\alpha_k^2}{k^2}<\infty\bigg\},$$
endowed with the norm
$$\|v\|^2_{\mathfrak{U}_0}=\sum_{k\geq1}\frac{\alpha_k^2}{k^2},\qquad v=\sum_{k\geq1}\alpha_k e_k.$$
Note that the embedding $\mathfrak{U}\hookrightarrow\mathfrak{U}_0$ is Hilbert-Schmidt. Moreover, trajectories of $W$ are $\prst$-a.s. in $C([0,T];\mathfrak{U}_0)$ (see \cite{daprato}).\\
Choosing $\mathfrak{U}= \ell^2$ we may identify the diffusion coefficients
$(\vc{F} e_k )_{k \geq 1}$ with a sequence of real functions
$(\vc{F}_k)_{k \geq 1}$,
\[
\vr \vc{F}(\vr, \vt, \vu) \dd W = \sum_{k=1}^\infty \vr \vc{F}_k(x, \vr, \vt, \vu) \dd \beta_k.
\]
\\
We suppose that $\vc{F}_k$ are smooth in their arguments, specifically,
\[
\vc{F}_k \in C^1( \Ov{Q} \times [0, \infty)^2 \times R^3; R^3),
\]
where
\begin{equation} \label{P-1}
\| {\vc{F}_k}(\cdot, \cdot, \cdot, 0) \|_{L^\infty} + \| {\nabla_{x,\vr,s, \vc{u}}} \vc{F}_k \|_{L^\infty} \leq f_k,\ \sum_{k=1}^\infty f_k^2 < \infty.
\end{equation}
Let us remark that \eqref{P-1} implies that $\bfF$ is bounded in $\varrho$ and $\vartheta$ but may grow linearly in $\bfu$.
We easily deduce from \eqref{P-1} the following bound
\[
\| \vr \vc{F}_k (\vr, \vt, \vu) \|_{W^{-k,2}(Q; R^3)} \aleq \| \vr \vc{F}_k (\vr, \vt, \vu)  \|_{L^1 (Q; R^3)} \aleq f_k
\left( \| \vr \|_{L^1(Q)} +\| \vr \vu \|_{L^1(Q; R^3)} \right)
\]
whenever $k > \frac{3}{2}$.
Accordingly, the stochastic integral
\[
\int_0^\tau \vr \vc{F} \dd W = \sum_{k = 1}^\infty \int_0^\tau \vr \vc{F}_k(\vr, \vt, \vu) \ \dd \beta_k
\]
can be identified with an element of the {Banach space space $C([0,T]; W^{-k,2}(Q))$},
\[
\begin{split}
&\intQ{ \left( \int_0^\tau \vr \vc{F}(\vr, \vt, \vu) \dd W \cdot \bfvarphi \right) } \\ &= \sum_{k=1}^\infty \int_0^\tau \left(
\intQ{ \vr \vc{F}_k(x,\vr, \vt, \vu) \cdot \varphi } \right) \dd \beta_k, \ \bfvarphi \in W^{k,2}(Q; R^3),\ k > \frac{3}{2}.
\end{split}
\]

\subsubsection{Heat source}

Similarly to \cite[Chapter 3, Section 3.1]{F}, the
heat source $H$ may depend on both $t$ and $x$, specifically, $H \in C([0,T] \times \Ov{Q})$. In addition, we suppose
\begin{equation} \label{ID3}
0 \leq H \leq \Ov{H},\ H \ \ (\mf_t) -\mbox{progressively measurable},
\end{equation}
where $\Ov{H}$ is a deterministic constant.
\\
A heat source appears in numerous real--world applications (see e.g. \cite{F}); therefore we find it important to include it in the existence theory.
From the mathematical point of view, however, its presence in the system can be accommodated rather easily at the same level of difficulty as the random initial data.

\subsection{Structural and constitutive assumptions}
\label{H}

Besides Gibbs' equation (\ref{m97}), we impose several restrictions on the specific shape of the thermodynamic functions
$p=p(\vr,\vt)$, $e=e(\vr,\vt)$ and $s=s(\vr,\vt)$. They are borrowed from \cite[Chapter 1]{F}, to which we refer
for the physical background and the relevant discussion.\\
We consider the pressure $p$ in the form
\bFormula{m98}
p(\vr, \vt) = p_M(\vr, \vt) + \frac{a}{3} \vt^4, \ a > 0,\
p_M(\vr, \vt) = \vt^{5/2} P\left( \frac{\vr}{\vt^{3/2}} \right),
\eF
\begin{equation} \label{mp8a}
e(\vr, \vt) = e_M(\vr, \vt) + {a} \frac{\vt^4}{\vr},\ e_M(\vr, \vt) = \frac{3}{2} \frac{p_M(\vr, \vt)}{\vr} =
\frac{3}{2} \frac{\vt^{5/2}}{\vr} P\left( \frac{\vr}{\vt^{3/2}} \right),
\end{equation}
\begin{equation}\label{md8!}
s(\vr, \vt) = s_M(\vr, \vt) + \frac{4a }{3} \frac{\vt^3}{\vr},\
s_M(\vr, \vt) = S \left( \frac{\vr} {\vt^{3/2}} \right),
\end{equation}
\begin{equation}\label{md8}
S = S(Z),\
S'(Z)=-\frac 32 \frac{\frac 53 P(Z)-ZP'(Z)}{Z^2}<0,
\end{equation}
where
\begin{equation}\label{m103-}
P \in C^1[0, \infty) \cap C^2(0, \infty),\; P(0)= 0,\;P'(Z)>0, \ \mbox{for all}\ {Z \geq 0},
\end{equation}
\begin{equation}\label{md7}
0<\frac 32 \frac{\frac 53 P(Z)-ZP'(Z)}{Z}  <c,\; \mbox{for all}\ Z > 0,
\end{equation}
and
\begin{equation}\label{md1}
\lim_{Z\to\infty}\frac {P(Z)} {Z^{5/3}}=p_\infty>0.
\end{equation}
\\
The viscosity coefficients $\mu$, $\eta$ are continuously
differentiable functions of the absolute temperature $\vt$, more
precisely $\mu, \ \lambda  \in C^1[0,\infty)$, satisfying
\bFormula{m105} 0 < \underline{\mu}(1 + \vartheta) \leq
\mu(\vartheta) \leq \Ov{\mu}(1 + \vartheta), \eF
\bFormula{*m105*} \sup_{\vartheta\in [0, \infty)}\big(|\mu'(\vartheta)|+|\lambda'(\vartheta)|\big)\le
\overline m, \eF
\bFormula{m106} 0 \leq \lambda (\vartheta) \leq \Ov{\lambda}(1 +
\vartheta). \eF
\\
The heat conductivity coefficient $\kappa \in C^1[0, \infty)$ satisfies
\bFormula{m108} 0 < \underline{\kappa} (1 + \vartheta^3) \leq \kappa(
\vartheta) \leq \Ov{\kappa} (1 + \vartheta^3).
\eF

\subsection{Martingale solutions}
\label{subsec:solution}

The solutions to \eqref{eq:1} will be weak in both probabilistic and PDE sense. From the point of view of the theory of PDEs, we follow the approach of \cite{F} and consider so-called finite energy weak solutions satisfying
the momentum balance (\ref{eq:11}), the equation of continuity (\ref{eq:12}), together with the entropy balance (\ref{eq:13'}), (\ref{eq:0202bis}), and
the total energy balance (\ref{eq0202bisE}). Solutions will satisfy these relations in the weak PDE sense, meaning all derivatives are understood in the sense of distributions.\\
From the probabilistic point of view, two concepts of solution are typically considered in the theory of stochastic evolution equations, namely, pathwise (or strong) solutions and martingale (or weak) solutions. In the former notion the underlying probability space as well as the driving process is fixed in advance while
in the latter case these stochastic elements become part of the solution of the problem. Clearly, existence of a pathwise solution is stronger and implies existence of a martingale solution. In the present work we are only able to establish existence of a martingale solution to \eqref{eq:1}. Due to the classical Yamada-Watanabe-type argument (see, e.g., \cite{krylov} and \cite{pr07}), existence of a pathwise solution would then follow if pathwise uniqueness held true. However, uniqueness for the Navier--Stokes--Fourier equations for compressible fluids is an open problem even in the deterministic setting. In hand with this issue goes the
the fact that the initial state is determined only by the measure $\Lambda$ introduced in (\ref{Ida}).

Let us summarize the above in the following definition.

\begin{definition}[Martingale solution]\label{def:sol}
Let $Q \subset R^3$ be a bounded domain of class $C^{2 + \nu}$, $\nu > 0$.
Let $\Lambda$ be a Borel probability measure on $L^{1}(Q)\times L^1(Q)\times L^1(Q; R^3)$ enjoying the properties
(\ref{Ida}).
Let $H \in C([0,T] \times \Ov{Q})$ be given satisfying (\ref{ID3}).
\\
Then
$$\big((\Omega,\mf,(\mf_t),\prst),\varrho,\vartheta,\bfu,W)$$
is called \emph{(weak) martingale solution} to problem \eqref{eq:1}, (\ref{beq:1}) with the initial data $\Lambda$ provided the following holds.
\begin{enumerate}
\item $(\Omega,\mf,(\mf_t),\prst)$ is a stochastic basis with a complete right-continuous filtration;
\item $W$ is an $(\mf_t)$-cylindrical Wiener process;
\item the random variables
\[
\vr \in L^1(Q_T),\ \vt \in L^1(Q_T),\ \vu \in L^2(0,T; W^{1,2}_0(Q; R^3))
\]
are progressively $(\mf_t)$-measurable, $\vr \geq 0$, $\vt > 0$ $\pas$;
\item the density $\vr \in C_w(0,T; L^{5/3}(Q))$,\\ the temperature $\vartheta \in L^\infty(0,T; L^4(Q))
\cap {L}^2([0,T]; W^{1,2}(Q))$ $\pas$,\\
the velocity $\vu \in {L}^2(0,T; W^{1,2}_0(Q;R^3))$ $\pas$, \\
and the momentum $\vr \vu \in C_w(0,T; L^{5/4}(Q; R^3))$ $\pas$ satisfy
\[
\E\bigg[\sup_{t \in [0,T]} \| \vr(t,\cdot) \|^{5/3}_{L^{5/3}(Q)} \bigg] < \infty,
\]
\[
\E\bigg[ {\rm ess}\sup_{t \in [0,T]} \| \vt(t,\cdot) \|^4_{L^{4}(Q)} \bigg] < \infty,
\
\E\bigg[ \int_0^T \| \vartheta \|^2_{W^{1,2}(\tor)} \ \dt \bigg] < \infty,
\]
\[
\E\bigg[\int_0^T \| \vu \|^2_{W^{1,2}_0(Q;R^3)} \ \dt  \bigg] < \infty ,
\
\expe{ \sup_{t \in [0,T]} \| \vr \vu (t,\cdot) \|^{5/4}_{L^{5/4}(Q;R^3)} };
\]
\item the equation of continuity
\begin{align}\label{wWS237final}
\int_0^T \intQ{ \left[ \vr \partial_t \psi + \vr \vu \cdot \Grad \psi \right] }\dt = - \intQ{ \vr_0 \psi(0) };
\end{align}
holds for all $\psi\in \DC([0,T) \times R^3)$ $\prst$-a.s.;
\item the momentum equation
\begin{align}\label{wWS238final}
\begin{aligned}
&\int_0^T \partial_t \psi \intQ{ \vr \vu \cdot \bfvarphi } \dt\\
&+\int_0^T \psi \intQ{ \varrho\bfu\otimes\bfu : \nabla \bfvarphi } \dt -\int_0^T \psi \intQ{ \mathbb S(\vartheta,\nabla\bfu) : \nabla\bfvarphi } \dt\\
&+\int_0^T \psi \intQ{ p(\varrho,\vartheta) \diver\bfvarphi }\dt + \int_0^t \psi \intQ{ \varrho{\vc{F}}(\varrho, \vt , \bfu) \cdot
\bfvarphi }\,\dif W \\&= - \intQ{ \vr_0 \vc{u}_0 \cdot \bfvarphi };
\end{aligned}
\end{align}
holds for all $\psi \in \DC[0,T)$,
 $\bfvarphi\in \DC (Q;R^3)$ $\prst$-a.s.
\item the entropy balance
\begin{align} \label{m217*final}\begin{aligned}
-\intO{ \varrho_0 s(\varrho_0,\vartheta_0)\psi }
 &\geq \int_0^T \intQ{ \left[ s(\varrho,\vartheta) \partial_t \psi + \varrho s (\varrho,\vartheta)\bfu \cdot \nabla\psi \right]}\dif t
\\& +\int_0^T\int_{Q}
\frac{1}{\vartheta}\Big[\mathbb S(\vartheta,\nabla\bfu):\nabla\bfu+\frac{\kappa(\vartheta)}{\vartheta}|\nabla\vartheta|^2\Big]
\psi\,\dif x\,\dif t \\
&+ \int_0^T\intQ{ \frac{\kappa(\vartheta)\nabla\vartheta}{\vartheta} \cdot \nabla\psi }\,\dif t + \int_0^T \intQ{ \frac{\vr}{\vt} \mathcal{Q} \psi } \dt
\end{aligned}
\end{align}
holds for all
$\psi\in \DC([0,T)\times R^3)$, $\psi \geq 0$ $\prst$-a.s.;
\item the total energy balance
\begin{equation} \label{EI20final}
\begin{split}
- \int_0^T \partial_t \psi &
\left( \intQ{ \mathcal E(\varrho,\vartheta, \vu) } \right) \ \dt =
\psi(0) \intQ{ \mathcal E \left(\varrho_0,\vartheta_0, \vu_0 \right) } + \int_0^T \psi \intQ{ \vr H } \dt \\&+\int_0^T  \psi \int_{Q}\varrho{\bf F}(\varrho,\vartheta,\bfu)\cdot\bfu\,\dd W\dx + \frac{1}{2} \int_0^T
\psi \bigg(\intQ{ \sum_{k \geq 1}  \varrho| {\bf F}_k (\varrho,\vartheta,  {\bf u}) |^2 } \bigg) {\rm d}t
\end{split}
\end{equation}
holds for any $\psi \in \DC([0, T)$ $\mathbb{P}$-a.s.
Here, we abbreviated
$$\mathcal E(\varrho,\vartheta,\vu)= \frac{1}{2} \varrho | {\bf u} |^2 + \varrho e(\varrho,\vartheta).$$
\item The initial data $(\vr_0, \vt_0, \vc{u}_0)$ are $\mf_0$-measurable,
$\Lambda=\prst\circ \big(\varrho_0,\vt_0,\vc{u}_0 \big)^{-1}$.
\end{enumerate}
\end{definition}

\begin{Remark} \label{NRR1}

The random variables $\vr$, $\vr \vu$ are progressively $(\mf_t)$-measurable and weakly continuous, in particular, the real valued processes
\[
t \mapsto \intQ{ \vr \psi },\ t \mapsto \intQ{ \vr \vu \cdot \bfphi }
\ \mbox{are}\ (\mf_t)-\mbox{adapted for any smooth}\ \psi, \ \bfphi \ \pas
\]
The entropy $\vr s(\vr,\vt)$ is ``weakly c\` adl\` ag'', meaning
\[
t \mapsto \intQ{ \vr s(\vr, \vt) \psi } \ \mbox{is c\` adl\` ag for any smooth} \ \psi \ \pas,
\]
cf. the discussion in \cite[Chapter 1]{F}.
The random variables $(\vr, \vt, \vu)$ being progressively $(\mf_t)$-measurable, all stochastic integrals are well defined.

\end{Remark}

\begin{Remark} \label{NRRR1}
The energy equality \eqref{EI20final}, together with the entropy inequality \eqref{m217*final}, forms a natural counterpart of the energy inequality obtained in \cite{BFH} for the barotropic case. As in the deterministic case, cf. \cite{F}, equality holds in \eqref{EI20final}.
\end{Remark}

\begin{Remark} \label{NRrR1}

The specific values of the exponents appearing in the integrals in (d) are related to the constitutive hypotheses
(\ref{m98}--\ref{md1}).

\end{Remark}

Now we are ready to state the existence theorem which is the main result of this paper.

\begin{Theorem}\label{thm:main}
Let $Q \subset R^3$ be a bounded domain of class $C^{2 + \nu}$, $\nu > 0$.
Suppose that the structural assumptions \eqref{P-1}--\eqref{m108} are satisfied.

Then there exists a martingale solution to problem \eqref{eq:1}, \eqref{beq:1} in the sense of Definition \ref{def:sol} with the initial law $\Lambda$.

\end{Theorem}

The rest of the paper is devoted to the proof of Theorem \ref{thm:main}. Concluding remarks are presented in Section \ref{conc}.

\section{{Basic approximate problem}}
\label{sec:galerkin}

{The solutions will be constructed by means of a multilevel approximation scheme. In order to simplify presentation, we carry over the proof in the absence of
the heat source $H$. The necessary modifications to accommodate this quantity in the proof are left to the reader and can be found
in \cite[Chapter 3]{F}.}\\
The standard approach to non-linear partial differential equations
starts with a finite dimensional approximation of Galerkin type. Unfortunately, this can be applied only to the momentum equation (\ref{eq:11}) since we need
the density $\varrho$ and temperature $\vt$ to be positive at the first level of approximation. Positivity of both results from a maximum
principle, where the latter is usually incompatible with a Galerkin type approximation.
It seems therefore more convenient to apply the artificial viscosity method adding diffusive terms to both (\ref{eq:11}) and (\ref{eq:12}).
In order to prove Theorem \ref{thm:main} we will regularize several quantities in the system \eqref{eq:1}. Adopting the approximation scheme
developed in \cite{F} we add an artificial viscosity to the continuity and momentum equations and regularize the pressure. For technical reasons,
however, we have to regularize several further quantities as well.\\
Following \cite{F}, we introduce
\begin{equation} \label{neu}
\begin{split}
p_\delta(\vr,\vt)&=p (\vr,\vt) + \delta({\vr^2} + \vr^\beta), \\
e_{M, \delta}(\vr, \vt) &= e_M(\vr, \vt) + \delta \vt,\
e_\delta(\vr, \vt) = e(\vr, \vt) + \delta \vt, \\
s_{M,\delta}(\vartheta,\varrho)  &=s_M(\vartheta,\varrho)+\delta\log\vartheta,\quad s_\delta(\vr, \vt) = s(\vr, \vt) + \delta \log(\vt),\\ \kappa_\delta(\vartheta )&=\kappa(\vartheta )+\delta\Big(\vt^\beta+\frac{1}{\vartheta}\Big),\ \cal{K}_\delta(\vt) = \int_0^\vt \kappa_\delta(z) \ {\rm d}z.
\end{split}
\end{equation}
\\
Let $\Delta_0$ be the Laplace operator defined on the domain $Q$, with the homogeneous Dirichlet boundary conditions. Let $\{ w_n \}_{n \geq 1}$ be
the orthonormal system of the associated eigenfunctions.
A suitable platform for the Galerkin method is the following family of finite--dimensional spaces,
\[
H_m = {\rm span} \left\{ w_n \ \Big| \ n \leq m \right\}^3,\ m = 1,2,\dots
\]
endowed with the Hilbert structure of the Lebesgue space $L^2(Q;R^3)$.
Let
\[
\Pi_m : L^2(Q;R^3) \to H_m
\]
be the associated $L^2-$orthogonal projection.
As $\partial Q$ is regular, we have $\mathcal{D}(\Delta_0) = W^{2,2}(Q; R^3) \cap W^{1,2}_0 (Q; R^3) \hookrightarrow\hookrightarrow C (\overline Q; R^3)$.
Accordingly,
\begin{equation} \label{wW2a}
\begin{split}
\left\| P_m [\vc{f}] \right\|_{L^\infty(Q; R^3)} &\aleq \left\| P_m [\vc{f}] \right\|_{W^{2,2}(Q; R^3)}
\aleq \left\| P_m [\vc{f}] \right\|_{\mathcal{D}(\Delta_0)} \aleq
\left\| \vc{f} \right\|_{W^{2,2}(Q; R^3)},
\end{split}
\end{equation}
where the associated embedding constants are \emph{independent of} $m$.
\\
Moreover, as $H_m$ is finite-dimensional, all norms are equivalent on $H_m$ for any fixed $m$ - a property
that will be frequently used at the first level of approximation.\\
Next we introduce a cut-off function
\[
\chi \in C^\infty(R), \ \chi(z) = \left\{ \begin{array}{l} 1 \ \mbox{for} \ z \leq 0, \\ \chi'(z) \leq 0 \ \mbox{for}\ 0 < z < 1, \\
\chi(z) = 0 \ \mbox{for}\ z \geq 1, \end{array} \right.
\]
together with the operators
\[
[\vc{v}]_R = \chi(\| \vc{v} \|_{H_m} - R ) \vc{v}, \ \mbox{defined for}\  \vc{v} \in H_m.
\]
\\
Finally, we regularize the diffusion coefficients replacing $\vc{F}$ by $\vc{F}_\ep$,
\[
\vc{F}_{\ep} = \left( \vc{F}_{k,\ep} \right)_{k \geq 1},\
\vc{F}_{k,\ep}(x,\vr, \vt, \vu) = \chi \left( \frac{\ep}{\vr} - 1 \right) \chi \left(|\vu| - \frac{1}{\ep} \right)
\vc{F}_k(x, \vr, \vt, \vu).
\]
Consequently, in view of (\ref{P-1}),
\begin{equation} \label{P-1new}
\| {\vc{F}_{k, \ep}} \|_{L^\infty} + \| {\nabla_{x,\vr,s, \vc{u}}} \vc{F}_k \|_{L^\infty} \leq f_{k,\ep},\ \sum_{k=1}^\infty f_{k,\ep} ^2 < \infty.
\end{equation}
In addition, we regularize $\vc{F}_{k,\ep}(\cdot, \vr, \vt, \vu)$ in a ``non--local'' way, specifically, we introduce $\vc{F}_{\ep, \xi}$,
\begin{equation} \label{P-1New}
\vc{F}_{k,\ep,\xi} = \omega_{\xi} * \left[ h_\xi \vc{F}_{k,\ep} ( \cdot, \vr, \vt, \vu) \right], \ h_\xi \in \DC(Q),
\end{equation}
where $(\omega_\xi (x) )_{\xi > 0}$ is a family of regularizing kernels.
\\
The basic \emph{approximate problem} reads:
\begin{subequations}\label{wW}
\begin{align} \label{wW3}
\D \vr &+ \Div (\vr [\vu]_{R} )\ \dt  = \ep \Del \vr \ \dt ,\\
\D \Pi_m [\vr \vu]
&+ \Pi_m [\Div (\vr [\vu]_R \otimes \vu) ] \dt
+  \Pi_m \Big[ \chi(\| \vu \|_{H_m} - R )\Grad \left(  p (\vr,\vt) + \delta ({\vr^2} + \vr^\beta) \right) \Big] \dt\ \label{wW4}\\
&= \Pi_m \Big[\ep \Del (\vr \vu) +  \chi(\| \vu \|_{H_m} - R ) \Div \mathbb{S} (\Grad \vu) + {\frac{1}{m}} \vu \Big] \dt \nonumber \\ &+
\Pi_m \left[ \vr \Pi_m \Big[ \vc{F}_{\ep, \xi} (\vr, \vt, \vc{u} )  \Big]  \right] \ \D \W,\nonumber
\end{align}
\begin{align}
\label{m119}
\dd(\varrho e_\delta(\varrho, \vartheta)) &+
\big[\Div (\varrho e_\delta(\varrho, \vartheta) [\vc u]_R) - \Div ( \kappa_\delta (\vt) \Grad \vt ) \big]\dt \\
&\nonumber = \Big[\chi(\| \vu \|_{H_m} - R )\tens{S}(\vartheta,\Grad\vc u): \Grad \vc u -
p(\varrho, \vartheta) \Div {[\vc u]_R} \Big] \dt \\ &+ \Big[ \ep
\delta (\beta \varrho^{\beta - 2} + 2) |\Grad \varrho |^2  +
\delta \frac{1}{{\vartheta}^2} - \ep \vartheta^5 + {\ep \vr |\Grad \vu|^2} \Big]\dt, \nonumber
\end{align}
\end{subequations}
to be solved in the space-time cylinder $[0, T] \times Q$, with
the Neumann boundary conditions
\begin{equation} \label{BCba}
\Grad \vr \cdot \vc{n}|_{\partial Q} = 0, \ \Grad \vt \cdot \vc{n}|_{\partial Q} = 0,
\end{equation}
and prescribed initial data
\begin{equation} \label{wW5}
\begin{split}
\vr(0, \cdot) &= \vr_0 \in C^{2 + \nu}(\Ov{Q}), \ \vr_0 > 0,
\vartheta(0, \cdot) = \vartheta_{0},\ \vartheta_{0} \in W^{1,2}(Q) \cap
C(\Ov{Q}), \ \vartheta_{0} > 0, \\
\vu (0, \cdot) &= \vu_{0} \in H_m.
\end{split}
\end{equation}
\\
In \eqref{wW3} and (\ref{wW4}) we recognize the artificial ``viscosity'' terms $\ep \Del \vr$, $\ep \Del (\vr \vu)$, the pressure
regularization $\delta (\vr^\beta + \vr^2)$ as well as the cut-off operators applied to the velocity in the convective terms and other quantities
to preserve the total energy balance. Note that {equations (\ref{wW3}) and (\ref{m119}) are deterministic},
meaning they can be solved pathwise,
while (\ref{wW4}) involves stochastic integration. It is worth noting that (\ref{m119}) expresses the balance of the internal energy while the target problem
is formulated in terms of the entropy. In the following we give a precise definition of solutions to the approximate problem.

\begin{definition} \label{WD1}
Let $\Lambda$ be a Borel probability measure on $C^{2+\nu}(\Ov{Q})\times W^{1,2} \cap C(\Ov{Q})\times H_m$.
\\
Then
$$\big((\Omega,\mf,(\mf_t),\prst),\varrho,\vartheta,\bfu,W)$$
is called a \emph{martingale solution} to problem \eqref{wW3}--\eqref{m119}, \eqref{BCba}, with the initial data $\Lambda$,  provided the following holds.

\begin{enumerate}
\item $(\Omega,\mf,(\mf_t),\prst)$ is a stochastic basis with a complete right-continuous filtration;
\item $W$ is an $(\mf_t)$-cylindrical Wiener process;
\item
the functions $\vr$, $\vt$, $\vu$ belong to the class:
\[
\begin{split}
&\vr \in C([0,T]; C^{2 + \nu}(\Ov{Q})),\ \vr > 0 \ \pas, \\
&{\vt \in C([0,T]; W^{1,2}),\ 0 < \underline{\vt} < \vt < \Ov{\vt}, \ \partial_t \vt, \ K_\delta(\vt) \in L^2((0,T) \times Q) \ \pas,} \\
&\vu \in C([0,T], H_m) \ \pas,
\end{split}
\]
$\vr(t, \cdot)$, $\vt(t,\cdot)$, $\vu(t, \cdot)$ are $(\mathfrak{F}_t)$-adapted for any $t \geq 0$;
\item
the approximate equation of continuity
\begin{equation} \label{rho0}
\partial_t \vr + \Div (\vr [\vu]_R) = \ep \Del \vr,\ \Grad \vr \cdot \vc{n}|_{\partial Q} = 0 \ \mbox{holds in}\ (0,T) \times Q,  \ \pas;
\end{equation}
\item the approximate equation of internal energy
\begin{align} \label{rhos0}
\begin{aligned}
\partial_t(\varrho e_\delta(\varrho, \vartheta)) &+
\Div (\varrho e_\delta(\varrho, \vartheta) [\vc u]_R) - \Delta {\cal
K}_\delta (\vartheta)\\
&=\chi(\| \vu \|_{H_m} - R )\tens{S}(\vartheta,\Grad\vc u): \Grad \vc u -
{p(\varrho, \vartheta) \Div [\vc u]_R } \\&+ \ep
\delta (\beta \varrho^{\beta - 2} + 2) |\Grad \varrho |^2  +
\delta \frac{1}{{\vartheta}^2} - \ep \vartheta^5 + {\ep \vr |\Grad \vu|^2}, \ \Grad \vt \cdot \vc{n}|_{\partial Q} =0,
\end{aligned}
\end{align}
holds a.a. in $(0,T) \times Q$ $\pas$;
\item the approximate momentum equation
\begin{align}
\nonumber
\big\langle\varrho\bfu(t),\bfvarphi\big\rangle&=\big\langle(\varrho \bfu)_0,\bfvarphi\big\rangle+\int_0^t\big\langle\varrho[\bfu]_R\otimes\bfu,\nabla\bfvarphi\big\rangle\,\dif s-\int_0^t\big\langle \chi(\| \vu \|_{H_m} - R ) \mathbb S(\vartheta,\nabla\bfu),\nabla\bfvarphi\big\rangle\,\dif s\\
\label{rhou0}&-\varepsilon\int_0^t\big\langle\varrho\bfu,\Delta\bfvarphi\big\rangle\,\dif s+\int_0^t\chi( \| \vu \|_{H^m} - R)\big\langle  p_\delta(\vartheta,\varrho),\diver\bfvarphi\big\rangle\,\dif s\\
&- {\frac{1}{m}\int_0^t\big\langle\bfu,\bfvarphi\big\rangle\,\dif s} +\int_0^t\big\langle\varrho\Pi_m[\vc{ F}_{\ep, \xi} (\varrho, \vartheta,\bfu)]\,\dif W,\bfvarphi\big\rangle,
\nonumber
\end{align}
holds $\prst$-a.s. for all
 $\bfvarphi\in H_m$ and all $t\in[0,T]$;
\item we have
\begin{equation} \label{wWS13}
(\vr \vu)_0 = \vr(0) \vu(0)\,\,\pas,\
\prst [\vr(0, \cdot),\vt(0,\cdot), \vu(0, \cdot)]^{-1} = \Lambda[\vr_0,\vt_0, \vu_0].
\end{equation}

\end{enumerate}

\end{definition}

Our main goal in this section is to prove the existence of martingale solutions to the approximate problem \eqref{wW3}--\eqref{m119}, \eqref{BCba}.
\begin{theorem} \label{wWP1}
Let $\beta > 6$.
Let $\Lambda$ be a Borel probability measure  on $C^{2 + \nu}(\Ov{Q})\times W^{1,2} \cap C(\Ov{Q}) \times H_m$ such that
\begin{align*}
&\Lambda \left\{  0 < \underline{\vr} \leq \vr_0 ,\  \| \vr_0 \|_{C^{2 + \nu}_x} \leq \Ov{\vr},\
\Grad \vr \cdot \vc{n}|_{\partial Q} = 0 \right\} = 1,\
{\Lambda \left\{  0 < \underline{\vt} \leq \vt_0 \leq \overline{\vt} ,\  \| \vt_0 \|_{W^{1,2}_x} \leq \Ov{\vt} \right\}} = 1,\\
&\qquad\int_{C^{2 + \nu}\times W^{1,2}\cap C \times H_m } \| \vu_0 \|^r_{H_m}  {\rm d} \Lambda \leq \Ov{u}
\end{align*}
for some positive deterministic constants $\underline{\vr}$, $\Ov{\vr}$, $\underline{\vt}$, $\Ov{\vt}$ and some $r > 2$.
\\
Then
the approximate problem \eqref{wW3}--\eqref{m119}, \eqref{BCba} admits a martingale solution in the sense of
Definition \ref{WD1}. The solution satisfies
\begin{equation} \label{wWS110a}
\sup_{t \in [0,T]} \left( \| \vr(t) \|_{C^{2 + \nu}_x} + \| \partial_t \vr(t) \|_{C^{\nu}_x} +
\| \vr^{-1}(t) \|_{C^0_x} \right) \leq\, c\quad \prst\text{-a.s.},
\end{equation}
\begin{align*} { \sup_{t \in [0,T]} \| \vartheta^{-1}
\|_{L^\infty_x} + \sup_{t \in [0,T]} \| \vartheta
\|_{W^{1,2}_x \cap L^\infty_x} } + \int_0^T \Big( \|
\partial_t \vartheta \|^2_{L^2_x} + \| \Delta_x {\cal
K}_{\delta}(\vartheta) \|^2_{L^2_x} \Big) \dt\leq c \quad \prst\text{-a.s.}
\end{align*}
\begin{equation}
\label{wWS113a}
\begin{split}
\expe{ \sup_{t \in [0,T] } \| \vu (t, \cdot) \|_{H_m}^r } \leq\,c  \Big(1+\expe{ \| \vu_0 \|_{H^m}^r }\Big),
\end{split}
\end{equation}
where $c=\left(m,R, T, \underline{\vr}, \Ov{\vr}, \underline{\vt}, \Ov{\vt}, \Ov{u} \right)$.

\end{theorem}

In order to prove Theorem \ref{wWP1} we adopt the following strategy: The Galerkin projection applied in (\ref{wW4})
reduces the problem to a variant of ordinary stochastic differential equation, where the other unknown
quantities $\vr$, $\vt$ are uniquely determined by the deterministic equations (\ref{wW3}) and (\ref{m119}) in terms of
$\vu$ and the data. Accordingly,  problem \eqref{wW}, \eqref{BCba}) can be solved by means of a simple \emph{iteration scheme}. This is the objective of Section \ref{wWS1S1}. In addition, the approximate solutions satisfy the associated energy balance
equation yielding the uniform bounds necessary to carry out the asymptotic limits $m \to \infty$, $R \to 0$, and $\xi \to 0$.

\subsection{Iteration scheme}
\label{wWS1S1}

To begin, we fix the initial data $(\vr_0, \vt_0, \vu_0)$ satisfying \eqref{wWS13}. The existence of such data along with a suitable probability space follows from the standard
Skorokhod representation theorem.
Solutions to problem \eqref{wW}, \eqref{BCba} will be constructed by means of a modification of the Cauchy collocation method.
Fixing a time step $h > 0$ we set
\begin{equation} \label{wWS14}
\vr(t, \cdot) = \vr_0 , \ \vt (t, \cdot) = \vt_0,\ \vu(t,\cdot)=\vu_0, \ \mbox{for} \ t \leq 0,
\end{equation}
and define recursively, for  $t \in [nh, (n+1)h)$
\begin{align} \label{wWS15}
\partial_t \vr &+ \Div (\vr [\vu(nh, \cdot)]_R ) = \ep \Del \vr, \ \Grad \vr \cdot \vc{n}|_{\partial Q} = 0,\  \vr(nh, \cdot) = \vr(nh-, \cdot),\\
\nonumber
\partial_t(\varrho e_\delta(\varrho, \vartheta)) &+
\Div (\varrho e_\delta(\varrho, \vartheta)  [\vu(nh, \cdot)]_R) - \Delta {\cal
K}_\delta (\vartheta)\\
&= \chi(\| \vu(nh, \cdot)  \|_{H_m} - R )\tens{S}(\vartheta,\Grad\vc u(nh, \cdot)): \Grad \vc u(nh, \cdot) -
p(\varrho, \vartheta) \Div [\vc u(nh, \cdot)]_R\label{wWS15b}\\& + \ep
\delta (\beta \varrho^{\beta - 2} + 2) |\Grad \varrho |^2  +
\delta \frac{1}{{\vartheta}^2} - \ep \vartheta^5 + {\ep \vr |\Grad \vu(nh, \cdot)|^2} , \\ &\ \ \ \ \Grad \vt \cdot \vc{n}|_{\partial Q} = 0,\ \vt(nh, \cdot) = \vt(nh-, \cdot).\nonumber
\end{align}
Note that system (\ref{wWS15}), \eqref{wWS15b} is uncoupled as the former equation is independent of $\vt$.
Finally, for $\varrho$ and $\vartheta$ given through (\ref{wWS15}), \eqref{wWS15b}, we solve
\begin{align} 
\D \Pi_m (\vr \vu)
&+ \Pi_m \left[ \Div \Big(\vr [\vu(nh, \cdot) ]_R \otimes \vu(nh, \cdot) \Big) \right] \dt \nonumber\\
\label{wWS16}&+  \Pi_m \Big[ \chi(\| \vu(nh, \cdot)  \|_{H_m} - R )\Grad \left( p (\vr,\vt) + \delta ({\vr^2} + \vr^\beta) \right) \Big] \dt
\\
&= \Pi_m \Big[ \ep \Del (\vr \vu(nh, \cdot) ) +  \chi(\| \vu(nh, \cdot)  \|_{H_m} - R ) \Div \mathbb{S} (\Grad \vu(nh, \cdot) )
+ { \frac{1}{m} \vu (nh, \cdot) } \Big] \dt\nonumber\\
&+
\Pi_m \left[ \vr \Pi_m [\vc{F}_{\ep, \xi} (nh, \cdot) \right] \D \W, \ t \in [nh, (n+1)h),\ \vu(nh, \cdot) = \vu(nh-).\nonumber
\end{align}
To proceed it is convenient to rewrite (\ref{wWS16}) in terms of $\D \vu$. To this end, we write
\[
\D \Pi_m (\vr \vu) = \Pi_m (\D \vr \vu) + \Pi_m (\vr \D \vu) = \Pi_m (\partial_t \vr \vu) \dt + \Pi_m (\vr \D \vu).
\]
Next, we introduce
a linear mapping $\mathcal{M}[\vr]$,
\[
\mathcal{M}[\vr] : H_m \to H_m, \
\mathcal{M}[\vr](\vc{v}) = \Pi_m (\vr \vc{v}),
\]
or, equivalently,
\[
\intQ{ \mathcal{M}[\vr] \vc{v} \cdot \bfphi } \equiv \intQ{ \vr \vc{v} \cdot \bfphi } \ \mbox{for all}\
\bfphi \in H_m.
\]
The operator $\mathcal M$ has been introduced in \cite[Section 2.2]{feireisl1}, where one can also find the following properties.
We have that $\mathcal{M}[\rho]$ is invertible we have
\begin{align}\label{eqM-1}
&\Big\|\mathcal{M}^{-1}\big[\rho\big]\Big\|_{\mathcal L(X_n^\ast,X_n)}\leq\Big(\inf_{x\in\mt}\rho\Big)^{-1}
\end{align}
as long as $\rho$ is bounded below away from zero, and, clearly,
\[
\mathcal{M}^{-1}[\rho] (\Pi_m [\rho \vc{v}]) = \vc{v} \ \mbox{for any}\ \vc{v} \in H_m.
\]
 Let us finally mention that
\begin{align}\label{eqM-1b}
\Big\|\mathcal{M}^{-1}\big[\rho\big]-\mathcal{M}^{-1}\big[\rho\big]\Big\|_{\mathcal L(X_n^\ast,X_n)}\leq\,c(m,\underline\rho)\|\rho^1-\rho^2\|_{L^1_x},
\end{align}
provided both $\rho^1$ and $\rho^2$ are bounded from below by some positive constant $\underline\rho$.
Accordingly, relation (\ref{wWS16}) can be written in the form

\begin{align}\label{wWS17}
\begin{aligned}
\vu(t) - \vu(nh-)
&+ \mathcal{M}^{-1}[\vr(t)]\int_{nh}^t  \Pi_m \left[ \Div \Big(\vr [\vu(nh, \cdot) ]_R \otimes \vu(nh, \cdot) \Big) \right] \dt \\
&+  \mathcal{M}^{-1}[\vr(t)]\int_{nh}^t \Pi_m \Big[ \chi(\| \vu(nh, \cdot)  \|_{H_m} - R )\Grad \left( p(\vr,\vt) + \delta (\vr + \vr^\beta) \right) \Big] \dt
\\
&= \mathcal{M}^{-1}[\vr(t)]\int_{nh}^t  \Pi_m \Big[  \chi(\| \vu(nh, \cdot)  \|_{H_m} - R )\Div \mathbb{S}(\vt,\Grad \vu(nh, \cdot)   \Big] \dt\\
&+  \mathcal{M}^{-1}[\vr(t)]\int_{nh}^t \Pi_m \Big[ \ep \Del (\vr \vu(nh, \cdot) ) +   \frac{1}{m} \vu(nh, \cdot)  \Big] \dt \\
&+
\mathcal{M}^{-1}[\vr(t)]\int_{nh}^t \Pi_m \left[\vc{F}_{\ep, \xi} (nh, \cdot) \right]\dif W,\ nh < t < (n+1)h. \end{aligned}
\end{align}

The iteration scheme (\ref{wWS15})--(\ref{wWS16}) provides a unique solution for any initial data (\ref{wWS14}). Indeed,
as $\vu(nh) \in H_m$ is a smooth function, equation (\ref{wWS15}) admits a unique solution for any initial data $\vr(nh)$.
Moreover, as a direct consequence of the
parabolic maximum principle, $\vr$ remains positive as long as the initial datum $\vr(nh)$ is positive. We may therefore infer that
(\ref{wWS14}--\ref{wWS16}) give rise to uniquely determined functions $\vr$,$\vartheta$,$\vu$. In fact, we find a solution $\varrho$ such that
\[
\begin{split}
\vr \in C([0,T]; C^{2 + \nu} (\Ov{Q})), \ &\vr > 0,\ \pas
\end{split}
\]
by applying standard results (see, for instance, \cite[Theorem 5.1.21]{lu}) pathwise.
For equation \eqref{wWS15b} we obtain a solution $\vartheta$ belonging $\p$-a.s. to the class
\bFormula{m125b} Y = \left\{
\begin{array}{c} \partial_t \vt \in L^2((0,T) \times Q),\
\Delta_x{\cal K}_\delta (\vt) \in L^2((0,T) \times
Q), \\ \\
\vt \in L^\infty(0,T; W^{1,2}(Q) \cap L^\infty
(Q)),\quad \frac 1\vartheta\in L^\infty((0,T)\times Q),
\end{array}
\right\} \eF
by applying \cite[Lemma 3.4]{F} pathwise. Finally, knowing $\varrho$ and $\vartheta$ we can find the velocity
$$\bfu\in C([0,T];X_m)\  \pas$$
solving \eqref{wWS16} recursively.
Note that our construction implies that $\vr$, $\vt$ and $\vu$ are {$(\mathfrak{F}_t)$-adapted} and continuous in the
time variable $\prst$-a.s.

\subsection{The limit for vanishing time step}
\label{wWS1S2}

Our next goal is to let $h \to 0$ in (\ref{wWS14}--\ref{wWS16}) to obtain a solution of the approximate problem \eqref{wW}, \eqref{BCba}. This step
leans essentially on suitable uniform bounds independent of $h$.
To simplify notation, we shall write
\[
[v ]_h = v(nh, \cdot),\
[v ]_{h,R}(t, \cdot) = [v (nh, \cdot) ]_R \ \mbox{for}\ t \in [nh, (n+1)h), \ n=0,1,\dots.
\]
As all norms are equivalent on the finite-dimensional space $H_m$ and $\partial Q$ is smooth, we get
\[
\| [\vu ]_{h,R}  \|_{C^l(\Ov{Q}; R^3)} \leq c(l,m,R)  \ \mbox{uniformly for} \ h > 0, \mbox{and}\ t \in [0,T]
\ \mbox{at least for}\ l \leq 2.
\]
Consequently, the approximate equation of continuity (\ref{wWS15}) admits a unique regular solution, the smoothness of which is determined by the
initial data. In particular, the solution $\vr$ belongs to the class
\begin{equation} \label{wWS18}
\vr \in C([0,T]; C^{2 + \nu}(\Ov{Q})), \ \partial_t \vr \in L^\infty([0,T]; C^\nu(\Ov{Q}))
\end{equation}
as soon as
$\vr_0 \in C^{2 + \nu}(\Ov{Q})$, $\Grad \vr_0 \cdot \vc{n}|_{\partial Q} = 0$ for some $\nu > 0$.
In addition, the standard parabolic maximum principle yields
\begin{equation} \label{wWS19}
0 < \underline{r} (T,m, R) \min_{\Ov{Q}} \vr_0 \leq \vr(t, \cdot) \leq \Ov{r}(T,m,R) \max_{\Ov{Q}} \vr_0  \ \mbox{for all}\ t \in [0,T].
\end{equation}
Note that the regularized velocity $[\vu]_{h,R}$ is only piece-wise continuous; whence the same is true for $\partial_t \vr$ and therefore we do not expect
$\partial_t \vr \in C([0,T]; C^\nu(\Ov{Q}))$. Also
note carefully that, thanks to the hypotheses imposed on the initial law $\Lambda$, $\vr$ is bounded in the aforementioned spaces only in terms of the initial datum $\vr_0$, meaning, no probabilistic averaging
has been applied. In particular, we may infer that $\pas$
\begin{equation} \label{wWS110}
{\rm ess}\sup_{t \in [0,T]} \left( \| \vr(t, \cdot) \|_{C^{2 + \nu}(\Ov{Q})} + \| \partial_t \vr(t, \cdot) \|_{C^{\nu}(\Ov{Q})} +
\| \vr^{-1}(t, \cdot) \|_{C(\Ov{Q})} \right) \aleq c
\end{equation}
with $c=c\left(m,R, T, \underline{\vr}, \Ov{\vr} \right)$,
whenever
\begin{equation} \label{wWS111}
0 < \underline{\vr} \leq \vr_0, \ \| \vr_0 \|_{C^{2+\nu}(\Ov{Q})} \leq \Ov{\vr} \ \pas
\end{equation}
for certain \emph{deterministic} constants $\underline{\vr}$, $\Ov{\vr}$.\\
As far as the internal energy equation (\ref{wWS15b}) {we recall that}
\begin{equation} \label{wWS111b}
0 < \underline{\vt} \leq \vt_0 \leq \overline\vt, \ \| \vt_0 \|_{W^{1,2} (Q)} \leq \Ov{\vt} \ \pas,
\end{equation}
where $\underline{\vt},\overline{\vt}$ are deterministic constants.
From \cite[Lemma 3.3, Corollary 3.2]{F} we obtain the following estimate for the solutions to \eqref{wWS15b}
\[ \sup_{t \in (0,T)} \| \vartheta
\|^2_{W^{1,2}(Q)} + \int_0^T \Big( \|
\partial_t \vartheta \|^2_{L^2(Q)} + \| \Delta_x {\cal
K}_{\delta}(\vartheta) \|^2_{L^2(Q)} \Big) \dt \]
\[
\leq C \Big( {\| \varrho \|_{W^{1,\infty}((0,T) \times Q)}}, \| [\vc u]_{h,R}
\|_{C([0,T]; X_n)}, (\inf_{(0,T) \times Q} \varrho)^{-1}, \|
\vartheta_{0} \|_{W^{1,2}(Q)} \Big),
\]
and
\[
{0 < \underline{\theta} \leq \vt \leq \Ov{\theta},}
\]
where $\underline{\theta}$, $\Ov{\theta}$ depend only on $\underline{\vt}$, $\Ov{\vt}$, $\| \varrho \|_{W^{1,\infty}((0,T) \times Q)}, \| [\vc u]_{h,R}
\|_{C([0,T]; X_n)}$.
\\
Thus we may infer that
\bFormula{m135}
\begin{split}
{\rm ess} \sup_{t \in (0,T)} \| \vartheta^{-1}
\|_{L^\infty(Q)}    &+
{\rm ess} \sup_{t \in (0,T)} \| \vartheta
\|^2_{W^{1,2} \cap L^\infty (Q)}\\ + \int_0^T \Big( \|
\partial_t \vartheta \|^2_{L^2(Q)} &+ \| \Delta_x {\cal
K}_{\delta}(\vartheta) \|^2_{L^2(Q)} \Big) \dt\aleq c\left(m, R, \underline{\vr}, \Ov{\vr}, \underline{\vt}, \Ov{\vt} \right)
\end{split}
\eF
uniformly in $h$, where, similarly to (\ref{wWS110}), the bound is deterministic.
This immediately implies the bounds
\begin{equation} \label{m135A}
\left\| \vt \right\|_{C^{\nu(\alpha)}([0,T]; W^{\alpha, 2}(Q))} \leq c\left(m, R, \underline{\vr}, \Ov{\vr}, \underline{\vt}, \Ov{\vt} \right),\
\nu(\alpha) > 0 \ \mbox{for any}\ 0 \leq \alpha < 1,
\end{equation}
see e.g. Amann \cite{Amann1}.
\\
Now, we are ready to estimate the velocity.
In order to do so, we will systematically use the fact that all norms are equivalent on the finite dimensional space $H_m$.
It follows from (\ref{wWS16}) and the equivalence of norms on $H_m$ that
\[
\begin{split}
\intQ{\vr \vu (\tau, \cdot) \cdot \bfphi } & \aleq  \| \vu_0 \|_{H_m}
+ \int_{0}^\tau \sup_{0 \leq s \leq t} \| \vu \|_{H_m} \dt + T \\
&+
\left\| \int_{0}^\tau \intQ{ \vr \Pi_m [\vc{F}_{\ep, \xi} ]_h  }  \ \D \W \right\|_{H_m} \mbox{for any}\ \bfphi \in H_m,\ \| \bfphi \|_{H_m} \leq 1,
\end{split}
\]
whenever $0 \leq \tau \leq T$.
Consequently, taking the supremum over $\bfphi$, we obtain
\[
\| \Pi_m[\vr \vu] (\tau, \cdot) \|_{H_m}  \aleq  \| \vu_0 \|_{H_m}
+ \int_{0}^\tau \sup_{0 \leq s \leq t} \| \vu \|_{H_m} \dt + T +
\left\| \int_{0}^\tau \intQ{ \vr \Pi_m [\vc{F}_{\ep, \xi} ]_h  }  \ \D \W \right\|_{H^m},
\]
or,
\[
\begin{split}
&\| \Pi_m[\vr \vu] (\tau, \cdot) \|_{H_m}^r \\ &\aleq  c(T,k) \left[ \| \vu_0 \|_{H_m}^r
+ \int_{0}^\tau \sup_{0 \leq s  \leq t} \| \vu \|^r_{H_m} \dt + T +
\left\| \int_{0}^\tau \intQ{ \vr \Pi_m [\vc{F}_{\ep,\xi} ]_h  }  \ \D \W \right\|_{H^m}^r \right]
\end{split}
\]
for $0 \leq \tau \leq T$, and
for any $r \geq 1$.
Next, we pass to expectations and apply Burkholder-
Davis-Gundy inequality to control the last integral obtaining
\begin{align}
&\expe{ \sup_{0 \leq t \leq \tau} \| \Pi_m[\vr \vu] (t, \cdot) \|_{H_m}^r }  \aleq  \expe{ \| \vu_0 \|_{H_m}^r }
+ \int_{0}^\tau \expe{ \sup_{0 \leq s \leq t} \| \vu \|_{H_m}^r } \dt + T \nonumber\\ &+
\expe{ \int_0^\tau \sum_{k = 1}^\infty \left| \intQ{ \vr \Pi_m [ [\vc{F}_{k,\ep, \xi} ]_h  ]  } \right|^2  \dt }^{r/2}\nonumber\\
&\aleq  \expe{ \| \vu_0 \|_{H_m}^r }
+ \int_{0}^\tau \expe{ \sup_{0 \leq s \leq t} \| \vu \|_{H_m}^r } \dt + T +
\expe{ \int_0^\tau \sum_{k = 1}^\infty \|  \Pi_m [ \vc{F}_{k,\ep, \xi}  ] \|^2_{L^{\infty}(Q)}  \dt }^{r/2} \label{wWS112}
\end{align}
where we have used the uniform bounds for the density  obtained in
(\ref{wWS110}).
Finally, we use (\ref{wW2a}) to bound the last integral,
\[
\|  \Pi_m [ \vc{F}_{k,\ep, \xi}  ] \|^2_{L^{\infty}_x} \aleq \left\| \vc{F}_{k,\ep, \xi} \right\|_{W^{2,2}_x}
\aleq c( \xi )  \left\| \vc{F}_{k,\ep, \xi} \right\|_{L^\infty_x} \aleq f_{k,\ep} c(\xi).
\]
\\
Seeing that
\[
\vu = \mathcal{M}^{-1}[\vr][\Pi_m [\vr \vu]]
\]
we may use again the bounds (\ref{wWS110}), (\ref{wWS111}) to conclude that
\[
\| \Pi_m [\vr \vu] \|_{H_m} \aleq \| \vu \|_{H_m} \aleq \| \Pi_m [\vr \vu] \|_{H_m}
\]
where the constants in $\aleq$ depend only on $\underline{\vr}$, $\Ov{\vr}$.
Consequently, a direct application of Gronwall's lemma gives rise to the estimate
\begin{equation} \label{wWS113}
\expe{ \sup_{\tau \in [0,T] } \| \Pi_m [\vr \vu](\tau, \cdot) \|_{H_m}^r } +
\expe{ \sup_{\tau \in [0,T] } \| \vu (\tau, \cdot) \|_{H_m}^r } \aleq c(r, T) \expe{1+ \| \vu_0 \|_{H^m}^r }, \ r \geq 1.
\end{equation}
\\
In addition to the uniform bound (\ref{wWS113}) we will need compactness of the approximate velocities in the space $C([0, T]; H_m)$.
Moreover, we have to control the difference
\[
(\vu - [\vu]_h)
\]
uniformly in time. To this end, estimates on the modulus of continuity of $\vu$ are needed.
Evoking (\ref{wWS16}) we obtain
\[
\begin{split}
&\intQ{ \left[\vr \vu (\tau_1, \cdot) - \vr \vu(\tau_2, \cdot) \right]  \cdot \bfphi }\\ &=
\int_{\tau_1}^{\tau_2} \intQ{ \Big(\vr [\vu ]_{h,R} \otimes [\vu]_h \Big): \Grad \bfphi } \dt  + \int_{\tau_1}^{\tau_2} \intQ{ \chi(\| [\vu]_h  \|_{H_m} - R )\, p_\delta (\vr,\vt) \Div \bfphi } \dt
\\
&+ \int_{\tau_1}^{\tau_2} \intQ{  \ep \vr [\vu]_h \cdot \Del \bfphi } \dt- \int_{\tau_1}^{\tau_2} \intQ{  \frac{1}{m}  [\vu]_h \cdot \bfphi } \dt \\&  - \int_{\tau_1}^{\tau_2} \intQ{  \chi(\| [\vu]_h \|_{H_m} - R ) \mathbb{S}(\vt,\Grad [\vu]_h ): \Grad \bfphi } \dt \\
&+
\int_{\tau_1}^{\tau_2} \intQ{ \vr \Pi_m [\vc{F}_{\ep, \xi}]_h \cdot \bfphi }  \ \D \W \ \mbox{for any}\ \bfphi \in H_m,\
0 \leq \tau_1 < \tau_2.
\end{split}
\]
With the bound (\ref{wWS113}) at hand, we may repeat the arguments leading to (\ref{wWS112}) to obtain
\[
\expe{ \left\| \Pi_m \left[ \vr \vu (\tau_1) - \vr \vu (\tau_2) \right] \right\|_{H_m}^r } \leq c(r,T) \left| \tau_1 - \tau_2 \right|^{r/2} \expe{
\left( \| \vu_0 \|_{H^m}^{r} + 1 \right) },\ r \geq 1
\]
whenever $0 \leq \tau_1 < \tau_2 \leq T$, $|\tau_1 - \tau_2| \leq 1$. Thus we may apply Kolmogorov continuity criterion to conclude that  $\Pi_m[\vr\vu]$  has $\prst$-a.s. $\beta$-H\"older continuous trajectories for all $\beta\in (0,\frac{1}{2}-\frac{1}{r})$.
\[
\expe{ \left\| \Pi_m [\vr \vu] \right\|^r_{C^{\beta }([0,T]; H_m) } }
\leq c(r,T) \expe{ \| \vu_0 \|_{H_m}^{r} + 1 } ,\ r > 2.
\]
Recalling the relation
\[
\vu = \mathcal{M}^{-1}[\vr]\Pi_m [\vr \vu],
\]
boundedness of $\vr$ from \eqref{wWS110} and \eqref{eqM-1} we may infer that
\begin{equation} \label{wWS114}
\expe{ \left\|  \vu \right\|_{C^{\beta }([0,T]; H_m) } }
\aleq  \expe{ \| \vu_0 \|_{H_m}^{r} + 1 }
\end{equation}
uniformly in $h$ whenever $r > 2$ and $\beta\in (0,\frac{1}{2}-\frac{1}{r})$ with a constant independent of $h$.
\\
With the estimates (\ref{wWS110}), (\ref{m135}), and (\ref{wWS114}) at hand, we are ready to perform the limit $h \to 0$ in the approximate scheme (\ref{wWS14}--\ref{wWS16}).
Consider the joint law of the basic state variables $(\vr,\vt, \vu, \W)$ ranging in the \emph{pathspace}
\begin{align*}
\mathfrak{X} &\equiv C^\iota([0,T]; C^{2 + \iota}(Q))\times \left[ C^\iota([0,T];L^2(Q)) \cap
L^2(0,T; W^{1,2}(Q)) \right]\\&
\times  C^\iota([0,T]; H_m)  \times C([0,T]; \mathfrak{U}_0),
\end{align*}
$\iota\in (0, \Ov{\nu})$, where $\Ov{\nu} > 0$ is the minimum of the H\" older exponents in (\ref{wWS110}), \eqref{m135A}, (\ref{wWS114}), and (\ref{wWS115}).
Let $[\vr_h,\vt_h, \vu_h, \W]$ be the (unique) approximate solution issuing from the iteration scheme (\ref{wWS14}--\ref{wWS16}), with the initial data being $\mathfrak F_0$ measureable and satisfying \eqref{wWS111}, \eqref{wWS111b}
as well as
\begin{equation} \label{wWS115}
\begin{split}
&\expe{ \| \vu_0 \|^r_{H_m} } \leq \Ov{u} \ \mbox{for some}\ r > 2.
\end{split}
\end{equation}
Let $\mathcal{L}[\vr_h,\vt_h,\vu_h,W]$ denote the joint law of $[\vr_h,\vt_h,\vu_h,W]$ on $\mathfrak{X}$, whereas $\mathcal{L}[\vr_h]$, $\mathcal L[\vt_h]$, $\mathcal{L}[\vu_h]$ and $\mathcal{L}[W]$ denote the corresponding marginals on%$\mathfrak{X}_\vr$, $\mathfrak X_{\vt}$, $\mathfrak{X}_\vu$ and $\mathfrak{X}_W$
, respectively.
\\
In view of the bounds (\ref{wWS18}), (\ref{wWS19}), \eqref{m135}, \eqref{m135A} and (\ref{wWS113}), (\ref{wWS114}), we conclude that $\mathcal{L}[\vr_h,\vt_h,\vu_h,W]$ is \emph{tight} on the \emph{Polish} space $\mathcal{X}$. We may therefore apply Skorokhod's representation
theorem to obtain the following.
\begin{Proposition}\label{prop:skorokhod}
There exists a complete probability space $(\tilde\Omega,\tilde\mf,\tilde\prst)$ with $\mathfrak{X}$-valued Borel measurable random variables $(\tilde\varrho_h,
\tilde \vt_h, \tilde\bu_h,\tilde W_h)$, $h\in(0,1)$, and $(\varrho, \vt, \bu,\tilde W)$ such that (up to a subsequence)
\begin{enumerate}
 \item the law of $(\tilde\varrho_h, \tilde \vt_h, \tilde\bu_h,\tilde W_h)$ on $\mathfrak{X}$ is given by $\mathcal{L}[\vr_h,\vt_h, \vu_h,W]$, $h\in(0,1)$,
\item the law of $(\varrho, \vt, \bu, \tilde W)$ on $\mathfrak{X}$ is a Radon measure,
 \item $(\tilde\varrho_h, \tilde \vt_h, \tilde\bu_h,\tilde W_h)$ converges $\,\tilde{\prst}$-almost surely to $(\tilde\varrho, \tilde\vt, {\tilde\bu},\tilde{W})$ in the topology of $\mathfrak{X}$, i.e.
\begin{equation} \label{wWS116}
\begin{split}
\tilde \vr_h &\to \tilde\vr \ \mbox{in}\ C^\iota([0,T]; C^{2 + \iota}(\Ov{Q})) \ \tilde{\mathbb P}\mbox{-a.s.}, \\
\tilde \vt_h &\to \tilde\vt \ \mbox{in}\ C^\iota([0,T];W^{1,2}(Q)) \cap L^2(0,T; W^{1,2}(Q)) \ \tilde{\mathbb P}\mbox{-a.s.}, \\
\tilde \vu_h &\to \tilde\vu \ \mbox{in}\ C^\iota([0,T]; H_m ) \ \tilde{\mathbb P}\mbox{-a.s.}, \\
\tilde \W_h &\to  \tilde\W \ \mbox{in}\ C([0,T]; \mathfrak{U}_0 )\ \tilde{\mathbb P}\mbox{-a.s.}
\end{split}
\end{equation}
\end{enumerate}
\end{Proposition}

Since the trajectories of $\tilde\vr$, $\tilde\vt$ and $\tilde\vu$ are $\tilde\p$-a.s. continuous, progressive measurability with respect to their canonical filtrations
follows from adaptivity of the approximate sequence.
Consequently, they are progressively measurable with respect to the canonical filtration generated by $[\tilde\vr,\tilde\vt,\tilde\vu,\tilde W]$, namely,
$$\tilde\mf_t:=\sigma\big(\sigma_t[\tilde\vr]\cup\sigma_t[\tilde\vt]\cup\sigma_t[\tilde\vu]\cup\cup_{k=1}^\infty\sigma_t[\tilde  W_k]\big),\quad t\in[0,T].$$
Moreover, it is standard to show that $\tilde W$ is a cylindrical Wiener process with respect to $(\tilde\mf_t)_{t\geq 0}$.
Now, we show that $[ \tilde\vr, \tilde\vu]$ solves the approximate continuity equation.

\begin{Lemma}\label{lem:cont1}
The process $[ \tilde\varrho, {\tilde\bu}]$ satisfies \eqref{rho0} in $(0,T)\times Q$, $\tilde\p$-a.s.
\end{Lemma}

\begin{proof}
As a consequence of the equality of laws from Proposition \ref{prop:skorokhod}
and Theorem \ref{RDT3}, we see that the approximate continuity equation \eqref{wWS15} is satisfied on the new probability space by $[\tilde\vr_h,\tilde\vu_h]$.
Moreover, the uniform bounds \eqref{wWS110}, \eqref{wWS114} hold true also for $[\tilde\vr_h,\tilde\vu_h]$. Hence by Proposition \ref{prop:skorokhod} and Vitali's convergence theorem we may pass to the limit in \eqref{wWS15} and deduce that $[\tilde\vr,\tilde\vu]$ is a weak solution to the approximate continuity equation \eqref{wW3}. Furthermore,  the bounds \eqref{wWS110}, \eqref{wWS114} are also valid for the limit process $[\tilde\vr,\tilde\vu]$. Consequently, \eqref{wW3} is satisfied a.e. in $(0,T)\times\mt$, $\tilde\p$-a.s.
Finally, using parabolic regularity theory, we conclude that \eqref{rho0} is satisfies in the classical sense.
\end{proof}

As the next step, we are now going to show that the quantity $[\tilde\vr, \tilde\vt, \tilde\vu , \tilde \W ]$ solves the approximate momentum equation.

\begin{Lemma}\label{prop:m1}
The process $[\tilde\varrho, \tilde\vt, {\tilde\bu},\tilde W]$ satisfies \eqref{rhou0} for all $\bfphi\in H_m$ and $t\in[0,T]$, $\tilde\p$-a.s.
\end{Lemma}

\begin{proof}
Modifying slightly the proof, the result of Theorem \ref{RDT3} remains valid in the current setting. Hence as a consequence of the equality of laws from Proposition~\ref{prop:skorokhod}, the approximate momentum equation \eqref{wWS16} is satisfied on the new probability space by $[\tilde\vr_h,
\tilde \vt_h, \tilde\vu_h,\tilde W_h]$.
It is enough to pass to the limit with respect to $h$.\\
We observe that
\[
\left\| [\tilde{\vc{u}}_h]_h(t) - \tilde{\vc{u}}_h (t)\right\|_{H_m} \aleq h^\iota \| \tilde{\vc{u}}_h \|_{C^\iota_t H_m}
\]
and similarly
\begin{align*}
\left\| [\tilde{\vr}_h]_h (t)- \tilde{\vr}_h (t)\right\|_{C^{2+\iota}_x} \aleq h^\iota \| \tilde{\vr}_h \|_{C^\iota_t C^{2+\iota}_x},\\
\left\| [\tilde{\vt}_h]_h (t)- \tilde{\vt}_h (t)\right\|_{L^2_x} \aleq h^\iota \| \tilde{\vt}_h \|_{C^\iota_t L^2_x}.
\end{align*}
\\
Now, with the convergences \eqref{wWS116}, the bounds  \eqref{wWS110}, \eqref{wWS114} and the assumption \eqref{P-1} at hand
we may  pass to the limit in the approximate momentum equation \eqref{wWS16}. The only term which needs an explanation is the stochastic integral. By the uniform convergence of $\tilde\vr_h$ and $\tilde\bfu_h$ (recall Proposition \ref{prop:skorokhod}), the continuity of the coefficients $\mathbf{F}_{k,\ep, \xi}$ and the continuity of $\Pi_m$,
it is easy to see that  $\tilde\p$-a.s.
\begin{align}\label{eq:1707a}
\begin{aligned}
\Pi_m &\left[ [\tilde\vr_h]_h \Pi_m [\mathbf{F}_{k,\ep, \xi}([\tilde\vr_h]_h,[\tilde\vt_h]),[\tilde\vu_h]_h)] \right] \\&\to \Pi_m \left[ \vr \Pi_m [\mathbf{F}_{k,\ep,\xi}(\tilde\vr, \tilde\vt,  \tilde\vu ) \right]  \ \mbox{in}\ L^q((0,T)\times\Q)
\end{aligned}
\end{align}
for all $k\in\mathbb N$ and all $q<\infty$. On the other hand, we have
\begin{align*}%\label{wWS230zz}
\tilde\E &\int_0^T \big\|\Pi_m \left[ [\tilde\vr_h]_h \Pi_m [\vc{F}_{\ep, \xi} ([\tilde\vr_h]_h,[\tilde\vt_h],[\tilde\vu_h]_h)] \right] \big\|_{L_2(\mathfrak U;L^2_x)}^2\dt\\
&\leq \sum_{k=1}^\infty\tilde\E \int_0^T \left\| \left[ [\tilde\vr_h]_h \Pi_m [\mathbf{F}_{k,\ep, \xi}([\tilde\vr_h]_h,[\tilde\vt_h] ,[\tilde\vu_h]_h)] \right] \right|^2_{L^2_x} \dt\\
&\leq \|\tilde\vr_h\|_{L^\infty_{\omega,t,x}}^2\sum_{k=1}^\infty\tilde\E \int_0^T \left\| \Pi_m [\mathbf{F}_{k,\ep,\xi}([\tilde\vr_h]_h, \tilde\vt_h]_h,[\tilde\vu_h]_h)] \right\|^2_{L^2_x} \dt\\
&\leq \|\tilde\vr_h\|_{L^\infty_{\omega,t,x}}^2\sum_{k=1}^\infty\tilde\E \int_0^T \left\| \mathbf{F}_{k,\ep,\xi}([\tilde\vr_h]_h,[\tilde\vt_h]_h,[\tilde\vu_h]_h) \right\|^2
_{L^2_x} \dt \aleq \|\tilde\vr_h\|_{L^\infty_{\omega,t,x}}^2\sum_{k=1}^\infty f_{k,\ep}^2\aleq c
\end{align*}
using \eqref{wW2a}, \eqref{P-1new} as well as \eqref{wWS110}.
Consequently, we can strengthen
\eqref{eq:1707a} to
 \begin{align}\label{eq:1707b}
\begin{aligned}
\Pi_m &\left[ [\tilde\vr_h]_h \Pi_m [\vc{F}_{\ep, \xi} ([\tilde\vr_h]_h,[\tilde\vt_h],[\tilde\vu_h]_h)] \right]\\ &\to \Pi_m \left[ \vr \Pi_m [\vc{F}_{\ep, \xi} (\tilde\vr, \tilde\vt, \tilde\vu)] \right]  \ \mbox{in}\ L^2(0,T;L_2(\mathfrak U;L^2(\mt))
\end{aligned}
\end{align}
$\tilde\p$-a.s.
Combining this with the convergence of $\tilde W_h$ from Proposition \ref{prop:skorokhod} we may apply Lemma Lemma \ref{RDL4} to pass to the limit in the stochastic integral and hence complete the proof.

\end{proof}

Next, we show:

\begin{Lemma}\label{prop:s1b}
The process $[\tilde\varrho, \tilde\vt, {\tilde\bu}]$ satisfies \eqref{rhos0} a.a. in $(0,T) \times Q)$ $\tilde{\mathbb{P}}$-a.s.
\end{Lemma}
\begin{proof}
As a consequence of the equality of laws from Proposition \ref{prop:skorokhod}
and Theorem \ref{RDT3}, we see that the approximate internal energy balance \eqref{wWS15b} is satisfied on the new probability space by $[\tilde\vr_h,\tilde\vt_h,\tilde\vu_h]$.
Moreover, the uniform bounds \eqref{m135} hold true also for $[\tilde\vr_h]$. Using Proposition \ref{prop:skorokhod} and  we may pass to the limit in \eqref{wWS15b} and deduce that $[\tilde\vr,\tilde\vt,\tilde\vu]$ is a weak solution to the approximate internal energy equation \eqref{rhos0}.
Furthermore, the limit process $\tilde\vt$ also belongs to the class \eqref{m125b}. Consequently, \eqref{rhos0} is satisfied a.e. in $(0,T)\times Q$, $\tilde\p$-a.s.
\end{proof}

Finally, as $\tilde{\vt}_h$ obeys the (deterministic) bounds \eqref{m135}, \eqref{m135A}, the limit $\tilde\vt$ belongs to the same class. In particular, the limit temperature
$\tilde\vt$ enjoys the regularity claimed in Theorem \ref{wWP1}.
\\
The proof of  Theorem \ref{wWP1} is hereby complete.
\subsection{Energy balance}
\label{wWS1S3}

We show that any solution of the approximate problem (\ref{wW3}--\ref{m119}) satisfies a variant of the energy balance
equation. To this end, we take the scalar product
of (\ref{wW4}) with $\vu$ and integrate the resulting expression by parts.
We apply It\^{o}'s formula to the scalar product
\[
\intQ{ \Pi_m (\vr \vu) \cdot \vu } = \intQ{ \vr |\vu|^2 }.
\]
As
\[
\intQ{ \mathcal{M}^{-1}[\vr] \Pi_m [\vc{v}] \cdot \Pi_m [\vr \vu] } = \intQ{ \vr \mathcal{M}^{-1}[\vr] \Pi_m [\vc{v}] \cdot \vu } =
\intQ{ \vc{v} \cdot \vu },
\]
we deduce from (\ref{wW4}) that
\begin{align}
\D& \intQ{ \frac{1}{2} \vr |\vu|^2 }\nonumber\\ = &- \intQ{ \Big[ \Div (\vr [\vu]_R \otimes \vu)
+   \chi(\| \vu \|_{H_m} - R )\Grad p_\delta (\vr,\vt) \Big] \cdot \vu } \dt\nonumber\\
&+  \intQ{ \Big[ \ep \Del (\vr \vu) + \chi(\| \vu \|_{H_m} - R ) \Div \mathbb{S}(\vt,\Grad \vu)+\frac{1}{m} \bfu \Big] \cdot \vu } \dt - \frac{1}{2} \intQ{ |\vu|^2 \D \vr }\nonumber
\\
& + \frac{1}{2}\sum_{k\geq0} \intQ{ \vr |\Pi_m [\vc{F}_{k,\ep, \xi} ] |^2 } \dt +
\intQ{  \vr \Pi_m [ \vc{F}_{\ep, \xi} ]  \cdot \vu } \ \D \W.\label{wWS118}
\end{align}
Furthermore, equation (\ref{wW3}) tells us that
\[
\frac{1}{2} \intQ{ |\vu|^2 \D \vr } =  \frac{1}{2} \intQ{ \ep |\vu|^2 \Del \vr }\dt - \frac{1}{2} \intQ{ \Div (\vr [\vu]_R ) |\vu|^2 } \dt,
\]
while
\[
\intQ{ \Div (\vr [\vu]_R \otimes \vu) \cdot \vu } = - \frac{1}{2} \intQ{ \vr [\vu]_R \cdot \Grad |\vu|^2 } =
\frac{1}{2}  \intQ{ \Div (\vr [\vu]_R ) |\vu|^2 },
\]
and
\[
\ep \intQ{ \Del (\vr \vu) \cdot \vu } = - \ep \intQ{ \vr |\Grad \vu|^2 } + \frac{1}{2} \intQ{ \ep |\vu|^2 \Del \vr }.
\]
Consequently, relation (\ref{wWS118}) reduces to
\begin{align} \nonumber
\D \intQ{ \frac{1}{2} \vr |\vu|^2 } &+ \chi(\| \vu \|_{H_m} - R )\intQ{ \mathbb{S}(\vt,\Grad \vu) : \Grad \vu } \dt + \frac{1}{m} \intQ{ |\vu|^2 } \dt + \ep \intQ{ \vr |\Grad \vu|^2 } \dt  \\
\label{wWS119}&=  \intQ{
\chi(\| \vu \|_{H_m} - R ) \, p_\delta (\vr,\vt) \Div \vu } \dt\\
& + \frac{1}{2} \intQ{ \vr |\Pi_m [\vc{F}_{\ep, \xi}] |^2 } \dt +
\intQ{  \vr \Pi_m [ \vc{F}_{\ep, \xi} ]  \cdot \vu } \ \D \W.\nonumber
\end{align}
\\
Seeing that
\[
\chi(\| \vu \|_{H_m} - R ) \,p_\delta (\vr,\vt) \Div \vu = p_\delta (\vr,\vt)\Div [\vu]_R
\]
we rewrite the energy balance (\ref{wWS119}) as
\begin{align} \nonumber
\D \intQ{ \frac{1}{2} \vr |\vu|^2 }
=& - \chi(\| \vu \|_{H_m} - R )\intQ{ \mathbb{S}(\vt,\Grad \vu) : \Grad \vu } \dt  + \intQ{
 p_\delta (\vr,\vt) \Div [\vu]_R } \dt\\
\label{wWS119'}& - \ep \intQ{ \vr |\Grad \vu|^2 } \dt- \frac{1}{m} \intQ{ |\vu|^2 } \dt+ \sum_{k\geq0}\frac{1}{2} \intQ{ \vr |\Pi_m [\vc{F}_{k,\ep, \xi}] |^2 } \dt\\& +
\intQ{  \vr \Pi_m [ \vc{F}_{\ep, \xi} ]  \cdot \vu } \ \D \W.
\nonumber
\end{align}
For the first term on the right-hand side we use the approximate internal energy equation \label{wWS11b} to see
\begin{align*}
-&{\ep \int_{Q} \vr |\Grad \vu|^2 \dx }   -\chi(\| \vu \|_{H_m} - R )\int_{Q}\tens{S}(\vartheta,\Grad\vc u): \Grad \vc u\dx +\int_{Q}
{p(\varrho, \vartheta) \Div [\vc u]_R }\dx \\
&= \intQ{\left[ -\partial_t(\varrho e_\delta(\varrho, \vartheta)) -\ep
\delta (\beta \varrho^{\beta - 2} + 2) |\Grad \varrho |^2 +
\delta \frac{1}{{\vartheta}^2} - \ep \vartheta^5 \right] }.
\end{align*}
{Next, multiplying the equation of continuity on $b'(\vr)$ we deduce a renormalized equation}
\[
{\rm d} b(\vr) + \Div (b(\vr) [\vu]_R ) + \left( b'(\vr) \vr - b(\vr) \right) \Div [\vu]_R \dt = \Div (b'(\vr) \Grad \vr) \dt - b''(\vr) |\Grad \vr |^2 \dt
\]
{for any twice continuously differentiable function $b$.}
Inserting this into \eqref{wWS119'} we can write the energy balance in its
final form
\begin{equation} \label{wWS121}
\begin{split}
\D & {\intQ{ \left[ \frac{1}{2} \vr |\vu|^2 + \varrho e_\delta(\vr,\vt) + \delta \left( \frac{\vr^\beta}{\beta - 1} + \vr^2 \right)  \right] } } \\  &+ {\frac{1}{m} \intQ{ |\vu|^2 } \dt + \ep \intQ{ \vartheta^5}\dt}  \\
& =\intQ{\delta \frac{1}{{\vartheta}^2}}\dt+  \frac{1}{2}\sum_{k\geq0} \intQ{ \vr |\Pi_m [\vc{F}_{k,\ep, \xi}] |^2 } \dt +
\intQ{  \vr \Pi_m [ \vc{F}_{\ep, \xi} ]  \cdot \vu } \ \D W.
\end{split}
\end{equation}
\\
We have shown the following version of the energy balance for the approximate martingale solutions.

\begin{proposition} \label{wWP2}

Under the hypotheses of Theorem \ref{wWP1}, let $(\vr,\vt, \vu, \W)$ be a martingale solution of the approximate problem \eqref{wW3}--\eqref{m119}.
\\
Then
the following total energy energy equation
\begin{align} \label{EI20a0}
\begin{aligned}
- \int_0^T& \partial_t \psi
\bigg(\int_{Q} {\mathcal E_\delta(\varrho,\vartheta, \vu) } \dx\bigg) \dt
 + \int_0^T {\psi} \int_{Q} \Big( \ep \vartheta^5+ \frac{1}{m} |\vu|^2\Big)  \dt\\
= &\,\psi(0) \intTor{{\mathcal E_\delta(\varrho_0,\vartheta_0, \vu_0 )} }+\int_0^T\int_{Q}\frac{\delta}{\vartheta^2}\psi\dx\dt\\
&+ \frac{1}{2} \int_0^T
\psi \bigg(
\intQ{ \sum_{k \geq 1} \varrho | \Pi_m[{\bf F}_{k,\ep, \xi } (\varrho,\vartheta, {\bf u})] |^2  } \bigg) {\rm d}t\\
&+\int_0^T  \psi\int_{Q}\varrho\Pi_m[{\vc{F}_{\ep, \xi} }(\varrho,\vartheta,\bfu)]\cdot\bfu\,\dd W\dx
\end{aligned}
\end{align}
holds true for any deterministic test function $\psi \in \DC[0,T)$ ,  $\mathbb{P}$-a.s. Here, we abbreviated
$${\mathcal E_\delta(\varrho,\vartheta, \vu)= \frac{1}{2} \varrho | {\bf u} |^2 + \varrho e_{\delta}(\varrho,\vartheta)
+ \delta \left( \frac{\vr^\beta}{\beta - 1} + \vr^2 \right) .}$$

\end{proposition}

\begin{Remark} \label{wWR8}

Consistently with the weak formulation of the field equations in Definition \ref{WD1}, we have rewritten \eqref{wWS121} in the form of a variational
equality with a deterministic test function $\psi$.

\end{Remark}

\subsection{{Entropy balance}}

As equations (\ref{rho0}) and (\ref{rhos0}) are satisfied in the strong sense and $\vt$ is strictly positive, we may divide (\ref{rhos0}) by $\vt$  obtaining the (regularized) entropy balance equation
\begin{equation} \label{apeneq}
\begin{split}
\partial_t &(\varrho s_\delta(\varrho, \vartheta)) +
\Div (\varrho s_\delta (\varrho, \vartheta) [\vc u]_R) - \Div \Big[
\Big( \frac{\kappa(\vartheta)}{\vartheta} + \delta (
\vartheta^{\beta - 1} + \frac{1}{\vartheta^2}) \Big) \Grad
\vartheta \Big] \\
&= \frac{1}{\vartheta} \Big[ \chi\left( \| \vu \|_{H^m} - R \right) \tens{S}(\vt, \Grad \vu) : \Grad \vc u + \Big(
\frac{\kappa(\vartheta)}{\vartheta} + \delta ( \vartheta^{\beta -
1} + \frac{1}{\vartheta^2}) \Big) |\Grad \vartheta|^2 + \delta
\frac{1}{{\vartheta}^2} \Big]
\\ &+ \frac{\ep \delta}{\vartheta} ( \beta \varrho^{\beta - 2} + 2)
|\Grad \varrho|^2
+\ep \frac{\Delta_x \varrho}{\vartheta} \Big( \vartheta
s_\delta(\varrho, \vartheta) - e_\delta(\varrho, \vartheta) -
\frac{p(\varrho, \vartheta)}{\varrho} \Big) - \ep \vartheta^4 + {\frac{\ep}{\vt} \vr |\Grad \vu|^2}
\end{split}
\end{equation}
satisfied a.a. in $(0,T) \times Q$, together with the boundary conditions $\Grad \vt \cdot \vc{n}|_{\partial Q} = 0$.

\section{{Galerkin approximation}}
\label{EGA}

Our goal is to perform several steps: (i) letting $R \to \infty$ in the approximate system (\ref{wW3})--(\ref{m119}), (ii) letting $m \to \infty$ in the resulting limit,
(iii) letting the parameter $\xi \to 0$.
The technique in these three steps is rather similar and is based on the uniform bounds enforced by the data.
In the following we amply use the Korn--Poincar\' e inequality:
\begin{equation} \label{Poinc}
\| \vc{v} \|^2_{W^{1,2}(Q, R^3)} \leq c \intQ{ \frac{1}{\vt} \mathbb{S}(\vt, \Grad \vc{v}): \Grad \vc{v} } \ \mbox{for all}\
\vc{v} \in W^{1,2}(Q; R^3),\ \vc{v}|_{\partial Q} = 0,
\end{equation}
cf. also hypothesis \eqref{m105}.

\subsection{Uniform bounds}

We start by introducing the \emph{ballistic free energy},
\[
H_\Theta(\vr, \vt) = \vr \left( e(\vr, \vt) - \Theta s (\vr, \vt) \right),\
H_{\delta, \Theta}(\vr, \vt) = \vr \left( e_\delta(\vr, \vt) - \Theta s_\delta (\vr, \vt) \right),
\]
cf. \cite[Chapter 2, Section 2.2.3]{F}.
Combining the total energy balance \eqref{EI20a0}, with the entropy equation \eqref{apeneq}
we get the \emph{total dissipation balance}
\begin{align} \label{wWS24}
\begin{aligned}
&
 \intQ{ \Big[ \frac{1}{2} \varrho | {\bf u} |^2 + H_{\delta,\Theta}(\varrho,\vartheta)+\frac{\delta}{\beta-1}\varrho^\beta
+ { \delta \vr^2}  \Big](\tau, \cdot) }
+\Theta \int_0^\tau\int_{\mt}\sigma_{R,m, \varepsilon,\delta}\dx\\
&+ \int_0^\tau\int_{Q} \Big( \ep \vartheta^5+  \frac{1}{m}|\vu|^2\Big)  \dt\\
&= \,\intQ{ \Big[ \frac{1}{2} \vr_0 |\vu_0|^2  + H_{\delta,\Theta}(\varrho_0,\vartheta_0)+\frac{\delta}{\beta-1}\varrho_0^\beta + {\delta \vr_0^2} \Big] }\\&+\varepsilon\int_0^\tau\int_{Q}\frac{\Theta}{\vartheta^2}\bigg(e_{M,\delta}(\varrho,\vartheta)+\varrho\frac{\partial e_M}{\partial\varrho}(\varrho,\vartheta)\bigg)\nabla\varrho\cdot\nabla\vartheta\dx\dt\\
&+\int_0^\tau\int_{Q}\bigg(\frac{\delta}{\vartheta^2}+\varepsilon\Theta \vartheta^4\bigg)\dx\dt
+\int_0^\tau \int_{Q}\varrho\Pi_m[{\vc F}_{\ep, \xi} (\varrho,\vartheta,\bfu)]\cdot\bfu\,\dd W\dx\\&+ \frac{1}{2} \int_0^\tau
\bigg(
\intQ{ \sum_{k \geq 1} \varrho | \Pi_m[{\bf F}_{k,\ep, \xi} (\varrho,\vartheta,{\bf u})] |^2  } \bigg) {\rm d}t
\end{aligned}
\end{align}
for any $0 \leq \tau \leq T$, and any positive constant $\Theta > 0$, $\pas$,
where
\begin{align*}
\sigma_{R,m, \varepsilon,\delta}&=\frac{1}{\vartheta}\Big[\chi(\| \vu \|_{H_m} - R )\mathbb S(\vartheta,\nabla\bfu):\nabla\bfu+\frac{\kappa(\vartheta)}{\vartheta}|\nabla\vartheta|^2+\frac{\delta}{2}\Big(\varrho^{\beta-1}+\frac{1}{\vartheta^2}\Big)|\nabla\vartheta|^2+\delta\frac{1}{\vartheta^2}\Big]\\
&+\frac{\varepsilon\delta}{2\vartheta} {\left( \beta\varrho^{\beta-2} + 2 \right)} |\nabla\varrho|^2+\varepsilon\frac{\partial p_M}{\partial\varrho}(\varrho,\vartheta)\frac{|\nabla\varrho|^2}{\varrho\vartheta} + \frac{\vr}{\vt} |\Grad \vu|^2 .
\end{align*}
\\
Keeping $\ep > 0$, $\delta > 0$, $\xi > 0$ fixed we derive bounds independent of the parameters $R$ and $m$.
As the projections $\Pi_m$ are bounded by \eqref{wW2a}, we get
\begin{equation} \label{wWS25}
\begin{split}
\sum_{k\geq0}\int_{Q} &\vr \left| \Pi_m \left[ \vc{F}_{k,\ep, \xi} (\vr, \vartheta ,\vu) \right] \right|^2 \dx \aleq \sum_{k\geq0}\| \vr \|_{L^1_x} \| \mathbf F_{k,\ep, \xi} (\vr,\vartheta, \vu) \|^2_{W^{2,2}_x}\\
 &\aleq c(\xi) \sum_{k\geq0}\| \vr \|_{L^1_x } \| \vc{F}_{k,\ep} (\vr, \vartheta, \vu) \|^2_{L^\infty_x} \aleq c(\ep, \xi, \Ov{\vr})
\end{split}
\end{equation}
using also
\begin{align}\label{eq:2610}
\| \vr \|_{L^1_x }=\| \vr_0 \|_{L^1_x }\leq\overline\vr.
\end{align}
Next, by means of the Burkholder-Davis-Gundy inequality,
\begin{equation} \label{wWS26}
\begin{split}
&\expe{ \sup_{0 \leq t \leq \tau} \left| \int_0^t \intQ{  \vr \Pi_m \Big[ \vc{F}_{\ep, \xi} (\vr, \vartheta, \vu) \Big]
\cdot \vu }  \ \D \W \right|^r }\\ &\aleq  \expe{ \int_0^\tau \sum_{k \geq 0} \left| \intQ{ \vr \Pi_m \left[ \vc{F}_{k,\ep, \xi} (\vr , \vartheta, \vu) \right]
\cdot \vu } \right|^2 }^{r/2}  , \ r \geq 1.
\end{split}
\end{equation}
Furthermore, using once more (\ref{wW2a}), we deduce
\[
\begin{split}
\left| \intQ{ \vr \Pi_m \left[ \vc{F}_{k,\ep, \xi} (\vr , \vartheta, \vu) \right]
\cdot \vu } \right|^2 &\aleq \Big| \| \sqrt{\vr} \|_{L^{2}_x} \| \sqrt{\vr} \vu \|_{L^2_x} \| \Pi_m [ \vc{F}_{k,\ep, \xi}(\vr, \vartheta, \vu) ] \|_{L^\infty_x} \Big|^2
\\ &\aleq c(\xi,\overline\vr) \| \sqrt{\vr} \vu \|_{L^2_x}^2 \| \vc{F}_{k,\ep}(\vr, \vartheta, \vu)  \|^2_{L^\infty_x} \leq c(\xi,\overline\vr) f_{k,\ep}^2 \| \sqrt{\vr} \vu \|_{L^2_x}^2.
\end{split}
\]
\\
Next,
we observe that the term
$\delta/ \vartheta^2$ on the right-hand side of (\ref{wWS24}) is
dominated by its counterpart $\delta/\vartheta^3$ in the entropy
production term $\sigma_{R,m, \ep, \delta}$. Analogously, the quantity
$\ep \Theta \vartheta^4$ on the right hand side
is ``absorbed'' by the term $\ep \vartheta^5$ at the left hand side
of (\ref{wWS24}).
\\
Consequently, it remains to handle the quantity
\[
\ep \intQ{ \frac{1}{\vartheta^2} \Big( e_M(\varrho, \vartheta) +
\varrho \frac{\partial e_M (\varrho, \vartheta)}{\partial \varrho }
\Big) \Grad \varrho \cdot \Grad \vartheta }
\]
appearing on the right-hand side of (\ref{wWS24}).
To this end, we first use hypothesis (\ref{md8!}), together with
(\ref{md7}) and (\ref{md1}), in order to obtain
\[
\Big| \frac{1}{\vartheta^2} \Big( e_M(\varrho, \vartheta) + \varrho
\frac{\partial e_M (\varrho, \vartheta)}{\partial \varrho } \Big)
\Grad \varrho \cdot \Grad \vartheta \Big| \leq c \Big(
\frac{\varrho^{\frac{2}{3}} + \vartheta}{\vartheta^2} \Big) |\Grad
\varrho| |\Grad \vartheta|,
\]
where, furthermore,
\[
\frac{|\Grad \varrho| |\Grad \vartheta|}{\vartheta} \leq \kappa
\frac{|\Grad \varrho |^2 }{\vartheta} + c(\kappa) \frac{|\Grad
\vartheta|^2}{\vartheta} \ \mbox{for any}\ \kappa>0,
\]
and, similarly,
\[
\frac{\varrho^{\frac{2}{3}} |\Grad \varrho| |\Grad
\vartheta|}{\vartheta^2} \leq \kappa
\frac{\varrho^{\frac{4}{3}}|\Grad \varrho |^2 }{\vartheta} +
c(\kappa) \frac{|\Grad \vartheta|^2}{\vartheta^3}.
\]
Thus we infer that
\begin{align}\label{m140}
\begin{aligned}\ep  \int_{Q}& \frac{1}{\vartheta^2}
\Big| e_M(\varrho, \vartheta) + \varrho \frac{\partial e_M (\varrho,
\vartheta)}{\partial \varrho } \Big| |\Grad \varrho| |\Grad
\vartheta| \dx \\
&\leq \frac{1}{2}  \intQ{ \Big[ \delta \Big( \vartheta^{\beta - 2} +
\frac{1}{\vartheta^3} \Big) |\Grad \vartheta |^2 + \frac{\ep
\delta}{\vartheta} \Big( \beta \varrho^{\beta - 2} + 2 \Big)
|\Grad \varrho |^2 \Big] }
\end{aligned}
\end{align}
provided $\ep = \ep(\delta) > 0$ is small enough.
Consequently, passing to expectations in (\ref{wWS24}) we may apply Gronwall's inequality to conclude that
\begin{align}
\E\bigg[ \sup_{0\leq \tau\leq T}\int_{Q}&\Big[ \frac{1}{2} \varrho | {\bf u} |^2 + H_{\delta,\Theta}(\varrho,\vartheta)+\frac{\delta}{\beta-1}\varrho^\beta + \delta \vr^2\Big] \dx
+\Theta \int_0^T\int_{Q}\sigma_{R,m, \varepsilon,\delta}\dx
 +\varepsilon\int_0^T\int_{Q}\vartheta^5\dx\dt\bigg]^r\nonumber\\
 \label{wWS27}\leq &\,c(T, \xi) \,\E\bigg[\intQ{ \Big[ \frac{1}{2} \vr_0 |\vu_0|^2   + H_{\delta,\Theta}(\varrho_0,\vartheta_0)+\frac{\delta}{\beta-1}\varrho_0^\beta + \delta \vr^2_0 \Big] }\bigg]^r
\end{align}
for any $r\geq1$.
\\
Taking into account the properties of the function
$H_{\delta,\Ov{\vartheta}}$, see \cite[Section 2.2.3, (2.49), (2.50)]{F}, we obtain
the following bounds depending only on the initial data $(\vr_0, \vt_0, \vu_0)$ determined in terms of their law $\Lambda$,
and the parameter $\xi$:
\begin{equation}\label{m142*}
\begin{split}
\expe{ \left|  \sup_{t \in (0,T)} \intQ{ \Big( \frac{1}{2}
\varrho |\vc u|^2 + H_{\delta,\Theta}(\varrho, \vartheta) +
{\delta}(\frac{ \varrho^{\beta}}{\beta - 1} + \varrho^2) \Big)
 } \right|^r } &\leq c(r, \xi, \Lambda)  ,
\\
\expe{ \left| \int_0^{T}\chi(\| \vu \|_{H_m} - R ) \intO{ \frac{1}{\vartheta} \Big[
\tens{S}(\vartheta,\Grad\vc u) : \Grad \vc u\Big] } \dt \right|^r } &\leq c(r, \xi, \Lambda),
\\
\expe{ \left| \int_0^{T}\intO{\frac 1\vartheta\Big(
\frac{\kappa(\vartheta)}{\vartheta} + \delta ( \vartheta^{\beta -
1} + \frac{1}{\vartheta^2}) \Big) |\Grad \vartheta|^2 \Big)} \ \dt \right|^r } &\leq c(r, \xi, \Lambda),
\\
\expe{ \left| \int_0^{T}\intO{\Big( \ep \delta
\frac{1}{{\vartheta}^3}+ \ep \vartheta ^5+ \frac{1}{m}|\bfu|^2+ {\ep \frac{\vr}{\vt}|\nabla\bfu|^2} \Big)} \ \dt \right|^r } &\leq c(r, \xi, \Lambda),
\\
{\ep \delta} \expe{ \left| \int_0^{T} \intO{ \frac{1}{\vartheta} (
\beta \varrho^{\beta - 2}+ 2)|\Grad \varrho|^2 } \right|^r \dt} &\leq c(r, \xi, \Lambda) ,
\\
\ep \expe{\left| \int_0^{T}\intO{\frac{\Ov\vartheta}{\varrho\vartheta}\frac{\partial
p_M}{\partial\varrho}(\varrho,\vartheta)|\Grad\varrho|^2 }\ {\rm d}t \right|^r } &\leq c(r, \xi, \Lambda).
\end{split}
\end{equation}
\\
As all norms are equivalent on the finite-dimensional space $H_m$ and $\partial Q$ is regular,,
we deduce from \eqref{m142*}$_4$ that
\begin{equation} \label{wWS28}
\begin{split}
&\expe{ \left| \int_0^T \| \vu \|_{C^{2}(\Ov{Q})}^2 \dt\right|^r }  \leq c(m, \xi, \ep, \Lambda).
\end{split}
\end{equation}
\\
Finally, we recall the coercivity properties of the function $H_{\delta, \Theta}$, see \cite[Chapter 3, Proposition 3.2]{F}:
\begin{equation} \label{Nhelm}
H_{\delta, \Theta} (\vr, \vt) \geq \frac{1}{4} \left( \vr e_\delta (\vr, \vt) + \Theta \vr |s(\vr, \vt)| \right) -
\left| (\vr - \Ov{\vr}) \frac{ \partial H_{\delta, 2 \Theta} }{\partial \vr} (\Ov{\vr}, 2 \Theta) + H_{\delta, 2 \Theta}(\Ov{\vr}, 2 \Theta )
\right|
\end{equation}
for any positive $\vr, \vt, \Ov{\vr}, \Theta$. Consequently, (\ref{m142*})$_1$ gives rise to
\begin{equation} \label{Nm142*}
\expe{ \left| \sup_{t \in (0,T) } \intO{ \left( \vt^4 + \delta \vr |\log(\vt)| \right) } \right|^r } \leq c(\xi, \Lambda).
\end{equation}

\subsection{Limit $R \to \infty$}
\label{wWS2S2}

Keeping $m$ fixed we consider
the rather restrictive hypothesis on the initial distribution of the data imposed in Theorem \ref{wWP1}:
\[
\begin{split}
\Lambda = \Lambda_R ,\ \Lambda_R &\left\{ 0 < \underline{\vr}_R \leq \vr_0 ,\
\| \vr_0 \|_{C^{2 + \nu}_x } \leq \Ov{\vr}_R,\ \Grad \vr_0 \cdot \vc{n}|_{\partial Q} = 0\right\} = 1 \\&
\Lambda_R \left\{ 0 < \underline{\vt}_R \leq \vt_0, \ \| \vt_0 \|_{W^{1,2}_x \cap C_x} \leq \Ov{\vt}_R  \right\}  = 1\\
&\int_{C^{2 + \nu}_x \times W^{1,2}_x \times H_m } \| \vu_0 \|^r_{H_m} \ {\rm d} \Lambda_R \leq \Ov{u}_R \ \mbox{for any}\ r \geq 1.
\end{split}
\]
As the uniform bounds will be lost in the limit $R \to \infty$ we suppose that
\begin{equation} \label{nGL1}
\Lambda_R \to \Lambda \ \mbox{weakly-(*) in} \ \mathcal{M}^+ \left( C^{2 + \nu} \times W^{1,2} \times H_m \right),
\end{equation}
where
\[
\begin{split}
&\Lambda_R \left\{ 0 < \underline{\vr} \leq \intQ{ \vr_0 } \leq \Ov{\vr},\ \Grad \vr_0 \cdot \vc{n}|_{\partial Q} = 0,\ \vr_0,\ \vt_0 > 0 \right\}\\
& =
\Lambda \left\{ 0 < \underline{\vr} \leq \intQ{ \vr_0 } \leq \Ov{\vr},\ \Grad \vr_0 \cdot \vc{n}|_{\partial Q} = 0,\ \vr_0,\ \vt_0 > 0 \right\}
= 1,\\ &\int_{C^{2 + \nu} \times W^{1,2} \times H_m }\left[ \| \vr_0 \|_{C^{2 + \nu}_x}
 + \| \vt_0 \|_{W^{1,2}_x \cap C_x } + \| \vu_0 \|_{H_m} \right]^r {\rm d} \Lambda_R \leq c(r)
\ \mbox{uniformly in}\ R,
\end{split}
\]
\\
By virtue of Theorem \ref{wWP1}, the approximate problem \eqref{wW}, \eqref{BCba} admits a martingale solution
$(\vr_R, \vt_R, \vu_R)$ with the initial law $\Lambda_R$ for any fixed $R > 0$. Our first goal is to justify the limit
$R \to \infty$.
The strategy is similar to Section \ref{wWS1S2}; we establish compactness of the phase variables and use
a variant of Skorokhod representation theorem.

\subsubsection{Compactness}

We start recalling the standard parabolic maximal regularity estimates (see e.g. \cite{HIPR} or \cite{LSU}) applied to (\ref{wW3}):
\begin{equation} \label{wWS220}
\begin{split}
\| \partial_t \vr \|_{L^p(0,T; L^q(Q))}&+\| \vr \|_{L^p(0,T; W^{2,q}(Q)) }\\ &\aleq \| \Div (\vr [\vu]_R) \|_{L^p(0,T; L^q(Q))} + \| \vr_0 \|_{C^{2 + \nu}(Q)}, \\
\| \vr \|_{L^p(0,T; W^{1,q}(Q)) } &\aleq \| \vr [\vu]_R \|_{L^p(0,T; L^q(Q))} + \| \vr_0 \|_{C^{2 + \nu}(Q)},
\end{split}
\end{equation}
for $1 < p,q < \infty$.
In (\ref{wWS220}), the regularity of the initial data can be considerably weakened. However, such generality is not needed here.
Now observe that (\ref{m142*}), (\ref{wWS28}) give rise to
\[
\expe{ \left| \| \vr [\vu]_R \|_{L^1(0,T; L^3(Q))} \right|^r } \leq c(r,m,R).
\]
It is worth noting that this estimate is independent of $R$ as long as
\begin{equation} \label{newEST}
\expe{ \left| \int_0^T \| \vu \|^2_{W^{1,2}(Q)} \dt \right|^r } \aleq 1.
\end{equation}
As we shall see below, estimate \eqref{newEST} remain valid at any stage of approximation.
\\
This interpolated with (\ref{m142*}) yields
\[
\expe{ \left| \| \vr \vu \|_{L^p(0,T; L^p(Q))} \right|^r } \aleq c(r) \ \mbox{for a certain}\ p > 2;
\]
which, plugged in the right--hand of (\ref{wWS220}), implies
\begin{equation} \label{wWS221}
\expe{ \left| \| \vr \|_{L^p(0,T; W^{1,p}(Q)) } \right|^r } \aleq c(r) \ \mbox{for some}\ p > 2.
\end{equation}
Finally, the estimates (\ref{wWS220}) and (\ref{wWS221}) can be used again in (\ref{wWS220}) to conclude that
\begin{equation} \label{wWS222}
\expe{ \left| \| \partial_t \vr \|_{L^p(0,T; L^p(Q))}+ \| \vr \|_{L^p(0,T; W^{2,p}(Q))} \right|^r } \aleq c(r)
\end{equation}
for some $p > 1$, where $c(r)$ depends also on the initial data, in particular on $\E\| \vr_0 \|_{C^{2 + \nu}_x}^r $.
\\
Now,
following the arguments introduced in Section \ref{wWS1S2}, we show compactness of $\vr \vu$ with respect to the time variable.
Similarly to Section \ref{wWS1S2} and in view of the bounds established above, it is enough to check the time continuity of the stochastic integral, namely
\[
\expe{ \left\| \int_{\tau_1}^{\tau_2}  \vr \Pi_m [\vc{F}_{\ep, \xi} (\vr, \vartheta, \vu) ] \D \W \right\|^r_{W^{-k,2}_x} }, \ k > \frac{N}{2}.
\]
Applying again the Burkholder-Davis-Gundy inequality, we obtain
\[
\begin{split}
&\expe{ \left\| \int_{\tau_1}^{\tau_2}  \vr \Pi_m [\vc{F}_{\ep, \xi} (\vr,\vartheta, \vu) ] \D \W \right\|^r_{ W^{-k,2}_x} } \\
&\aleq |\tau_1 - \tau_2 |^{r/2} \expe{ \sum_{k \geq 1}\sup_{0 \leq t \leq T} \left( \left\| \vr \Pi_m[ \vc{F}_{k,\ep, \xi}(\vr, \vartheta, \vu) ] \right\|_{L^1_x}  \right)^2 }^{\frac{r}{2}}
\ \mbox{for any}\ r \geq 1.
\end{split}
\]
Next, by H\" older's inequality,
\[
\begin{split}
\left\| \vr \Pi_m[ \vc{F}_{k,\ep, \xi }(\vr, \vartheta, \vu) ] \right\|_{L^1_x} &\leq \left\| \vr \right\|_{L^1_x} \left\| \Pi_m [ \vc{F}_{k,\ep, \xi}(\vr \vartheta,
\vu) ]
\right\|_{L^{\infty}_x } \\ &\leq c(\Ov{\vr}) \left\| \vc{F}_{k,\ep, \xi } (\vr, \vartheta, \vu) \right\|_{W^{2,2}_x} \aleq
c(\Ov{\vr}, \xi )f_{k,\ep}
\end{split}
\]
using also \eqref{P-1new} and \eqref{eq:2610}.
Consequently, we may use the bounds (\ref{m142*}) and apply the Kolmogorov continuity criterion
to obtain
\begin{equation} \label{wWS224}
\expe{ \| \vr \vu \|_{C^{s}([0,T]; W^{-k,2}(Q)) }^r } \aleq c(r) \ \mbox{for a certain}\ 0 < s(r) < \frac{1}{2}
\end{equation}
provided $r \geq 2$, where $c(r)$ behaves like \eqref{wWS222}.
Apparently, estimates (\ref{wWS222}), (\ref{wWS224}) imply strong (pointwise) compactness of $\vr$ and $\vr \vu$ necessary for passing to the limit in the nonlinear terms. As we shall see below, the same property of the temperature will follow from equation \eqref{apeneq}.

\subsubsection{Asymptotic limit}
\label{s:AL1}

Suppose now that $m$ is fixed. Given a family of approximate solutions $(\vr_R, \vt_R, \vu_R)_{R > 0}$ we let $R \to \infty$.
Unfortunately, the available estimates are considerably weaker than those obtained in Section \ref{wWS1S2}, making the choice of the
appropriate \emph{path space} more delicate. In particular,
we have to use the \emph{weak topologies} that are in general not Polish. Here and hereafter we systematically follow the approach proposed in
\cite[Chapter 2.8]{BFHbook} and consider bounded sequences in Banach spaces together with their norms as a new random variables, applying
the standard Skorokhod theorem. The key result is the following theorem:
\begin{Theorem} \label{RDT4}

Let $(\vc{U}_{0,n})_{n \geq 1}$ be a sequence of
random variables in a Polish space $Y_0$, $(\vc{U}_n)_{n \geq 1}$ a sequence of random variables in $L^1(Q_T; R^M)$, and $(W_n)_{n \geq 1}$ a sequence of cylindrical Wiener processes
defined on a complete probability space $\{ \Omega, \mathfrak{B}, \prst \}$. Suppose that the family of laws of $(\vc{U}_{0,n})_{n \geq 1}$
is tight in $Y_0$. In addition, suppose that
for any $\ep > 0$, there exists $M > 0$ such that
\[
\begin{split}
\prst \left\{ \ \| \vc{U}_n  \|_{L^q(Q_T; R^M)} > M \ \right\} &< \ep \ \mbox{for some}\ q \geq 1;
\\
\prst \left\{ [[ \vc{U}_n ]] > M \ \right\} &< \ep
\end{split}
\]
uniformly for $n = 1,2, \dots$, where $[[ \cdot ]]: W^{-m,2}(Q_T) \to [0, \infty]$, $m > \frac{N+1}{2}$ is a Borel measurable function.
\\
Then there exist subsequences of random variables $(\tilde{\vc{U}}_{0,n(j)} )_{j \geq 1}$
in $Y_0$, $(\tilde{\vc{U}}_{n(j)} )_{j \geq 1}$,\ $\tilde{\vc{U}}_{n(j)} \in L^1(Q_T; R^M)$ and
cylindrical Wiener processes $\tilde{W}_{n(j)}$ on the standard probability space\\
$\Big\{ [0,1], \Ov{ \mathfrak{B}[0,1] }, \mathfrak{L} \Big\}$ enjoying the following properties $\mathfrak{L}-$a.s.:
\[
\begin{split}
\left[ \vc{U}_{0,n(j)}, \vc{U}_{n(j)}, W_{n(j)} \right] &\sim \left[ \tilde{\vc{U}}_{0, n(j)}, \tilde{\vc{U}}_{n(j)}, \tilde{W}_{n(j)} \right] \ \mbox{(equivalence in law)};
\\
\tilde{\vc{U}}_{n(j)} \to \vc{U} \ \mbox{in}\ W^{-m,2}(Q_T; R^M),\
g \left( \tilde{\vc{U}}_{n(j)} \right) &\rightharpoonup^\ast \Ov{g (\vc{U})}
\ \mbox{in}\ L^\infty(Q_T) \ \mbox{for any}\ g \in C_c(R^M);
\\
\tilde{\vc{U}}_{0, n(j)} \to \vc{U}_0 \ \mbox{in}\ Y_0,\ \tilde{W}_{n(j)} &\to W \ \mbox{in}\ C([0,T]; \mathfrak{U}_0);
\\
\sup_{j \geq 1} [[ \tilde{\vc{U}}_{n(j)} ]] &< \infty.
\end{split}
\]
\\
If, in addition $q > 1$, then $\mathfrak{L}-$a.s.
\[
\tilde{\vc{U}}_{n(j)} \to \vc{U} \ \mbox{weakly in}\ L^q(Q_T; R^M),\
f \left( \tilde{\vc{U}}_{n(j)} \right) \to \Ov{f (\vc{U})}
\ \mbox{weakly in}\ L^r(Q_T)
\]
for any $f \in C(R^M)$ such that
\[
| f(\vc{v}) | \leq c \left(1 + |\vc{v}|^{s} \right),\ 1 \leq s < q,\ r = \frac{q}{s} > 1.
\]

\end{Theorem}

\begin{Remark} \label{Nrcom1}

Here and hereafter the symbol $\Ov{f(\vc{U})}$ denotes a weak $L^1-$limit of a sequence $(f(\vc{U}_n))_{n \geq 1}$. The existence of such a limit 
for any $f$ with appropriate growth implies the existence of a \emph{Young measure} $\{ \nu_{t,x} \}_{(t,x) \in Q_T}$ associated to the sequence 
$(f(\vc{U}_n))_{n \geq 1}$, In particular, 
\[
\left< \nu_{t,x} ; f \right> = \Ov{f(\vc{U})}(t,x) \ \mbox{for a.a.}\ (t,x) \in Q_T,\ f \in C_c(R^M),
\] 
cf. Pedregal \cite{Pe}.

\end{Remark}

\begin{Remark} \label{Nrcom}

Note that $L^1(Q_T) \hookrightarrow W^{-m,2}(Q_T)$ as soon as $m > (N+1)/2$. In applications, $\vc{U}$ is a vector of random variables (unknowns) like
$(\vr, \vt, \vu, \Grad \vt, \Grad \vu)$ and $[[ \cdot ]]$ represents the sum of available uniform bounds, like
\[
[[ (\vr, \vt, \vu, \Grad \vt, \Grad \vu) ]] = \| \vr \|_{L^\infty(0,T; L^\beta)} + \|\Grad \vt \|_{L^2} + \|\Grad \vu \|_{L^2} \dots
\]

\end{Remark}

{In addition, we need an abstract result on changing law in a system of stochastic PDE's. Specifically, evoking the situation considered
in \cite[Chapter 2.9]{BFHbook}, we consider an abstract stochastic PDE}
\[
D(\vc U) (\tau) - D_0 + \int_0^\tau \Div \vc{F}(\vc U) \dt = \int_0^\tau \vc{G}(\vc U) \, \D W,
\]
or, in the weak form,
\begin{align} \label{RD8}
\begin{aligned}
\int_0^T &\left[ \partial_t \psi \left< D(\vc U),\varphi \right> + \psi \left< \vc{F}(\vc U),\Grad\varphi \right> \right] \dt\\&\qquad
+ \int_0^T \psi \left< \vc{G}(\vc U),\varphi \right>  \D W + \psi(0) \left< D_0,\varphi \right> = 0
\end{aligned}
\end{align}
for any $\varphi \in \DC(Q)$, $\psi \in C^\infty_c[0,T)$, where
\begin{align*}
D = D(x, \vc U),\ \vc{F} = \vc{F}(x,\vc U),\ \vc{G}_k = \vc{G}_k (x,\vc U), \ x \in \tor,\ \vc U \in \R^M,
\end{align*}
are nonlinear superposition operators given by Carath\'eodory functions.
{We report the following result \cite[Chapter 2, Theorem 2.9.1]{BFHbook}:}

\begin{Theorem} \label{RDT3}

Let $\vc U \in L^1([0,T]; L^1(Q))$, $D_0 \in L^1(Q)$ be a random variables such that
\[
D(\vc U) , \, \vc{F}(\vc U) ,\, \vc{G}_k(\vc U) \in L^1([0,T]; L^1(Q))\ \pas,
\]
\[
\int_0^T \sum_{k=1}^\infty \left| \left< \vc{G}_k (\vc U),\varphi \right> \right|^2 {\rm d}t < \infty\ \pas
\ \mbox{for any}\ \varphi \in \DC (Q).
\]
\\
Let
$W = \left( W_k \right)_{k \geq 0}$ be a cylindrical Wiener process. Suppose that
the filtration
\[
\mathfrak{F}_t = \sigma \big( \sigma_t[\vc U] \cup \cup_{k=1}^\infty \sigma_t [W_k] \big),\ t \geq 0,
\]
is non-anticipative with respect to $W$.
Let $\tilde{\vc U}$, $\tilde{D}_0$ be another pair of random variables and $\tilde{W}$ another stochastic process such their joint laws coincide, namely,
\[
[D_0, \vc U, W] \sim [\tilde{D}_0, \tilde{\vc U}, \tilde{W} ].
\]
\\
Then $\tilde W$ is a cylindrical Wiener process, the filtration
\[
\tilde{\mathfrak{F}}_t = \sigma \left( \sigma_t[\tilde{\vc U}] \cup \cup_{k=1}^\infty \sigma_t [\tilde{W}_k] \right) ,\ t \geq 0,
\]
is non-anticipative with respect to $\tilde{W}$, and
\[
\begin{split}
&\left[ \int_0^T \left[ \partial_t \psi \left< D(\vc U),\varphi \right> + \psi \left< \vc{F}(\vc U),\Grad\varphi \right> \right] \dt
\right.\\
&\quad\quad \left.+ \int_0^T \psi \left< \vc{G}(\vc U),\varphi \right>  \D W + \psi(0) \left< D_0,\varphi \right> \right] \sim  \left[ \int_0^T \left[ \partial_t \psi \left< D(\tilde{\vc U}),\varphi \right> + \psi \left< \vc{F}(\tilde{\vc U}),\Grad\varphi \right> \right] \dt\right.
\\
&\quad\quad \left.+ \int_0^T \psi \left< \vc{G}(\tilde{\vc U});\varphi \right>  \D \tilde{W} + \psi(0) \left< \tilde{D}_0),\varphi \right> \right]
\ \mbox{(equivalence in law in $R$)}
\end{split}
\]
for any deterministic $\psi \in \DC([0,T))$, $\phi \in C^\infty(\tor)$.

\end{Theorem}

We apply Theorems \ref{RDT4}, \ref{RDT3} to the sequences $(\vc{U}_{0,R})$, $(\vc{U}_R)_{R > 0}$,
\[
\vc{U}_{0,R} = \Big[\vr_{0,R}, \vt_{0,R}, \vu_{0,R}],\
\vc{U}_R = \Big[\vr_R, \vt_R, \vu_R, \Grad \vu_r, \Grad \vr_R, \Grad \vt_R \Big],
\]
with the associated Wiener processes $W_R$, the existence of which is guaranteed by Theorem \ref{wWP1}.
Here, the initial data $(\vr_{0,R}, \vt_{0,R}, \vu_{0,R})$ are considered in the
space $Y_0 \in C^{2 + \nu} \times W^{1,2} \cap C \times H_m$. The functional
$[[ \cdot ]]$ is taken as the sum of all norm appearing in the estimates (\ref{m142*}), (\ref{wWS28}), (\ref{Nm142*}), and
(\ref{wWS224}) together with the norm of the initial data, specifically,
\[
\begin{split}
[[ \vc{U}_R ]] &= \sup_{t \in (0,T)} \intQ{ \Big( \frac{1}{2}
\varrho_R |\vc u_R|^2 + H_{\delta,\Theta}(\varrho_R, \vartheta_R) +
{\delta}(\frac{ \varrho^{\beta}_R}{\beta - 1} + \varrho^2_R) \Big)
 } + \dots \\
&+  \| \vr_R \vu_R \|_{C^{s}([0,T]; W^{-k,2}(Q))}.
\end{split}
\]
\\
In accordance with Theorems \ref{RDT4} and \ref{RDT3},
we obtain a new family of random variables $(\tilde{\vr}_R, \tilde{\vt}_R, \tilde{\vu}_R)$, together with the Wiener processes
$\tilde{W}_R$, defined on the standard standard probability space $([0,1], \Ov{\mathfrak{B}[0,1]}, \mathfrak{L})$, with the associated
right-continuous complete filtration  $(\tilde{\mathfrak{F}}^R_t)_{t \geq 0}$ such that:
\begin{itemize}
\item
$\left[ \left( [0,1], \Ov{\mathfrak{B}[0,1]}, (\tilde{\mathfrak{F}}^M_t)_{t \geq 0} ,  \mathfrak{L},\right), \tilde{\vr}_R, \tilde{\vt}_R, \tilde{\vu}_R,
\tilde{W}_R \right]$
is a weak martingale solution of problem \eqref{wW}, \eqref{BCba} in the sense of Definition \ref{WD1}, with the initial law $\Lambda_R$
specified in Section \ref{wWS2S2}. In addition, the triple
$[\tilde{\vr}_R, \tilde{\vt}_R, \tilde{\vu}_R]$ satisfies the total energy balance \eqref{EI20a0} and the entropy equation \eqref{apeneq};
\item
in accordance with (\ref{nGL1}), the initial data satisfy
\[
\begin{split}
\tilde{\vr}_{0,R} &\to \tilde{\vr}_0 \ \mbox{in} \ C^{2 + \nu}(\Ov{Q}),
\ \tilde{\vt}_{0,R} \to \tilde{\vt}_0 \ \mbox{in}\ W^{1,2} \cap C(\Ov{Q}), \ {\vr}_0, {\vt}_0 > 0,\\
\tilde{\vu}_{0,R} &\to \tilde{\vu}_0 \ \mbox{in}\ H_m \ \mathfrak{L}-\mbox{a.s.};
\end{split}
\]
\item the functionals
bounded in expectations in (\ref{m142*}), (\ref{wWS28}), (\ref{Nm142*})
and (\ref{wWS224}) are bounded for
$(\tilde{\vr}_R, \tilde{\vt}_R, \tilde{\vu}_R)_{R > 0}$
uniformly for $R \to \infty$ $\mathfrak{L}-$a.s.;
\item we have
\begin{equation} \label{wWS229}
\begin{split}
\tilde{\vr}_R &\to \tilde{\vr} \ \mbox{weakly in}\ W^{1,p}(0,T; L^p(Q)) \cap L^p (0,T; W^{2,p}(Q)) \ \mbox{for some}\ p > 1 \ \mathfrak{L}-\mbox{a.s.},\\
\tilde{\vt}_R &\to \tilde\vt \ \mbox{weakly in}\ L^2(0,T; W^{1,2}(Q)) \ \mbox{and weakly-(*) in} \
L^\infty (0,T; L^4(Q)) \ \mathfrak{L}-\mbox{a.s.},\\
\tilde{\vu}_R &\to \tilde\vu \ \mbox{weakly-(*) in}\ L^2(0,T; W^{2, \infty }(Q;R^3))\ \mathfrak{L}-\mbox{a.s.},\\
\tilde{\vr}_R \tilde{\vu}_R &\to \tilde\vr \tilde\vu \ \mbox{in}\ C([0,T]; W^{-k,2}(Q; R^3))\ \mathfrak{L}-\mbox{a.s.}, \ k > \frac{3}{2};
\end{split}
\end{equation}
\item
the sequence $\left( \tilde{\vr}_R, \tilde{\vt}_R, \tilde{\vu}_R, \Grad \tilde{\vu}_r, \Grad \tilde{\vr}_R, \Grad \tilde{\vt}_R \right)_{R > 0}$
generates a Young measure;
\item we have
\[
\tilde{W}_R \to \tilde{W} \ \mbox{in}\ C([0,T]; \mathfrak{U}_0 )\ \mathfrak L \mbox{-a.s.}
\]

\end{itemize}

Now, we observe that (\ref{wWS229}) yields also \emph{strong} (in $L^1$) convergence of $(\tilde{\vr}_R)_{R > 0}$. As for the velocity,
we have,
\[
\Div \tilde{\vu}_R \ \mbox{bounded in}\ L^1(0,T; L^\infty(\tor) ) \ \mbox{uniformly in}\ R \ \mathfrak{L}-\mbox{a.s.}
\]
Consequently, evoking the well known estimate for the parabolic equation (\ref{wW3})
\[
\tilde{\vr}_R (\tau,x) \geq \min_{\tor} \tilde{\vr}_0 \exp \left(- \int_0^\tau \| \Div \tilde{\vu}_R \|_{L^\infty_x} \dt \right)
\]
we deduce that $\tilde{\vr}_R$ is bounded below away from zero in terms of the initial data. Consequently, relation (\ref{wWS229}) implies
\[
\tilde{\vu}_R \to \vu \ \mbox{in}\ C([0,T]; W^{-k,2}(Q; R^3));
\]
whence, in view of equivalence of norm on $H_m$,
\begin{equation} \label{NWS229}
\tilde{\vu}_R \to \tilde \vu \ \mbox{in}\ C([0,T]; W^{2,\infty}(Q; R^3)) \ \mathfrak{L}-\mbox{a.s.}
\end{equation}
\\
The strong convergence of the temperature,
\begin{equation} \label{NWS229bis}
\tilde{\vt}_R \to \vt \ \mbox{in, say,}\ L^2((0,T) \times Q) \ \mathfrak{L}-\mbox{a.s.},
\end{equation}
can be deduced from the equation (\ref{wW4}), exactly as in the deterministic case, see \cite[Chapter 3, Section 3.5.3]{F}. Here the proof is particularly simple and may be carried over by means of a variant of Lions--Aubin lemma. We will give a detailed proof under more delicate circumstances in Section
\ref{sec:vanishingviscosity} below. \\
Finally,
to perform the limit in the stochastic integral, we use the following result proved in \cite[Chapter 2, Lemma 2.6.6]{BFHbook}.

\begin{Lemma} \label{RDL4}
Let $( \Omega, \mathfrak{F}, \prst )$
be a complete probability space, and $\ell \geq 0$. For $n\in\mn,$ let $ W_n$ be an $(\mathfrak{F}^n_t)$-cylindrical Wiener process and
let $ \vc{G}_n$ be an  $( \mathfrak{F}^n_t
)$-progressively measurable stochastic process such that $\bfG_n\in L^2(0,T; L_2(\mathfrak{U};W^{-\ell,2}(Q)))$ a.s.
Suppose that
\[
W_n \to W \ \mbox{in}\ C([0,T]; \mathfrak{U}_0) \ \mbox{in probability},
\]
\[
\vc{G}_n \to \vc{G} \ \mbox{in}\ L^2(0,T; L^2(\mathfrak{U}; W^{-\ell,2}(Q)) \  \mbox{in probability},
\]
where $W =\sum_{k=1}^\8 e_k W_k$.
Let $( \mathfrak{F}_t )_{t \geq 0}$ be the filtration given as
\[
\mathfrak{F}_t = \sigma \big( \cup_{k = 1}^ \infty \sigma_t[ \bfG e_k ] \cup \sigma_t [W_k] \big).
\]
\\
Then, after a possible change on a set of zero measure in $\Omega \times (0,T)$,  $\vc{G}$ is
$( \mathfrak{F}_t )$-progressively measurable, and
\[
\int_0^\cdot \vc{G}_n \, \D W_n \to \int_0^\cdot \vc{G} \, \D W \ \mbox{in}\ L^2(0,T; W^{-\ell,2}(\TN)) \ \mbox{in probability.}
\]

\end{Lemma}

With Lemma \ref{RDL4} and the compactness stated in (\ref{wWS229}), (\ref{NWS229}), and (\ref{NWS229bis}) at hand, it is not difficult to pass to the
limit in the equations (\ref{wW3}--\ref{m119}) to obtain the following system:
\begin{subequations}\label{NwW}
\begin{align} \label{NwW3}
\D \vr &+ \Div (\vr \vu )\ \dt  = \ep \Del \vr \ \dt, \ \Grad \vr \cdot \vc{n}|_{\partial Q} = 0 ,\\
\D \Pi_m [\vr \vu]
&+ \Pi_m [\Div (\vr \vu \otimes \vu) ] \dt
+  \Pi_m \Big[ \Grad \left(  p (\vr,\vt) + \delta ({\vr^2} + \vr^\beta) \right) \Big] \dt\nonumber\\
&= \Pi_m \Big[\ep \Del (\vr \vu) +  \Div \mathbb{S} (\Grad \vu) + { \frac{1}{m}  \vu } \Big] \dt \nonumber\\&+
\Pi_m \left[ \vr \Pi_m \Big[ \vc{F}_{\ep, \xi}(\vr, \vt, \vc{u} )  \Big]  \right] \ \D \W\label{NwW4}\\
\nonumber
\dd(\varrho e_\delta(\varrho, \vartheta)) &+
\big[\Div (\varrho e_\delta(\varrho, \vartheta) \vc u) - \Div ( \kappa_\delta (\vt) \Grad \vt ) \big]\dt \\
&\nonumber = \Big[\tens{S}(\vartheta,\Grad\vc u): \Grad \vc u -
p(\varrho, \vartheta) \Div \vc u \Big] \dt \\ &+ \Big[ \ep
\delta (\beta \varrho^{\beta - 2} + 2) |\Grad \varrho |^2  +
\delta \frac{1}{{\vartheta}^2} - \ep \vartheta^5 \Big]\dt + { \ep \vr |\Grad \vu|^2 \dt },\
\Grad \vt \cdot \vc{n}|_{\partial Q} = 0. \label{Nm119}
\end{align}
\end{subequations}
\\
Let us summarize the results obtained so far.
\begin{Theorem} \label{wWP1bis}
Let $\beta > 6$.
Let $\Lambda$ be a Borel probability measure  on $C^{2 + \nu}(\Ov{Q}) \times W^{1,2} \cap C (\Ov{Q}) \times H_m$ such that
\[
\begin{split}
\Lambda &\left\{ 0 < \underline{\vr} \leq \intQ{ \vr_0 } \leq \Ov{\vr},\ \Grad \vr_0 \cdot \vc{n}|_{\partial Q} = 0,\ \vr_0,\ \vt_0 > 0 \right\}
= 1,\\ &\int_{C^{2 + \nu} \times W^{1,2} \times H_m }\left[ \| \vr_0 \|_{C^{2 + \nu}_x}
 + \| \vt_0 \|_{W^{1,2}_x \cap C_x } + \| \vu_0 \|_{H_m} \right]^r {\rm d} \Lambda \leq c(r)
 \ \mbox{for any}\ r \geq 1.
\end{split}
\]
\\
Then
the approximate problem \eqref{NwW3}--\eqref{Nm119} admits a martingale solution in the sense of
Definition \ref{WD1} (with the obvious modifications for $\xi,\varepsilon,\delta>0$). In addition, the solution satisfies
\begin{align} \label{NEnB}
\begin{aligned}
- \int_0^T& \partial_t \psi
\bigg(\int_{Q} {\mathcal E_\delta(\varrho,\vartheta, \vu) } \dx\bigg) \dt
 +\int_0^T {\psi} \int_{Q} \Big( \ep \vartheta^5 + \frac{1}{m}|\vu|^2\Big)  \dt\\
= &\,\psi(0) \intQ{ {\mathcal E_\delta(\varrho_0,\vartheta_0, \vu_0 )} }+\int_0^T\int_{Q}\frac{\delta}{\vartheta^2}\psi\dx\dt\\
&+ \frac{1}{2} \int_0^T
\psi \bigg(
\intQ{ \sum_{k \geq 1} \varrho | \Pi_m[{\bf F}_{k,\ep, \xi} (\varrho,\vartheta, {\bf u})] |^2  } \bigg) {\rm d}t\\
&+\int_0^T  \psi\int_{Q}\varrho\Pi_m[{\vc{F}_{\ep, \xi} }(\varrho,\vartheta,\bfu)]\cdot\bfu\,\dd W\dx,
\end{aligned}
\end{align}
with
\[
\mathcal E_\delta(\varrho,\vartheta, \vu)= \frac{1}{2} \varrho | {\bf u} |^2 + \varrho e_{\delta}(\varrho,\vartheta)
+ \delta \left( \frac{\vr^\beta}{\beta - 1} + \vr^2 \right)
\]
for any deterministic smooth test function { $\psi \in \DC[0,T)$}.

\end{Theorem}

\subsection{Limit $m \to \infty$}

Keeping $\ep > 0$, $\delta > 0$ and $\xi > 0$ fixed, our next goal is to let $m \to \infty$ in the approximate system
(\ref{NwW}). With the same initial law $\Lambda$, this can be done in a similar way as in the preceding step modulo certain modifications
due to the lost of regularity of the velocity in the asymptotic limit. In particular, the bound (\ref{wWS28}) and related estimates
on the density are no longer valid
for $m \to \infty$.\\
At this stage, it is also convenient to replace the internal energy equation (\ref{Nm119}) by the entropy balance
\begin{equation} \label{NEntb}
\begin{split}
\D &(\varrho s_\delta(\varrho, \vartheta)) +
\Div (\varrho s_\delta (\varrho, \vartheta) \vc u) \dt - \Div \Big[
\Big( \frac{\kappa(\vartheta)}{\vartheta} + \delta (
\vartheta^{\beta - 1} + \frac{1}{\vartheta^2}) \Big) \Grad
\vartheta \Big] \dt  \\
&=
\frac{1}{\vartheta} \Big[ \tens{S}(\vt, \Grad \vu) : \Grad \vc u + \Big(
\frac{\kappa(\vartheta)}{\vartheta} + \delta ( \vartheta^{\beta -
1} + \frac{1}{\vartheta^2}) \Big) |\Grad \vartheta|^2 + \delta
\frac{1}{{\vartheta}^2} \Big] \dt
\\ &+ \left[ \frac{\ep \delta}{\vartheta} ( \beta \varrho^{\beta - 2} + 2)
|\Grad \varrho|^2
+\ep \frac{\Delta_x \varrho}{\vartheta} \Big( \vartheta
s_\delta(\varrho, \vartheta) - e_\delta(\varrho, \vartheta) -
\frac{p(\varrho, \vartheta)}{\varrho} \Big) \right] \dt \\
 &+\left[- \ep \vartheta^4 + \ep \frac{\vr}{\vt} |\Grad \vu|^2 \right] \dt
\end{split}
\end{equation}
Note carefully that this is possible as $\vt > 0$ and equation \eqref{Nm119} is satisfied a.a. Thus \eqref{NEntb} follows by dividing
\eqref{Nm119} by $\vt$.
\\
Now, exactly as in \cite[Chapter 3, Section 3.5]{F}, we deduce from (\ref{NEntb}) the inequality
\begin{equation} \label{NEntb+}
\begin{split}
\D &(\varrho s_\delta(\varrho, \vartheta)) +
\Div (\varrho s_\delta (\varrho, \vartheta) \vc u) \dt - \Div \Big[
\Big( \frac{\kappa(\vartheta)}{\vartheta} + \delta (
\vartheta^{\beta - 1} + \frac{1}{\vartheta^2}) \Big) \Grad
\vartheta \Big] \dt  \\
&- \ep \Div \left[ \left( \vt s_{M,\delta} (\vr, \vt) - e_{M, \delta} (\vr, \vt) - \frac{p_M (\vr, \vt) }{\vr} \right) \frac{ \Grad \vr}{\vt} \right] \dt \\
& \geq
\frac{1}{\vartheta} \Big[ \tens{S}(\vt, \Grad \vu) : \Grad \vc u + \Big(
\frac{\kappa(\vartheta)}{\vartheta} + \frac{\delta}{2} ( \vartheta^{\beta -
1} + \frac{1}{\vartheta^2}) \Big) |\Grad \vartheta|^2 + \delta
\frac{1}{{\vartheta}^2} \Big] \dt
\\ &+ \left[ \frac{\ep \delta}{2 \vartheta} ( \beta \varrho^{\beta - 2} + 2)
|\Grad \varrho|^2
+\ep \frac{1}{\vr \vt} \frac{\partial p_M}{\partial \vr} (\vr, \vt) |\Grad \vr|^2 - \ep \vartheta^4\right] \dt .
\end{split}
\end{equation}
\\
Consequently, adapting the dissipation inequality (\ref{wWS24}) to the present setting, we deduce that (\ref{wWS28}) must be replaced by
\[
\expe{ \int_0^T \frac{1}{\vt} \mathbb{S} (\vt, \Grad \vu) : \Grad \vu \dt }^r \leq c(r),
\]
which, combined with inequality (\ref{Poinc}), gives rise to
\begin{equation} \label{NwWS28}
\expe{ \int_0^T \| \vu \|^2_{W^{1,2}(Q; R^3)} \dt }^r \leq c(r).
\end{equation}
\\
At this stage, we choose the initial velocity $\vu_0 \in L^2(Q; R^3)$ and adjust the initial law $\Lambda$ to $\vu_0= \vu_{0,m} = \Pi_m \vu_0$.
Accordingly, Theorem \ref{wWP1bis} yields a family $(\vr_m, \vt_m, \vu_m)$ of approximate solutions.
Now, we may repeat step by step the arguments of the preceding part to obtain a martingale solution of the following approximate problem:
\begin{equation} \label{Napde1}
\D \vr + \Div (\vr \vu )\ \dt  = \ep \Del \vr \ \dt ,\ \Grad \vr \cdot \vc{n}|_{\partial Q} = 0,
\end{equation}
\begin{equation} \label{Napde2}
\begin{split}
\D (\vr \vu)
&+ \Div (\vr \vu \otimes \vu)  \dt
+  \Grad \left(  p (\vr,\vt) + \delta ({\vr^\beta} + \vr^2) \right) \dt \\
&= \Big[\ep \Del (\vr \vu) +  \Div \mathbb{S} (\Grad \vu)  \Big] \dt +
\vr \vc{F}_{\ep, \xi} (\vr, \vt, \vc{u} )  \ \D \W ,
\end{split}
\end{equation}
\begin{equation}
\label{Napde3}
\begin{split}
\D &(\varrho s_\delta(\varrho, \vartheta)) +
\Div (\varrho s_\delta (\varrho, \vartheta) \vc u) \dt - \Div \Big[
\Big( \frac{\kappa(\vartheta)}{\vartheta} + \delta (
\vartheta^{\beta - 1} + \frac{1}{\vartheta^2}) \Big) \Grad
\vartheta \Big] \dt  \\
&- \ep \Div \left[ \left( \vt s_{M,\delta} (\vr, \vt) - e_{M, \delta} (\vr, \vt) - \frac{p_M (\vr, \vt) }{\vr} \right) \frac{ \Grad \vr}{\vt} \right] \dt \\
& \geq
\frac{1}{\vartheta} \Big[ \tens{S}(\vt, \Grad \vu) : \Grad \vc u + \Big(
\frac{\kappa(\vartheta)}{\vartheta} + \frac{\delta}{2} ( \vartheta^{\beta -
1} + \frac{1}{\vartheta^2}) \Big) |\Grad \vartheta|^2 + \delta
\frac{1}{{\vartheta}^2} \Big] \dt
\\ &+ \left[ \frac{\ep \delta}{2 \vartheta} ( \beta \varrho^{\beta - 2} + 2)
|\Grad \varrho|^2
+\ep \frac{1}{\vr \vt} \frac{\partial p_M}{\partial \vr} (\vr, \vt) |\Grad \vr|^2 - \ep \vartheta^4\right] \dt, \ \Grad \vt \cdot \vc{n}|_{\partial Q} = 0,
\end{split}
\end{equation}
\begin{equation} \label{Napde4}
\begin{split}
&\D \intQ{\left[ \frac{1}{2} \varrho | {\bf u} |^2 + \varrho e_{\delta}(\varrho,\vartheta)
+ \delta \left( \frac{\vr^\beta}{\beta - 1} + \vr^2 \right) \right]}
+ \left( \int_{Q} \left(  \ep \vt^5 - \frac{\delta}{\vartheta^2} \right) \dx \right) \dt \\
&= \frac{1}{2}
\bigg(
\intQ{ \sum_{k \geq 1} \varrho | {\bf F}_{k,\ep, \xi} (\varrho,\vartheta, {\bf u})|^2  } \bigg) {\rm d}t
+ \left( \int_{Q}\varrho {\vc{F}_{\ep, \xi} }(\varrho,\vartheta,\bfu)\cdot\bfu \dx \right) \D \W.
\end{split}
\end{equation}
Note that the entropy inequality (\ref{Napde3}) as well as the total energy balance (\ref{Napde4}) must be already interpreted in the weak sense as in
Definition \ref{def:sol}. On the other hand, we still recover the strong (a.a. pointwise) convergence of the arguments in the diffusion coefficients
$\vc{F}_{k,\ep,\xi}$ to perform the limit in the stochastic integral.

\subsection{Limit $\xi \to 0$}

Our final goal is to perform the limit $\xi \to 0$. To this end, we choose the cut--off functions $h_\xi$ in \eqref{P-1New} to approach $1$
and the regularizing kernels $\omega_\xi$ to approach the Dirac mass. As the stochastic terms in \eqref{Napde2} and \eqref{Napde4}
do not contain the projection
$\Pi_m$,  we no longer need \eqref{wW2a}. Instead, we simply use
\[
\left\| \vc{F}_{k,\ep, \xi}( \vr, \vt, \vu) \right\|_{L^\infty_x} \aleq \left\| \vc{F}_{k,\ep} ( \vr, \vt, \vu) \right\|_{L^\infty_x}
\leq c(\ep) f_{k,\ep}
\]
in (\ref{wWS25}), (\ref{wWS26}) and all other terms involving the stochastic integral.
\\
Summarizing, we record the following result.

\begin{Theorem} \label{wWP1a}
Let $\beta > 6$.
Let $\Lambda$ be a Borel probability measure  on $C^{2 + \nu}(\Ov{Q})\times W^{1,2}\cap C(\Ov{Q}) \times L^2(Q;R^3)$ such that
\[
\begin{split}
\Lambda &\left\{ 0 < \underline{\vr} \leq \intQ{ \vr_0 } \leq \Ov{\vr},\ \Grad \vr_0 \cdot \vc{n}|_{\partial Q} = 0,\ \vr_0,\ \vt_0 > 0 \right\}
= 1,\\ &\int_{C^{2 + \nu} \times W^{1,2} \times L^2 }\left[ \| \vr_0 \|_{C^{2 + \nu}_x}
 + \| \vt_0 \|_{W^{1,2}_x \cap C_x } + \| \vu_0 \|_{L^2_x} \right]^r {\rm d} \Lambda \leq c(r)
 \ \mbox{for any}\ r \geq 1.
\end{split}
\]
\\
Then
the approximate problem
\begin{equation} \label{Npde1}
\D \vr + \Div (\vr \vu )\ \dt  = \ep \Del \vr \ \dt ,\ \Grad \vr \cdot \vc{n}|_{\partial Q} = 0,
\end{equation}
\begin{equation} \label{Npde2}
\begin{split}
\D (\vr \vu)
&+ \Div (\vr \vu \otimes \vu)  \dt
+  \Grad \left(  p (\vr,\vt) + \delta ({\vr^\beta} + \vr^2) \right) \dt \\
&= \Big[\ep \Del (\vr \vu) +  \Div \mathbb{S} (\Grad \vu)  \Big] \dt +
\vr \vc{F}_{\ep} (\vr, \vt, \vc{u} )  \ \D \W ,
\end{split}
\end{equation}
\begin{equation}
\label{Npde3}
\begin{split}
\D &(\varrho s_\delta(\varrho, \vartheta)) +
\Div (\varrho s_\delta (\varrho, \vartheta) \vc u) \dt - \Div \Big[
\Big( \frac{\kappa(\vartheta)}{\vartheta} + \delta (
\vartheta^{\beta - 1} + \frac{1}{\vartheta^2}) \Big) \Grad
\vartheta \Big] \dt  \\
&- \ep \Div \left[ \left( \vt s_{M,\delta} (\vr, \vt) - e_{M, \delta} (\vr, \vt) - \frac{p_M (\vr, \vt) }{\vr} \right) \frac{ \Grad \vr}{\vt} \right] \dt \\
& \geq
\frac{1}{\vartheta} \Big[ \tens{S}(\vt, \Grad \vu) : \Grad \vc u + \Big(
\frac{\kappa(\vartheta)}{\vartheta} + \frac{\delta}{2} ( \vartheta^{\beta -
1} + \frac{1}{\vartheta^2}) \Big) |\Grad \vartheta|^2 + \delta
\frac{1}{{\vartheta}^2} \Big] \dt
\\ &+ \left[ \frac{\ep \delta}{2 \vartheta} ( \beta \varrho^{\beta - 2} + 2)
|\Grad \varrho|^2
+\ep \frac{1}{\vr \vt} \frac{\partial p_M}{\partial \vr} (\vr, \vt) |\Grad \vr|^2 - \ep \vartheta^4\right] \dt, \ \Grad \vt \cdot \vc{n}|_{\partial Q} = 0,
\end{split}
\end{equation}
\begin{equation} \label{Npde4}
\begin{split}
&\D \intO{\left[ \frac{1}{2} \varrho | {\bf u} |^2 + \varrho e_{\delta}(\varrho,\vartheta)
+ \delta \left( \frac{\vr^\beta}{\beta - 1} + \vr^2 \right) \right]}
+ \left( \int_{\mt} \left(  \ep \vt^5 - \frac{\delta}{\vartheta^2} \right) \dx \right) \dt \\
&= \frac{1}{2}
\bigg(
\intQ{ \sum_{k \geq 1} \varrho | {\bf F}_{k,\ep} (\varrho,\vartheta, {\bf u})|^2  } \bigg) {\rm d}t
+ \left( \int_{\mt}\varrho {\vc{F}_{\ep} }(\varrho,\vartheta,\bfu)\cdot\bfu \dx \right) \D \W.
\end{split}
\end{equation}
admits a martingale solution in the sense of
Definition \ref{def:sol} (with the obvious modification for $\varepsilon,\delta>0$).

\end{Theorem}

\section{The vanishing viscosity limit}
\label{sec:vanishingviscosity}

Our ultimate goal is to perform successively the asymptotic limits $\ep \to 0$ and $\delta \to 0$ in the approximate system
(\ref{Npde1}--\ref{Npde4}). It is worth noting that the relations (\ref{Napde1}) and (\ref{Napde3}) are \emph{deterministic} and the same as in
\cite[Sections 3.6 - 3.7]{F}, where a similar limit in the absence of stochastic forcing is performed. Thus, at least in (\ref{Napde1}), (\ref{Napde3}),
the limit process is exactly the same as in \cite[Chapter 3]{F} as long as suitable uniform bounds are established. Accordingly, we adopt the following strategy:

\begin{itemize}

\item Making use of an appropriate form of the total dissipation balance (\ref{wWS24}) we derive the energy bounds.

\item We derive all other estimates, in particular for the pressure and the velocity, that require stochastic averaging.

\item Changing the probability space we recover (weak) compactness pointwise with respect to the random parameter. Accordingly, we perform the limit passage,
in which the equation of continuity and the entropy balance are handled in the same way as in \cite[Chapter 3]{F}.

\item We pass to the limit in the stochastic integral using Lemma \ref{RDL4}.

\end{itemize}

We aim to perform the limit $\ep \to 0$ extending the validity of Theorem \ref{wWP1a} to the following system.
\begin{equation} \label{Napd1}
\D \vr + \Div (\vr \vu )\ \dt  = 0 ,
\end{equation}
\begin{equation} \label{Napd2}
\begin{split}
\D (\vr \vu)
+ \Div (\vr \vu \otimes \vu)  \dt
+  \Grad \left(  p (\vr,\vt) + \delta ({\vr^2} + \vr^\beta) \right) \dt
= \Div \mathbb{S} (\Grad \vu)   \dt +
\vr \vc{F}(\vr, \vt, \vc{u} )  \ \D \W ,
\end{split}
\end{equation}
\begin{equation}
\label{Napd3}
\begin{split}
\D &(\varrho s_\delta(\varrho, \vartheta)) +
\Div (\varrho s_\delta (\varrho, \vartheta) \vc u) \dt - \Div \Big[
\Big( \frac{\kappa(\vartheta)}{\vartheta} + \delta (
\vartheta^{\beta - 1} + \frac{1}{\vartheta^2}) \Big) \Grad
\vartheta \Big] \dt  \\
& \geq
\frac{1}{\vartheta} \Big[ \tens{S}(\vt, \Grad \vu) : \Grad \vc u + \Big(
\frac{\kappa(\vartheta)}{\vartheta} + \frac{\delta}{2} ( \vartheta^{\beta -
1} + \frac{1}{\vartheta^2}) \Big) |\Grad \vartheta|^2 + \delta
\frac{1}{{\vartheta}^2} \Big] \dt,
\end{split}
\end{equation}
\begin{equation} \label{Napd4}
\begin{split}
&\D \intO{\left[ \frac{1}{2} \varrho | {\bf u} |^2 + \varrho e_{\delta}(\varrho,\vartheta)
+ \delta \left( \frac{\vr^\beta}{\beta - 1} + \vr^2 \right) \right]}
- \left( \int_{\mt} \frac{\delta}{\vartheta^2}  \dx \right) \dt \\
&= \frac{1}{2}
\bigg(
\intTor{ \sum_{k \geq 1} \varrho | {\bf F}_{k} (\varrho,\vartheta, {\bf u})|^2  } \bigg) {\rm d}t
+ \left( \int_{\mt}\varrho {\vc{F} }(\varrho,\vartheta,\bfu)\cdot\bfu \dx \right) \D \W.
\end{split}
\end{equation}
\\
The main result of this section is the following.
\begin{Theorem} \label{wWT3}
Let $\beta > 6$.
Let $\Lambda$ be a Borel probability measure  on $C^{2 + \nu}(\Ov{Q})\times W^{1,2} \cap C(\Ov{Q}) \times L^2(Q;R^3)$ such that
\[
\begin{split}
\Lambda &\left\{ 0< \underline{\vr} \leq \intO{ \vr_0 } \leq \Ov{\vr} ,\ \Grad \vr_0 \cdot \vc{n}|_{\partial Q} = 0,\ \vr_0, \vt_0 > 0 \right\} = 1,\\
&\int_{C^{2 + \nu} \times W^{1,2} \times L^2} \left[ \| \vr_0 \|_{C^{2 + \nu}_x} + \| \vt_0 \|_{W^{1,2}_x \cap C_x} + \| \vu_0 \|_{L^2_x} \right]^r {\rm d}
\Lambda \leq c(r) \ \mbox{for all}\ r \geq 1.
\end{split}
\]
\\
Then
the approximate problem \eqref{Napd1}--\eqref{Napd4} admits a martingale solution in the sense of
Definition \ref{def:sol}.

\end{Theorem}

The proof of Theorem \ref{wWT3} requires the full strength of the method developed in the context of the deterministic Navier--Stokes system.
Possible oscillations of the density are ruled out thanks to the weak compactness of a quantity called \emph{effective viscous flux},
\[
\left(\frac{4}{3} \mu(\vt) + \eta(\vt) \right) \Div \vu - p(\vr, \vt),
\]
where $\mu$, $\eta$  are the viscosity coefficients.

\subsection{Uniform energy bounds}
\label{wWS3S1}

Using (\ref{Npde3}), (\ref{Npde4}), or rather their variational formulation via the appropriate versions of (\ref{m217*final}) and (\ref{EI20final}) containing the artificial pressure, we deduce the total dissipation balance
\begin{align} \label{NTDB}
\begin{aligned}
&
 \intQ{ \Big[ \frac{1}{2} \varrho | {\bf u} |^2 + H_{\delta,\Theta}(\varrho,\vartheta)+\frac{\delta}{\beta-1}\varrho^\beta
+ { \delta \vr^2}  \Big] }
+\Theta \int_0^\tau\int_{Q}\sigma_{\varepsilon,\delta}\dx
+ \int_0^\tau\int_{Q} \ep \vartheta^5   \dt\\
&\leq \,\intQ{ \Big[ \frac{1}{2} \vr_0 |\vu_0|^2  + H_{\delta,\Theta}(\varrho_0,\vartheta_0)+\frac{\delta}{\beta-1}\varrho_0^\beta + {\delta \vr_0^2} \Big] }\\
&+\int_0^\tau\int_{Q}\bigg(\frac{\delta}{\vartheta^2}+\varepsilon\Theta\vartheta^4\bigg)\dx\dt
+\int_0^\tau \int_{Q}\varrho {\vc F}_\ep(\varrho,\vartheta,\bfu)\cdot\bfu\, \dx \ \dd W\\&+ \frac{1}{2} \int_0^\tau
\bigg(
\intQ{ \sum_{k \geq 1} \varrho | {\bf F}_{k,\ep} (\varrho,\vartheta,{\bf u}) |^2  } \bigg) {\rm d}t
\end{aligned}
\end{align}
for any $0 \leq \tau \leq T$ $\pas$,
where
\begin{align*}
\sigma_{\varepsilon,\delta}&=\frac{1}{\vartheta}\Big[\mathbb S(\vartheta,\nabla\bfu):\nabla\bfu+\frac{\kappa(\vartheta)}{\vartheta}|\nabla\vartheta|^2+\frac{\delta}{2}\Big(\varrho^{\beta-1}+\frac{1}{\vartheta^2}\Big)|\nabla\vartheta|^2+\delta\frac{1}{\vartheta^2}\Big]\\
&+\frac{\varepsilon\delta}{2\vartheta}{\left( \beta\varrho^{\beta-2} + 2 \right)} |\nabla\varrho|^2+\varepsilon\frac{\partial p_M}{\partial\varrho}(\varrho,\vartheta)\frac{|\nabla\varrho|^2}{\varrho\vartheta}.
\end{align*}
\\
Thanks to hypothesis (\ref{P-1}) we get
\[
\int_{Q} \vr \left|\vc{F}_{k,\ep} (\vr, \vartheta ,\vu)\right||\vc{u}| \dx \aleq f_k \intO{ \vr (1 + |\vu|^2) },\
\intQ{ \vr \left|  \vc{F}_{k,\ep} (\vr, \vartheta ,\vu) \right|^2 } \aleq f_k^2 \intO{ \vr (1 + |\vu|^2) }.
\]
Thus we may pass to expectations in (\ref{NTDB})
and apply a Gronwall-type argument to  deduce the following bounds, cf. \cite[Chapter 3, Section 3.6.1]{F}:
\begin{equation} \label{Nbv1}
\expe{ \left| {\rm ess} \sup_{t \in (0,T)} \intQ{ \Big[ \frac{1}{2} \varrho | {\bf u} |^2 + H_{\delta,\Theta}(\varrho,\vartheta)+\frac{\delta}{\beta-1}\varrho^\beta
+ { \delta \vr^2}  \Big] } \right|^r } \leq c(r, \Lambda),
\end{equation}
in particular
\begin{equation} \label{Nbv2}
\begin{split}
\expe{ \left| \sup_{t \in (0,T)} \| \vr \|^\beta_{L^\beta(Q)} \right|^r } &\leq c(r, \Lambda),\\
\expe{ \left| \sup_{t \in (0,T)} \| \vr \vu \|_{L^{\frac{2 \beta}{\beta + 1}}(Q; R^3)} \right|^r } &\leq c(r, \Lambda), \\
\expe{ \left| {\rm ess} \sup_{t \in (0,T)} \| \vt \|^4_{L^4(Q)} \right|^r } &\leq c(r, \Lambda).
\end{split}
\end{equation}
Moreover, boundedness of the entropy production rate
\begin{equation} \label{Nbv3}
\expe{ \| \sigma_{\ep, \delta} \|_{L^1((0,T) \times Q}^r } \leq c(r, \Lambda)
\end{equation}
gives rise to
\begin{equation} \label{Nbv4}
\expe{ \| \Grad \vu \|^r_{L^2(0,T) \times Q; R^{3 \times 3}) } } + \expe{ \| \Grad \vt \|^r_{L^2(0,T) \times Q; R^3) } } \leq c(r, \Lambda);
\end{equation}
whence, by Poincare's inequality and (\ref{Nbv2}),
\begin{equation} \label{Nbv5}
\expe{ \| \vu \|^r_{L^2(0,T; W^{1,2}(Q; R^3) ) } } + \expe{ \| \vt \|^r_{L^2(0,T; W^{1,2}(Q)) } } \leq c(r, {\Lambda}).
\end{equation}
Note that we keep the initial law $\Lambda$ the same as in the previous section.
\\
Finally, we deduce from the equation of continuity (\ref{NwW3}) that
\begin{equation} \label{Nbv6}
\intO{ \vr(t, \cdot) } = \intO{ \vr_0 },\
\expe{ \| \sqrt{\ep} \Grad \vr \|_{L^2((0,T) \times \tor)}^r } \leq c(r, \Lambda).
\end{equation}
Note that all estimates are independent of $\ep$.
\\
The above bounds are not strong enough to control the pressure term proportional to $\vr^\beta$ that is for the current stage bounded only in the non-reflexive space
$L^1_x$. The adequate estimates will be derived in the next section.

\subsection{Pressure estimates}
\label{wWS3S2}

Following \cite[Chapter 2, Section 2.2.5]{F} we introduce the Bogovskii operator $\mathcal{B}$ enjoying the following properties:
\[
\mathcal{B}: L^q_0 \equiv \left\{ f \in L^q(R^3) \ \Big| \ (f)_{Q} = 0 \right\} \to W^{1,q}_0(Q; R^3), \ \Div \mathcal{B}[f] = f
\ \mbox{in} \ q, \ 1 < q < \infty,
\]
\[
\left\| \mathcal{B} \left[ \Div \vc{g} \right] \right\|_{L^r(Q)} \aleq \| \vc{g} \|_{L^r(Q)}
\ \mbox{whenever} \Div \vc{g} \in L^q_0, \ 1 < q,r < \infty.
\]
\\
The idea, borrowed again from \cite[Chapter 2, Section 2.2.5]{F}, is to use the quantity
\[
\mathcal{B} \left[ \vr - (\vr)_{Q} \right]
\]
as a  test function in the variational formulation of the
momentum balance (\ref{Napde2}) (cf. formula (\ref{wWS238final})). Note that this is not straightforward as the legal test functions allowed
have the form $\psi(t) \bfphi(x)$, where both $\psi$ and $\bfphi$ are smooth and deterministic.
Nevertheless, such a procedure can be rigorously justified by the application of  a suitable version of the generalized It\^o formula to the functional
$$
(\rho,\bfq)\mapsto\int_{\mt}\bfq\cdot \Del^{-1} \Grad \vr\dx
$$
(see \cite[Sec. 5]{BrHo}).
We rewrite (\ref{Napde2}) in the differential form
\begin{equation} \label{kolmo}
\begin{split}
\D &\intQ{ \vr \vu \cdot \bfphi } - \intQ{ \Big[ \vr \vu \otimes \vu : \Grad \bfphi  +  p_\delta(\vr,\vt)
 \Div \bfphi  \Big] }\dt
\\
&=- \intQ{ \Big[ \mathbb{S}(\vt,\Grad \vu) : \Grad \bfphi  + \ep \Grad(\vr \vu) : \Grad \bfphi  \Big] } \dt
+ \intQ{ \vr  {\vc{F}_\ep} (\vr, \vt, \vu) \cdot \bfphi } \ \D  \W .
\end{split}
\end{equation}
Seeing that
\[
\D \left( \mathcal{B} \left[ \vr - ( \vr )_{Q} \right]\right) = - \mathcal{B} \Div (\vr \vu) \dt +
\ep \mathcal{B} [\Delta \vr]  \dt,
\]
we obtain
\begin{align}
&\int_0^T \intQ{ p_\delta(\vr,\vt)\left[ \vr - (\vr )_{Q} \right]    }\dt\nonumber\\
&= \left[ \intQ{ \vr \vu \cdot \mathcal{B} \left[  \vr - ( \vr )_{Q}\right] } \right]_{t = 0}^{t=T}  - \int_0^T \intQ{ \vr \vu \otimes \vu : \Grad \mathcal{B} \left[  \vr - ( \vr )_{Q}\right]} \dt \nonumber\\
&\qquad+ \int_0^T \intQ{ \mathbb{S}(\vt,\Grad \vu) : \Grad \mathcal{B} \left[  \vr - ( \vr )_{Q}\right] } \dt
+\int_0^T \intQ{ \vr \vu \cdot  \mathcal{B}[ \Div (\vr \vu) ] } \dt \nonumber\\
&\qquad+ \ep \intQ{\left[ \Grad (\vr\vu) \cdot \Grad \mathcal{B}\left[ \vr - (\vr)_Q \right]  - \vr \vu \cdot \mathcal{B}(\Delta \vr) \right] }\dt
\nonumber
\\&\qquad - \int_0^T \intQ{ \vr  {\vc {F}_\ep} (\vr,\vt, \vu) \cdot \mathcal{B}\left[ \vr - (\vr)_Q \right] } \ \D \W .\label{wWS310}
\end{align}
\\
As shown in \cite[Chapter 3, Section 3.6.3]{F}, all deterministic integrals on the right--hand side of \eqref{wWS310} are controlled
by the energy bounds. This means, in the present setting, their expected values are controlled by the bounds (\ref{Nbv2}--\ref{Nbv6}).
\\
As for the stochastic integral, we apply the Burkholder--Davis--Gundy inequality:
\[
\begin{split}
&\expe{ \left| \int_0^\tau \intQ{ \vr  {\vc {F}_\ep} (\vr, \vt,\vu) \cdot \mathcal{B} \left[ \vr - (\vr)_Q \right]  } \ \D \W \right|^r } \\
&\leq \expe{ \sup_{t\in[0,\tau]}\left| \int_0^t \intQ{ \vr  {\vc {F}_\ep} (\vr, \vt,\vu) \cdot \mathcal{B} \left[ \vr - (\vr)_Q \right] } \ \D \W \right|^r } \\
&\aleq \expe{ \int_0^\tau \sum_{k =1}^\8 \left|\, \intQ{\vr  {\bf {F}}_{k,\ep} (\vr,\vt, \vu) \cdot \mathcal{B} \left[ \vr - (\vr)_Q \right] }\right|^2\dt }^{r/2},
\end{split}
\]
where, due to \eqref{P-1} and the properties of $\mathcal{B}$ and $\beta>3$,
\[
\begin{split}
\left|\, \intQ{\vr  {\bf {F}}_{k, \ep } (\vr,\vt, \vu) \cdot \mathcal{B} \left[ \vr - (\vr)_Q \right] } \right| &\aleq f_k \| \mathcal{B} \left[ \vr - (\vr)_Q \right] \|_{L^\infty_x}
{\intQ{  \vr(1 + |\vu|)   } } \\ &\aleq f_k \| \vr \|_{L^\beta_x}
\intQ{ \vr (1 + |\vu|)  }.
\end{split}
\]
Hence we conclude that
\begin{equation} \label{wWS316}
\expe{ \left| \int_0^T \intQ{ p_\delta(\vr,\vt) \vr} \dt \right|^r } \leq c(r, \delta, \Lambda).
\end{equation}

\subsection{Compactness of the momentum in time}
\label{CMIT}

Using the same arguments as in Section \ref{EGA}, specifically, the Kolmogorov continuity criterion, we may deduce from the momentum equation (\ref{kolmo}) the estimate
\begin{equation} \label{Nbv7}
\expe{ \| \vr \vu \|^r_{C^{s}([0,T]; W^{-k,2}(Q; R^3) } } \aleq c(r) \ \mbox{for a certain}\ 0 < s(r) < \frac{1}{2}, \ k > \frac{3}{2}.
\end{equation}

\subsection{Stochastic compactness method}
\label{subsub:stochep}

The uniform bounds derived in the previous section are optimal in view of the energy method. We are ready to perform the limit $\ep \to 0$. We proceed in two steps.
Similarly to Section \ref{EGA}, we make use of the general results stated in Theorems \ref{RDT4} and \ref{RDT3} to pass to the standard probability space
$([0,1], \Ov{\mathfrak{B}[0,1]}, \mathfrak{L})$.
Then
we adapt the method known for the deterministic case to show compactness of the {temperatures} and the densities which is the main issue here.
\\
Applying Theorem \ref{wWP1a} we get a family of martingale solutions $(\vre, \vte, \vue)_{\ep > 0}$ of problem (\ref{Npde1}--\ref{Npde4}). Evoking the compactness
method used in Section \ref{EGA}, we get a new family of random variables $(\tvre, \tvte, \tvue)_{\ep > 0}$, together with the processes $\tilde{W}_\ep$ and
with the associated right-continuous complete filtration  $(\tilde{\mathfrak{F}}^\ep_t)_{t \geq 0}$
such that:
\begin{itemize}
\item
$\left[ \left( [0,1], \Ov{\mathfrak{B}[0,1]}, (\tilde{\mathfrak{F}}^\ep_t)_{t \geq 0} ,  \mathfrak{L},\right), \tvre , \tvte, \tvue,
\tilde{W}_\ep \right]$
is a weak martingale solution of problem (\ref{Npde1}--\ref{Npde4}), with the initial law $\Lambda$;
\item
the initial data $(\tilde{\vr}_{0,\ep}, \tilde{\vt}_{0,\ep}, \tilde{\vu}_{0,\ep})$ satisfy
\[
\begin{split}
\tilde{\vr}_{0,\ep} &\to \tilde{\vr}_0 \ \mbox{in} \ C^{2 + \nu}(\Ov{Q}),
\tilde{\vt}_{0,\ep}  \to \tilde{\vt}_0 \ \mbox{in}\ W^{1,2} \cap C(\Ov{Q}), \ {\vr}_0, {\vt}_0 > 0,\\
\tilde{\vu}_{0,\ep} &\to \tilde{\vu}_0 \ \mbox{in}\ L^2(Q, R^3),
\end{split}
\]
$\mathfrak{L}-$a.s.;
\item the functions $(\tvre, \tvte, \tvue)$ satisfy the bounds (\ref{Nbv1}--\ref{Nbv5}), (\ref{wWS316}), and (\ref{Nbv7}) $\mathfrak{L}-$a.s.
uniformly for $\ep \to 0$;
\item we have
\begin{equation} \label{Ncv1}
\begin{split}
\tvre &\to \tilde{\vr} \ \mbox{weakly in}\ L^{\beta + 1}((0,T) \times Q) \  \mbox{and weakly-(*) in}\ L^\infty(0,T; L^\beta(Q)) \ \mathfrak{L}-\mbox{a.s.},\\
\tvte &\to \tilde\vt \ \mbox{weakly in}\ L^2(0,T; W^{1,2}(Q)) \ \mbox{and weakly-(*) in} \
L^\infty (0,T; L^4(\Omega)) \ \mathfrak{L}-\mbox{a.s.},\\
\tvue &\to \tilde\vu \ \mbox{weakly in}\ L^2(0,T; W^{1,2}_0(Q;R^3))\ \mathfrak{L}-\mbox{a.s.},\\
\tvre \tvue &\to \tilde\vr \tilde\vu \ \mbox{in}\ C([0,T]; W^{-k,2}(Q; R^3)) \ k > \frac{3}{2},\  \mathfrak{L}-\mbox{a.s.};
\end{split}
\end{equation}
\item
the family $(\tilde\vr_\ep, \tilde\vt_\ep, \tilde\vu_\ep\tilde\vr_\ep, \Grad \tilde\vt_\ep, \Grad \tilde\vu_\ep)_{\ep > 0}$ generates a Young measure;
\item we have
\[
\tilde{W}_\ep \to \tilde{W} \ \mbox{in}\ C([0,T]; \mathfrak{U}_0 )\ \mathfrak L\mbox{-a.s.}
\]

\end{itemize}

\subsection{{Strong convergence of the temperature}}

We establish the strong convergence of the temperature fields $(\tvte)_{\ep > 0}$ exploiting \mbox{pathwise} the piece of information provided by the
deterministic entropy inequality (\ref{Napde3}). Indeed repeating step by step the arguments in \cite[Chapter 3, Section 3.6.2]{F} for an fixed $\omega \in (0,1)$ we can show that
\begin{equation} \label{Ncv2}
\tvte \to \tilde\vt \ \mbox{in}\ L^4((0,T) \times \tor) \ \mathfrak L\text{-a.s}.
\end{equation}
\\
In a similar way, we recover the weak formulation of the {\emph{renormalized equation of continuity}:}
\begin{equation} \label{Npde1c}
\begin{split}
&\int_0^T \intQ{ \left[ b(\tvre) \partial_t \psi + b(\tvre) \tvue \cdot \Grad \psi + \left(b(\tvre) - \tvre b'(\tvre) \right) \Div \tvue \psi  \right] }\dt \\
&=\ep \int_0^T \intQ{ b''(\tvre) |\Grad \tvre |^2 } \dt + \ep \int_0^T \intQ{ b'(\tvre) \Grad \tvre \Grad \psi } \dt   - \intQ{ b(\tilde{\vr}_{0,\ep}) \psi(0) }
\end{split}
\end{equation}
for any $\psi \in \DC([0,T) \times R^3)$ and any sufficiently smooth $b$.

\subsection{The limit in the stochastic integral}
\label{sec:limstochint}
To perform the limit in the momentum equation (\ref{Npde2}) we have to handle the stochastic integral. In view of Lemma \ref{RDL4}, we need to
show the strong convergence
\begin{align}\label{wWS329a}
\tilde\vre \vc{F}_\ep (\tvre,\tvte, \tvue) \to \Ov{ \tilde\vr \vc{F}(\tilde\vr, \tilde\vt, \tilde\vu) } \ \mbox{in}\ L^2(0,T; L^2(\mathfrak{U}_0; W^{-k,2}(Q))), \ k > \frac{3}{2}.
\end{align}
\\
First, we claim that it is enough to show that
\begin{align}\label{wWS329b}
\tilde\vre \vc{F}_\ep (\tvre,\tilde\vt, \tilde\vu) \to \Ov{ \vr \vc{F}(\tilde\vr, \tilde\vt, \tilde\vu) } \ \mbox{in}\ L^2(0,T; L^2(\mathfrak{U}_0; W^{-k,2}(Q))), \ k > \frac{3}{2}.
\end{align}
Indeed as the functions $\vc{F}_{k,\ep}$ are globally Lipschitz (uniformly for $\ep \to 0$, recall \eqref{P-1}), we have
\[
\left| \tilde\vre \vc{F}_{k,\ep} (\tvre,\tvte, \tvue) - \vre \vc{F}_{k,\ep} (\tvre, \tilde\vt, \tilde\vu) \right| \aleq
f_{k} \vre \left( | \tvte - \tilde\vt | + |\tvue - \tilde\vu| \right),
\]
where, by virtue of (\ref{Ncv1}), (\ref{Ncv2}),
\[
\left\| \tilde\vre (\tvte - \tilde\vt) \right\|_{L^1(Q)} \leq \| \tvre \|_{L^{4/3}(Q)} \| \tvte - \tilde\vt \|_{L^4(Q)} \to 0
\ \mbox{in}\ L^4(0,T)\ \mathfrak{L}\mbox{-a.s.}
\]
Similarly,
\begin{equation} \label{Ncv4}
\left\| \tilde\vre (\tvue - \tilde\vu) \right\|_{L^1(Q)} \leq \| \sqrt{\tvre} \|_{L^{2}(\tor)} \| \sqrt{\tvre}( \tvue - \tilde\vu ) \|_{L^2(Q; R^3)}.
\end{equation}
However, in view of (\ref{Ncv1}) again, we have
\begin{equation} \label{Ncv5}
\int_0^T \intQ{ \tvre |\tvue|^2 } \ \dt \to \int_0^T \intQ{ \tilde\vr |\tilde\vu|^2 } \dt;
\end{equation}
whence the right--hand side of (\ref{Ncv4}) tends to zero in $L^2(0,T)$ $\mathfrak{L}$-a.s.
\\
Now, convergence in (\ref{wWS329b}) follows the fact that $\tvre$ satisfy the renormalized equation (\ref{Npde1c}), specifically,
\[
b(\tvre) \to \Ov{b(\tilde\vr)}\ \mbox{in}\ C_w([0,T]; L^p(Q)) \ \mathfrak{L}-a.s. \ \mbox{for any bounded continuous}\ b,
\]
see \cite[Chapter 4]{BFHbook}.
\\
In view of Lemma \ref{RDL4}, we may pass to the limit in the approximate momentum equation (\ref{Npde2}) obtaining
\begin{equation}\label{LMeq}
\begin{split}
&\int_0^T \partial_t \psi \intQ{ \tilde\vr \tilde\vu \cdot \bfvarphi } \dt\\
&+\int_0^T \psi \intQ{ \tilde\varrho\tilde\bfu\otimes\tilde\bfu : \nabla \bfvarphi } \dt -\int_0^T \psi \intQ{ \mathbb S(\tilde\vartheta,\nabla\tilde\bfu) : \nabla\bfvarphi } \dt\\
&+\int_0^T \psi \intQ{ \Ov{ p_\delta(\tilde\varrho,\tilde\vartheta) } \diver\bfvarphi }\dt + \int_0^t \psi \intQ{ \Ov{ \tilde\varrho{\vc{F}}(\tilde\varrho, \tilde\vt , \tilde\bfu)} \cdot
\bfvarphi }\,\dif \tilde{W} \\&= - \intQ{ \tilde\vr_0 \tilde\bfu_0 \cdot \bfvarphi };
\end{split}
\end{equation}
for all $\psi \in \DC[0,T)$, $\bfvarphi\in \DC (Q;R^3)$ $\mathfrak{L}$-a.s.

\subsection{Strong convergence of the density}
\label{subsec:strongconvdensity}

In the first step, we proceed as in the proof of \eqref{wWS316} and test the momentum
equation \eqref{Npde2} by {$\psi(t) \zeta(x) \nabla \Delta^{-1}[1_Q \tvre]$},
where $\psi \in \DC(0,T)$, $\zeta \in \DC(Q)$ and $\Delta^{-1}$ is the inverse of the Laplacean
on the whole space $R^3$, cf. \cite[Chapter 3, Section 3.6.5]{F}. In the stochastic terms, we apply It\^{o}'s formula to the function $f(\rho,\bfq)=\int_{Q} \bfq\cdot \zeta \Grad \Delta^{-1}[1_Q\rho]\dx$. Similarly to \eqref{wWS310} we obtain the integral identity:
\begin{equation}\label{newEVF}
\int_0^T \intQ{ \psi \zeta \Big( p_\delta(\tvre, \tvte) \tvre -
\mathbb{S}(\tvte,\Grad \tvue) : {\cal
R}[1_Q \tvre] \Big) } \ \dt = \sum_{j=1}^8 I_{j, \ep} \ \mathfrak{L}-\mbox{a.s.}
\end{equation}
where
\[
\begin{split}
I_{1,\ep} &= - \ep \int_0^T \intQ{ \psi \zeta \ \tvre \tvue \cdot
\Grad \Delta^{-1}[ \Div (1_Q \Grad \tvre)] } \ \dt\\
I_{2,\ep} &= \ep \int_0^T \intQ{ \ep \psi \Grad (\tvre \tvue) \cdot \Grad \left(
\zeta \nabla \Delta^{-1}[1_Q \tvre] \right) }
\dt\\
I_{3, \ep} &= \int_0^T \intQ{ \psi \zeta \Big( \tvre \tvue \cdot
{\cal R} [ 1_Q \tvre \tvue ] - (\tvre \tvue \otimes \tvue):
{\cal R}[1_Q \tvre ] \Big)} \ \dt,
\\
I_{4, \ep} &= - \int_0^T \intQ{ \psi p_\delta(\tvre, \tvte) \Grad
\zeta \cdot \Grad \Delta^{-1}_x[1_Q \tvre ] } \ \dt,
\\
I_{5,\ep} &= \int_0^T \intQ{ \psi
\mathbb{S}(\tvte,\Grad \tvue) : \Grad \zeta
\otimes \Grad \Delta^{-1}_x[1_Q \tvre ] } \ \dt ,
\\
I_{6,\ep} &= - \int_0^T \intQ{ \psi (\tvre \tvue \otimes \tvue ) :
\Grad \zeta \otimes \Grad \Delta^{-1}_x[1_Q \tvre ] } \ \dt,
\\
I_{7,\ep} &= - \int_0^T \intQ{ \partial_t \psi \ \zeta \tvre \tvue
\cdot \Grad \Delta^{-1}_x[ 1_Q \tvre ] } \ \dt,
\end{split}
\]
and
\[
I_{8,\ep} = - \int_0^T \psi \intQ{\zeta \vr \vc{F}_{\ep} (\tvre, \tvte, \tvue) \cdot \Grad \Delta^{-1}_x[ 1_Q \tvre ] }
\ \D \tilde{W}_\ep
\]
Here, the symbol ${\cal R}$ stands for the \emph{double Riesz
transform}, defined componentwise as ${\cal R}_{i,j} =
\partial_{x_i} \Delta^{-1}
\partial_{x_j}$, cf. \cite[Chapter 3, Section 3.6.5]{F}.
\\
Applying the same treatment with $\psi(t) \zeta(x) \nabla \Delta^{-1}[1_Q \vr]$ as the multiplier
to the limit equation (\ref{LMeq}), we obtain
\begin{equation}\label{newEVFL}
\int_0^T \intQ{ \psi \zeta \Big( \Ov{p_\delta(\tilde\vr, \tilde\vt)} \vr  -
\mathbb{S}(\tilde\vt,\Grad \tilde\vu) : {\cal
R}[1_Q \tilde\vr] \Big) } \ \dt = \sum_{j=3}^8 I_{j} \ \mathfrak{L}\mbox{-a.s.}
\end{equation}
where
\[
\begin{split}
I_{3} &= \int_0^T \intQ{ \psi \zeta \Big( \tilde\vr \tilde\vu \cdot
{\cal R} [ 1_Q \vr \vu ] - (\tilde\vr \tilde\vu \otimes \tilde\vu):
{\cal R}[1_Q \vr ] \Big)} \ \dt,
\\
I_{4} &= - \int_0^T \intQ{ \psi \Ov{p_\delta(\tilde\vr, \tilde\vt)} \Grad
\zeta \cdot \Grad \Delta^{-1}_x[1_Q \tilde\vr ] } \ \dt,
\\
I_{5} &= \int_0^T \intQ{ \psi
\mathbb{S}(\tilde\vt,\Grad \tilde\vu) : \Grad \zeta
\otimes \Grad \Delta^{-1}_x[1_Q \tilde\vr ] } \ \dt ,
\\
I_{6} &= - \int_0^T \intQ{ \psi (\tilde\vr \tilde\vu \otimes \tilde\vu ) :
\Grad \zeta \otimes \Grad \Delta^{-1}_x[1_Q \tilde\vr ] } \ \dt,
\\
I_{7} &= - \int_0^T \intQ{ \partial_t \psi \ \zeta \tilde\vr \tilde\vu
\cdot \Grad \Delta^{-1}_x[ 1_Q \tilde\vr ] } \ \dt,
\end{split}
\]
and
\[
I_{8} = - \int_0^T \psi \intQ{\zeta \Ov{ \tilde\vr \vc{F} (\tilde\vr, \tilde\vt, \tilde\vu) } \cdot \Grad \Delta^{-1}_x[ 1_{Q} \tilde\vr ] }
\ \D \tilde{W}.
\]
\\
Now, using the deterministic arguments, exactly as in \cite[Chapter 3, Section 3.6.5]{F}, we may use the uniform bounds established in Section \ref{wWS3S1}, notably (\ref{Nbv6}) to show that
\[
I_{1,\ep}, \ I_{2,\ep} \to 0,\ I_{j, \ep} \to I_j \ \mbox{for}\ j=3,\dots,7 \ \mathfrak{L}-\mbox{a.s.}
\]
\\
Finally, we claim convergence of the stochastic integrals $I_{8,\ep} \to I_8$.
To this end, we use Lemma \ref{RDL4}. To begin, we observe that (\ref{wWS329a}), together with the available uniform bounds, implies
\begin{equation} \label{wWS329'}
\tilde\vre \vc{F}_\ep (\tvre, \tilde\vte, \tvue) \to \Ov{ \tilde\vr \vc{F}(\tilde\vr, \tilde\vt, \tilde\vu) } \quad \text{in}\quad L^2(0,T; L^2(\mathfrak{U}_0; W^{-1,2}(Q)))
\end{equation}
$\mathfrak{L}$-a.s. Moreover, we have
\begin{align}\label{eq:comrho}
\zeta \Delta^{-1}\nabla\tilde\varrho_\varepsilon\rightarrow \zeta \Delta^{-1}
\nabla\tilde\varrho\quad\text{in}\quad C([0,T]; W^{1,2}_0(Q;R^3))
\end{align}
by the compactness of the operator $\zeta \Delta^{-1}\nabla: L^p(R^3)\rightarrow W^{1,2}_0(Q)$ for $p>\frac{6}{5}$.
Combining \eqref{wWS329'} and \eqref{eq:comrho} with Lemma \ref{RDL4}
we conclude $I_{8,\ep} \to I_8$.\\
\\
{Thus we have shown that $\mathfrak{L}$-a.s}
\begin{equation}\label{eq:flux0}
\begin{split}
\int_0^T\int_{Q}& \psi \zeta \Big(p_\delta(\tilde\varrho_\varepsilon,\tilde\vartheta_\varepsilon)\tilde\varrho_\varepsilon
- \mathbb S(\tilde\vartheta_\varepsilon,\nabla \tilde\bfu_\varepsilon):\nabla\Delta^{-1}\Grad 1_Q \tvre \Big)\,\dif x\,\dif t\\
\to \int_0^T\int_{Q}& \psi \zeta \Big( \Ov{p_\delta(\tilde\varrho , \tilde\vartheta )} \tilde\varrho
- \mathbb S(\tilde\vartheta ,\nabla \tilde\bfu ):\nabla\Delta^{-1}\Grad 1_Q \tilde\vr \Big)\,\dif x\,\dif t.
\end{split}
\end{equation}
\\
{Relation (\ref{eq:flux0}) gives rise to the effective viscous flux identity discovered by Lions \cite{Li2} and, after a tedious and rather nonstandard manipulation, gives
rise to the strong (pointwise) convergence of the density, more specifically,}
\begin{equation} \label{Ndconv}
\intQ{ \tvre \log(\tvre)(\tau) } \to \intQ{ \tilde\vr \log (\tilde\vr)(\tau) } \ \mbox{for any}\ \tau \in [0,T]\ \mathfrak L\text{-a.s}.
\end{equation}
The arguments are purely deterministic and use only the renormalized equation of continuity (\ref{Npde1c}) and compactness ($\pas$) of the initial density distribution
$\tilde{\vr}_{\ep,0}$. A detailed proof is given in \cite[Chapter 3, Section 3.6.5]{F}; see \cite[Chap. 4.4]{BFHbook} for the barotropic stochastic case. As the function $\vr \mapsto \vr \log \vr$ is strictly convex,
relation (\ref{Ndconv}) implies (up to a subsequence) strong (a.a. pointwise) convergence of $(\tvre)_{\ep > 0}$.

\subsection{Conclusion}

As we have established strong convergence of the families $(\tvre)_{\ep > 0}$, $(\tvte)_{\ep > 0}$ it is a routine matter to pass to the limit in the
weak formulation of (\ref{Npde1}--\ref{Npde4}) to obtain the limit system (\ref{Napd1}--\ref{Napd4}). We have shown Theorem \ref{wWT3}.

\section{The limit in the artificial pressure}
\label{sec:vanishingpressure}

In this final section we let $\delta\rightarrow0$ in the approximate system (\ref{Napd1}--\ref{Napd4}) and complete the proof of the existence of solutions stated in Theorem \ref{thm:main}.  As in the preceding sections, the proof consists in {\bf (i)} showing uniform bounds independent of $\delta$, {\bf (ii)} applying the stochastic
compactness method based on Skorokhod representation theorem, {\bf (iii)} showing compactness of the temperature and the density by means of deterministic arguments.

\subsection{{Initial data}}

The initial data considered in Theorem \ref{wWT3} are quite regular. In order to achieve the generality of the initial law in Theorem \ref{thm:main}, we consider
a family of Borel probability measure $(\Lambda_\delta)_{\delta > 0}$ defined on $L^1(Q) \times L^1(Q) \times L^1(Q; R^3)$ such that
\[
\begin{split}
\Lambda_\delta &= \left\{ \vr_0 > 0,\ 0 < \underline{\vr} \leq \intO{ \vr_0 } \leq \Ov{\vr} ,\ \Grad \vr_0 \cdot \vc{n} = 0, \vr_0,\ \vt_0 > 0 \right\} = 1,\\
&\int_{L^1 \times L^1 \times L^1} \left[ \| \vr_0 \|^r_{C^{2 + \nu}_x} + \| \vt_0 \|_{W^{1,2}_x \cap C_x}^r + \| \vu_0 \|^r_{L^2_x} \right] {\rm d}
\Lambda_\delta \leq c(r, \delta) \ \mbox{for all}\ r \geq 1.
\end{split}
\]
Let $\Lambda$ be the law specified in \eqref{Ida}. We consider a sequence $(\Lambda_\delta)_{\delta > 0}$ such that
\[
\begin{split}
\Lambda_\delta &\to \Lambda \ \mbox{weakly-(*) in}\ \mathcal{M}^+ (L^1(Q) \times L^1(Q) \times L^1(Q; R^3)),\\
&\int_{L^1_x \times L^1_x \times L^1_x} \left\| \vr_0 |\vu_0|^2 + \vr_0 e_\delta(\vr_0, \vt_0) +
\vr_0 s_\delta (\vr_0, \vt_0) + \delta \vr_0^\beta \right\|_{L^1_x}^r {\rm d}
\Lambda_\delta \leq c(r) \ \mbox{for all}\ r \geq 1
\end{split}
\]
uniformly for $\delta \to 0$.

\subsection{Uniform energy bounds}
\label{wWS4S1}

We start with the energy estimates that basically mimick those obtained in Section~\ref{wWS3S1}. Let $(\vr,\vt,\vu)$ be a dissipative martingale solution to (\ref{Napd1}-\ref{Napd4})  constructed by means of Theorem \ref{wWT3}.
We intend to derive estimates which hold true uniformly in $\delta$.
The total dissipation balance (\ref{NTDB}), that can be derived exactly as in Section \ref{wWS3S1},  provides the following
uniform bounds, see \cite[Chapter 3, Section 3.7.1]{F} for details:

\begin{equation} \label{wWS47}
{\expe{ \left| \sup_{t \in [0,T]} \| \vr \|^{5/3}_{L^{5/3}_x} \right|^r } } +
\expe{ \left| \sup_{t \in [0,T]} \delta \|  \vr \|^\beta_{L^\beta_x} \right|^r }
\aleq c_1(r, \Lambda);
\end{equation}
\begin{equation} \label{wWS48}
\begin{split}
\expe{ \left| {\rm ess} \sup_{t \in [0,T]} \left\| \vr |\vu|^2 \right\|_{L^1_x} \right|^r +
{\left| \sup_{t \in [0,T]} \left\| \vr \vu \right\|^\frac{5}{4}_{L^{\frac{5}{4}}_x} \right|^r } } \aleq c_1(r, \Lambda );
\end{split}
\end{equation}
\begin{equation} \label{wWS49}
\expe{ \left\| \vu \right\|^{2r}_{L^2_tW^{1,2}_x} } \aleq c_2(r, \Lambda),
\end{equation}
\begin{equation} \label{wWS217aSS}
\expe{ \left| {\rm ess} \sup_{t \in [0,T]} \| \vt \|^4_{L^4_x} \right|^r }+\expe{\left\| \Grad \vt \right\|^{r}_{L^2((0,T)\times Q)} } \aleq c(r, \Lambda);
\end{equation}
\begin{equation} \label{wWS217SS}
\expe{ \left\| \frac{\kappa_\delta(\vartheta)}{\vartheta}\nabla\vartheta \right\|^{r}_{L^2((0,T)\times Q)} } \aleq c(r, \Lambda).
\end{equation}
\\
Finally, we report the conservation of mass principle
\begin{equation} \label{wWS411}
\| \vr(\tau, \cdot) \|_{L^1_x} = \intQ{ \vr(\tau, \cdot) } = \intQ{ \vr_0 } \leq \Ov{\vr} \ \mbox{for all}\ \tau\in[0,T].
\end{equation}

\subsection{Pressure estimates}
\label{wWS4S1S}

In order to derive refined estimates of the pressure, we apply a   method similar to  Section \ref{wWS3S2}. We consider
\[
\mathcal{B} \left[ b(\vr) - (b(\vr))_{Q} \right]
\]
as test function in the momentum equation (\ref{Napd2}). Here $b$ is a smooth function with moderate growth specified below.
Next, using the regularization method of DiPerna and Lions \cite{DL} we can show that $\vr$, $\vu$ satisfy the renormalized continuity equation
(cf. \cite[Chapter 2.2.5]{F}) holds
\begin{equation} \label{Npd1c}
\begin{split}
&\int_0^T \intQ{ \left[ b(\vr) \partial_t \psi + b(\vr) \vu \cdot \Grad \psi + \left(b(\vr) - \vr b'(\vr) \right) \Div \vu \psi  \right] }\dt \\
&=  - \intQ{ b(\vr_0) \psi(0) }
\end{split}
\end{equation}
for any test function $\psi \in \DC([0,T) \times R^3)$. In other words, the renormalized equation \eqref{Npd1c} provided $\vu$ has been extended to be zero
outside $Q$.
Consequently, we get
\begin{equation} \label{wWS412}
\begin{split}
\D \mathcal{B} [ b(\vr) - (b(\vr))_Q ] &= - \mathcal{B}[ \Div (b(\vr) \vu )] \dt \\
&\qquad+ \mathcal{B} \left[  \left( b(\vr) - b'(\vr) \vr \right) \Div \vu - \left( \left( b(\vr) - b'(\vr) \vr \right) \Div \vu
\right)_Q \right]\dt.
\end{split}
\end{equation}
Exactly as in Section \ref{wWS3S2} we deduce
\begin{equation} \label{wWS413}
\begin{split}
&\int_0^\tau \intQ{  p_\delta(\vr,\vt)
\left[ b(\vr) - (b(\vr) )_{Q} \right]    }\dt\\
&= \left[ \intQ{ \vr \vu \cdot \mathcal{B} \left[  b(\vr) - (b(\vr))_Q \right] } \right]_{t = 0}^{t=\tau}   - \int_0^\tau \intQ{ \vr \vu \otimes
\vu : \Grad \mathcal{B} \left[  b(\vr) - (b(\vr))_Q \right] } \dt \\
&\qquad+ \int_0^\tau \intQ{ \mathbb{S}(\vt,\Grad \vu) : \Grad \mathcal{B} \left[  b(\vr) - (b(\vr))_Q \right] } \dt- \int_0^\tau \intQ{ \vr \vu \cdot \mathcal{B} [ \Div( b(\vr) \vu)] } \dt \\
&\qquad+ \int_0^\tau \intQ{ \vr \vu \cdot \Del^{-1} \mathcal{B} \left[  \left(b'(\vr) \vr - b(\vr) \right) \Div \vu -
\left( \left(b'(\vr) \vr - b(\vr) \right) \Div \vu \right)_Q  \right] }\,\dif t
\\
&\qquad- \int_0^\tau \intQ{ \vr  {\vc {F}} (\vr,\vt, \vu) \cdot \mathcal{B} \left[  b(\vr) - (b(\vr))_Q \right] } \, \D \W \equiv \sum_{i=1}^6 I_i.
\end{split}
\end{equation}
Now, consider
$
b(\vr) = \vr^\alpha
$
where $0 < \alpha <  \frac{1}{3}$.
\\
The estimates for the ``deterministic'' integrals can be obtained exactly as in \cite[Sec. 4.5.2]{BFHbook} using the uniform bounds from the previous section.  We only give the details for the stochastic integral.
As a consequence of
(\ref{wWS411}), standard $L^q$-estimates for the inverse Laplacian, and the embedding relation $W^{1, q}(Q) \hookrightarrow C(\Ov{Q})$ for $q>3$,
\begin{equation} \label{wWS414}
\sup_{t \in [0,T]} \left\| \mathcal{B} \left[ b(\vr) - (b(\vr))_Q \right] \right\|_{L^\infty_x} \aleq 1 \ \pas
\end{equation}
where the norm is controlled by a deterministic constant proportional to $\Ov{\vr}$.
Next, we have by the Burgholder--Davis--Gundy inequality
\[
\begin{split}
&\expe{ \left| \int_0^\tau \intQ{ \vr  {\vc{F}} (\vr,\vt, \vu) \cdot \mathcal{B} \left[ b(\vr) - (b(\vr))_Q \right] } \, \D \W \right|^r } \\
&\leq \expe{ \sup_{t\in[0,\tau]}\left| \int_0^t \intQ{ \vr  {\vc {F}} (\vr, \vt, \vu) \cdot \mathcal{B} \left[ \vr^\alpha
- (\vr^\alpha)_Q \right] } \, \D \W \right|^r } \\
&\aleq \expe{ \int_0^\tau \sum_{k=1}^\8 \left|\, \intQ{\vr  {\bf {F}}_{k} (\vr, \vt, \vu) \cdot \mathcal{B} \left[ \vr^\alpha
- (\vr^\alpha)_Q \right] } \right|^2 \dif t}^{r/2},
\end{split}
\]
where, due to \eqref{P-1} and (\ref{wWS411}), \eqref{wWS414},
\[
\begin{split}
\left|\, \intQ{\vr  {\bf {F}}_{k} (\vr, \vt, \vu) \cdot \mathcal{B} \left[ \vr^\alpha
- (\vr^\alpha)_Q \right] } \right| &\aleq f_k \left\| \mathcal{B} \left[ \vr^\alpha
- (\vr^\alpha)_Q \right] \right\|_{L^\infty_x}
{\intQ{(\vr + \vr|\vu|) }}\\ \aleq c(\Ov{\vr}) f_k \intQ{ (\vr + \vr|\vu|) }.
\end{split}
\]
Consequently, we may infer, similarly to Section \ref{wWS3S2}, that
\begin{equation} \label{wWS422}
{\expe{ \left|\int_0^T  \intQ{  p_\delta(\vr,\vt) \vr^\alpha  } \dt \right|^r } \aleq c(\Ov{\vr},r) }
\end{equation}
for a certain $\alpha > 0$.

\subsection{Stochastic compactness method}
\label{wWS4S2}

Exactly as in Section \ref{CMIT}, we show the estimate
\begin{equation} \label{Nbv7bis}
\expe{ \| \vr \vu \|^r_{C^{s}([0,T]; W^{-k,2}(\Q)) } } \aleq c(r) \ \mbox{for a certain}\ 0 < s(r) < \frac{1}{2},\ k>\frac{3}{2}.
\end{equation}
\\
Now, following the arguments of Section \ref{subsub:stochep} based on Theorems
\ref{RDT4} and \ref{RDT3} we pass to the Skorokhod representation on the standard probability space
$([0,1], \Ov{\mathfrak{B}[0,1]}, \mathfrak{L})$ obtaining
a family of random variables $(\tvrd, \tvtd, \tvud)_{\delta > 0}$, together with the processes $\tilde{W}_\delta$ and
with the associated right-continuous complete filtration  $(\tilde{\mathfrak{F}}^\delta_t)_{t \geq 0}$
such that:
\begin{itemize}
\item
$\left[ \left( [0,1], \Ov{\mathfrak{B}[0,1]}, (\tilde{\mathfrak{F}}^\delta_t)_{t \geq 0} ,  \mathfrak{L}\right), \tvrd , \tvtd, \tvud,
\tilde{W}_\delta \right]$
is a weak martingale solution of problem (\ref{Napd1}--\ref{Napd4}), with the initial law $\Lambda_\delta$;
\item
the initial data satisfy $(\tilde{\vr}_{0,\delta}, \tilde{\vt}_{0,\delta}, \tilde{\vu}_{0,\delta})$ satisfy
\[
\begin{split}
\tilde{\vr}_{0,\delta} &\to \tilde{\vr}_0 \ \mbox{in} \ L^{1}(Q),\ \vr_0 \geq 0\\
\tilde{\vt}_{0,\delta}  &\to \tilde{\vt}_0 \ \mbox{in}\ L^1(Q),\ {\vt}_0 > 0,\\
\tilde{\vu}_{0,\delta} &\to \tilde{\vu}_0 \ \mbox{in}\ L^1(Q, R^3); 
\end{split}
\]
$\mathfrak{L}$-a.s.;
\item the functions $(\tvrd, \tvtd, \tvud)$ satisfy the bounds (\ref{wWS47}--\ref{wWS411}), (\ref{wWS422}), and (\ref{Nbv7bis}) $\mathfrak{L} \ \pas$
uniformly for $\delta \to \infty$;
\item we have
\begin{equation} \label{Ncv1bis}
\begin{split}
\tvrd &\to \tilde{\vr} \ \mbox{weakly in}\ L^{q}((0,T) \times Q) \ \ \mbox{for some}\
q > \frac{5}{3},  \mbox{weakly-(*) in}\ L^\infty(0,T; L^{5/3}(Q)),\\
\tvtd &\to \tilde\vt \ \mbox{weakly in}\ L^2(0,T; W^{1,2}(Q)), \ \mbox{weakly-(*) in} \
L^\infty (0,T; L^4(Q)),\\
\tvud &\to \tilde\vu \ \mbox{weakly in}\ L^2(0,T; W^{1,2}_0(Q;R^3)),\\
&\tvrd \tvud \to \tilde\vr \tilde\vu \ \mbox{in}\ C([0,T]; W^{-k,2}(Q; R^3)) \ k > \frac{3}{2},
\end{split}
\end{equation}
$\mathfrak{L}$-a.s.;
\item
the sequence $(\tvrd, \tvtd, \tvud, \Grad \tvtd, \Grad \tvud)_{\delta > 0}$ generates a Young measure;
\item we have
\[
\tilde{W}_\delta \to \tilde{W} \ \mbox{in}\ C([0,T]; \mathfrak{U}_0 )\ \mbox{a.s.}.
\]

\end{itemize}

\subsection{Compactness of the temperature and the density}

Following the strategy from Section
\ref{sec:vanishingviscosity} we can use the deterministic arguments specified in \cite[Chapter 3, Section 3.7.3]{F}
to show that
\bFormula{m257}
\tilde\vtd \to \tilde\vt \ \mbox{in} \ L^4((0,T) \times Q)\ \mathfrak{L}-\mbox{a.s.}
\eF
\\
Next, we easily adapt the arguments of Section \ref{subsec:strongconvdensity}
to the multipliers of the form
\[
\xi (t) \zeta(x) \Delta^{-1} \Grad \left( 1_Q T_k(\tvrd) \right),\ \mbox{and}
\ \xi (t) \zeta(x) \Delta^{-1} \Grad \left( 1_Q \Ov{T_k(\vr)} \right),
\]
where $T_k$ is a family of cut--off functions,
\[
\begin{split}
T_k(r) &= k T \left( \frac{r}{k} \right),\\
&T \in C^\infty[0, \infty),\
T(r) = \left\{ \begin{array}{l}  r \ \mbox{for}\ 0 \leq r \leq 1, \\ T''(r) \leq 0 \ \mbox{for}\
r \in (1,3), \\ 2 \ \mbox{for}\ r \geq 3.
\end{array} \right.
\end{split}
\]
\\
Using the same arguments as in Section \ref{subsec:strongconvdensity}, in particular when handling the limit in the stochastic integral,
we deduce {the effective viscous flux identity in the form}:
\begin{equation} \label{evfid}
\begin{split}
\Ov{p_M(\vr, \vt) T_k(\vr)}  &-
\left( \frac{4}{3} \mu(\vt) + \eta (\vt) \right)
\Ov{ T_k(\vr) \Div \vu) }  \\
&
=  \Ov{p_M(\vr, \vt)}\ \Ov{T_k(\vr) } - \left( \frac{4}{3} \mu(\vt) + \eta (\vt) \right) \Ov{T_k(\vr)} \Div\vu) \ \mathfrak{L}-\mbox{a.s.}
\end{split}
\end{equation}
see also \cite[Chapter 3, Section 3.7.4]{F} and \cite[Chap. 4.5]{BFHbook}.
\\
Relation (\ref{evfid}) in fact yields strong (a.a. pointwise) convergence of the density. Here, the argument is more involved than in Section \ref{subsec:strongconvdensity},
however, still purely deterministic. The reader may find a detailed proof in \cite[Chapter 3.7.5]{F}.
\\
Having established the strong convergence of the approximate densities, we easily perform the limit $\delta \to 0$ in the system (\ref{Napd1}--\ref{Napd4}) to recover the original
problem. We have proved Theorem \ref{thm:main}.

\section{{Conclusion, possible extensions, absence of stationary solutions}}
\label{conc}

For the sake of simplicity, we have assumed that the viscosity coefficients behave like linear functions of the temperature, see (\ref{m105}).
In particular, this gives rise to the velocity field living in the ``standard'' energy space $L^2(0,T; W^{1,2}_0(Q; R^3)$. A more elaborate treatment in the spirit of
\cite[Chapter 3]{F} would yield the existence result in the physically more relevant range of viscosities, namely,
\[
0< \underline{\mu}(1 + \vt^\alpha) \leq \mu(\vt) \leq \Ov{\mu} (1 + \vt^\alpha),\ \frac{2}{5} < \alpha \leq 1.
\]
Further generalizations of the constitutive assumptions as well as the underlying spatial domain and boundary conditions in the spirit of \cite[Chapter 3]{F}
are possible.

Finally, it is worth--noting that, unlike the simplified barotropic problem (see \cite{BFHM}), the full Navier--Stokes--Fourier system does not admit, in general,
\emph{stationary solutions} (that is, solutions which law is independent of time). This is a simple consequence of the total energy balance (\ref{eq0202bisE}). Indeed passing to expectations and assuming 
the total energy to be stationary, we get 
\[
\expe{ \int_{\tau_1}^{\tau_2} \intQ{ \frac{1}{2} \vr \sum_{k \geq 1} | \vc{F}_k (\vr, \vt, \vu ) |^2 } \dt 
+ \int_{\tau_1}^{\tau_2} \intQ{ \vr H } \dt } = 0
\]
for a.a. $0 \leq \tau_1 < \tau_2$ $\pas$. Thus if $H \geq 0$, the driving force must vanish identically a.s. along the paths of any stationary solution. This is in strong contrast to the barotropic case, where the existence of (non-trivial) stationary martingale solutions has been recently shown in \cite{BFHM}.

\medskip

\centerline{\bf Acknowledgement}

The research of E.F.~leading to these results has received funding from the
European Research Council under the European Union's Seventh
Framework Programme (FP7/2007-2013)/ ERC Grant Agreement
320078. The Institute of Mathematics of the Academy of Sciences of
the Czech Republic is supported by RVO:67985840.


\begin{thebibliography}{19}
\bibitem{Amann1}
H.~Amann.
\newblock Nonhomogeneous linear and quasilinear elliptic and parabolic boundary
  value problems.
\newblock In {\em Function spaces, differential operators and nonlinear
  analysis ({F}riedrichroda, 1992)}, volume~{\bf 133} of {\em Teubner-Texte
  Math.}, pages 9--126. Teubner, Stuttgart, 1993.


\bibitem{bensoussan} A. Bensoussan, Stochastic Navier-Stokes equations, Acta Appl. Math., 38 (3) (1995), 267-304.
\bibitem{BeTe} A. Bensoussan, R. Temam, \'{E}quations stochastiques du type Navier-Stokes. J. Functional Analysis 13 (1973), 195--222.
\bibitem{berthvovelle} F. Berthelin, J. Vovelle, Stochastic isentropic Euler equations, arXiv:1310.8093.
\bibitem{BFHa} D. Breit, E. Feireisl \& M. Hofmanov\'{a}: \emph{Incompressible limit for compressible fluids with stochastic forcing.} Arch. Rational Mech. Anal. 222, 895--926. (2016)
\bibitem{BFH} D. Breit, E. Feireisl \& M. Hofmanov\'{a}: \emph{Compressible fluids driven by stochastic forcing: The relative energy inequality and applications.}
Commun. Math. Phys. 350, 443--473. (2017)
\bibitem{BFHbook} D. Breit, E. Feireisl \& M. Hofmanov\'{a}: \emph{Stochastically forced compressible fluid flows.}
De Gruyter (in press)
%\bibitem{BFH2} D. Breit, E. Feireisl \& M. Hofmanov\'{a}: \emph{Local strong solutions to the stochastic compressible
%Navier--Stokes system.} Preprint
%at arXiv:1606.05441v1
\bibitem{BFHM} D. Breit, E. Feireisl, M. Hofmanov\'{a} \& B. Maslowski: \emph{Stationary solutions to the compressible Navier--Stokes system driven by stochastic forcing.} Preprint at arXiv:1703.03177v1 
\bibitem{BrHo} D. Breit \& M. Hofmanov\'a: Stochastic Navier--Stokes equations for compressible fluids. Indiana Univ. Math. J. 65, 1183--1250. (2016)
%\bibitem{b2} Z. Brze\'zniak, Stochastic partial differential equations in M-type 2 Banach spaces, Potential Anal. 4 (1995), 1-45.
\bibitem{on1} Z. Brze\'zniak, M. Ondrej\'at, Strong solutions to stochastic wave equations with values in Riemannian manifolds, J. Funct. Anal. 253 (2007) 449-481.
\bibitem{daprato} G. Da Prato, J. Zabczyk, Stochastic Equations in Infinite Dimensions, Encyclopedia Math. Appl., vol. 44, Cambridge University Press, Cambridge, 1992.
%\bibitem{debusergo} A. Debussche, Ergodicity results for the stochastic Navier-Stokes equations: an introduction, Lecture Notes in Math., Springer, Heidelberg, 2073, 23--108 (2013).
\bibitem{debussche1} A. Debussche, N. Glatt-Holtz, R. Temam, Local Martingale and Pathwise Solutions for an Abstract Fluids Model, Physica D 14-15 (2011), 1123-1144.
%\bibitem{DHV} A. Debussche, M. Hofmanov\'a, J. Vovelle: Degenerate parabolic stochastic partial differential equations: Quasilinear case, arXiv:1309.5817.
\bibitem{DL} R. J. DiPerna and P.-L. Lions, Ordinary differential equations, transport theory and
Sobolev spaces, Invent. Math. 98 (1989), 511--547.
%\bibitem{dudley} R.\,M. Dudley, Real analysis and probability, Cambridge University Press, Cambridge 2002.
\bibitem{Eb} Ebin, David G. Viscous fluids in a domain with frictionless boundary. Global analysis--analysis on manifolds, 93--110, Teubner-Texte Math., 57, Teubner, Leipzig, 1983.
\bibitem{fei3} E. Feireisl, Dynamics of Compressible Flow, Oxford Lecture  Series in Mathematics and its Applications, Oxford University Press, Oxford, 2004.
%\bibitem{FeHoMaNo} E. Feireisl, R. Ho{\v s}ek, D. Maltese, D. and A. Novotn{\' y}.
%Convergence and error estimates for bounded numerical solutions of the barotropic {N}avier-{S}tokes system.



\bibitem{feireisl2} E. Feireisl, B. Maslowski, A. Novotn\'y, Compressible fluid flows driven by stochastic forcing, J. Differential Equations 254 (2013) 1342-1358.
\bibitem{F} E. Feireisl, A. Novotn\'{y} (2009): \emph{Singular limits in thermodynamics of viscous fluids.} Birkh\"auser-Verlag, Basel, 2009.
\bibitem{FeiNov10} E.Feireisl, A.Novotn{\' y}, Weak-strong uniqueness property for the full {N}avier-{S}tokes-{F}ourier system ,Arch. Rational Mech. Anal.
2014 (2012) 683--706.
\bibitem{feireisl1} E. Feireisl, A. Novotn\'y, H. Petzeltov\'a, On the existence of globally defined weak solutions to the Navier-Stokes equations, J. Math. Fluid. Mech. 3 (2001) 358-392.
\bibitem{FP20} E. Feireisl, H. Petzeltov\'a, H., On the long-time behaviour of solutions to the
              {N}avier-{S}tokes-{F}ourier system with a time-dependent
              driving force, J. Dynam. Differential Equations, 19 (2007) 685--707.

\bibitem{FeSu2015_N1} E. Feireisl, Y. Sun, Conditional regularity of very weak solutions to the
              {N}avier-{S}tokes-{F}ourier system. In Recent advances in partial differential equations and
              applications, Contemp. Math. 666, pages 179--199, Amer. Math. Soc., Providence, RI, 2016

\bibitem{Fl2} F. Flandoli (2008): An introduction to 3D stochastic fluid dynamics. In SPDE in Hydrodynamic:
Recent Progress and Prospects. Lecture Notes in Math. 1942 51--150. Springer,
Berlin.
%\bibitem{FlGa} F. Flandoli, D. G\c{a}tarek, Martingale and stationary solutions for stochastic Navier--Stokes equations, Probab. Theory Related
%Fields 102 (1995) 367--391.
\bibitem{franco} F. Flandoli, A. Mahalov, Stochastic three-dimensional rotating Navier-Stokes equations: averaging, convergence and regularity, Arch. Rational Mech. Anal. 205 (2012) 195--237.
\bibitem{Gra}  Grafakos, Loukas Classical Fourier analysis. Second edition. Graduate Texts in Mathematics, 249. Springer, New York, 2008. xvi+489 pp.
%\bibitem{FlRo} F. Flandoli, M. Romito, Markov selections for 3D stochastic Navier--Stokes equation, Probab. Theory Related Fields 140
%(2008) 407--458.
%\bibitem{FMS} Frehse, J.; Steinhauer, M.; Weigant, W. The Dirichlet problem for steady viscous compressible flow in three dimensions. J. Math. Pures Appl. (9) 97 (2012), no. 2, 85--97.
%\bibitem{Fu} H. Fujita Yashima, Equations de Navier-Stokes stochastiques non homog\`enes et applications, Tesi di Perfezionamento, Scuola Normale Superiore, Pisa, 1992.
\bibitem{HIPR}
M.~Hieber and J.~Pr{\"u}ss.
\newblock Heat kernels and maximal {$L^p$-$L^q$} estimates for parabolic
  evolution equations.
\newblock {\em Commun. Partial Differential Equations}, {\bf
  22}(9,10):1647--1669, 1997.
\bibitem{jakubow} A. Jakubowski, The almost sure Skorokhod representation for subsequences in nonmetric spaces, Teor. Veroyatnost. i Primenen 42 (1997), no. 1, 209-216; translation in Theory Probab. Appl. 42 (1997), no. 1, 167-174 (1998).
\bibitem{krylov} I. Gy\"{o}ngy, N. Krylov, Existence of strong solutions for It\^o's stochastic equations via approximations, Probab. Theory Related Fields 105 (2) (1996) 143-158.
%\bibitem{karatzas} I. Karatzas, S. Shreve, Brownian Motion and Stochastic Calculus, 2nd ed., Springer, New York, 1991.
%\bibitem{Li1}P.\,L. Lions, Mathematical topics in fluid mechanics. Vol. 1. Incompressible models. Oxford Lecture Series in Mathematics and its Applications, 3. Oxford Science Publications, The Clarendon Press, Oxford University Press, New York, 1996.
\bibitem{LSU} O. A. Ladyzhenskaya, V. A. Solonnikov, and N. N. Ural'tseva.
Linear and quasilinear equations of parabolic type. Translated from
the Russian by S. Smith. Translations of Mathematical Monographs,
Vol. 23. American Mathematical Society, Providence, R.I., 1967.
\bibitem{Li2} P.\,L. Lions, Mathematical topics in fluid mechanics. Vol. 2. Compressible models. Oxford Lecture Series in Mathematics and its Applications, 10. Oxford Science Publications, The Clarendon Press, Oxford University Press, New York, 1998.
\bibitem{lu} A. Lunardi, Analytic semigroups and optimal regularity in parabolic problems, Birkh\"auser, Berlin, 1995.
%\bibitem{Ma} J. Mattingly, On recent progress for the stochastic Navier Stokes equations, in: Journ\'{e}e Equations aux Deriv\'{e}es Partielles, Forges-les-Eaux, 2003.
%\bibitem{malek} J. M\'alek, J. Ne\v{c}as, M. Rokyta, M. R\r{u}\v{z}i\v{c}ka, Weak and Measure-valued Solutions to Evolutionary PDEs, Chapman \& Hall, London, Weinheim, New York, 1996.
\bibitem{romeo} P.R. Mensah: Existence of martingale solutions and the incompressible limit for stochastic compressible flows on the whole space. Ann. Mat. Pura Appl., DOI 10.1007/s10231-017-0656-1.
\bibitem{mikul} R. Mikulevicius, B.\,L. Rozovskii, Stochastic Navier-Stokes equations for turbulent flows, SIAM J. Math. Anal. 35 (5) (2004) 1250-1310.
%\bibitem{veraar} J.\,M.\,A.\,M. van Neerven, M.\,C. Veraar, L. Weis, Stochastic integration in UMD Banach spaces, Annals Probab. 35 (2007), 1438-1478.
\bibitem{novot} A. Novotn\'y, I. Stra\v{s}kraba, Introduction to the Mathematical Theory of Compressible Flow, Oxford Lecture Series in Mathematics and its Applications, Oxford University Press, Oxford, 2004.
\bibitem{on2} M. Ondrej\'at, Stochastic nonlinear wave equations in local Sobolev spaces, Electronic Journal of Probability 15 (33) (2010) 1041-1091.
%\bibitem{ondrejat3} M. Ondrej\'at, Uniqueness for stochastic evolution equations in Banach spaces, Dissertationes Mathematicae 426 (2004), 1-63.
%\bibitem{onsebr} M. Ondrej\'at, J. Seidler, Z. Brze\'zniak, Invariant measures for stochastic nonlinear beam and wave equations. J. Differ. Equ. 260 (2016), 4157--4179.
\bibitem{Pe} P. Pedregal. Parametrized measures and variational principles. Birkh\"auser,
Basel, 1997.
%\bibitem{PlVa} Plotnikov, P. I.; Weigant, W. Steady 3D viscous compressible flows with adiabatic exponent $\gamma\in(1,\infty)$. J. Math. Pures Appl. (9) 104 (2015), no. 1, 58--82.
\bibitem{pr07} C. Pr\'ev\^ot, M. R\"ockner, A concise course on stochastic partial differential equations, vol. 1905 of Lecture Notes in Math., Springer, Berlin, 2007.

\bibitem{2015arXiv150400951S}
S.~{Smith}.
\newblock {Random Perturbations of Viscous Compressible Fluids: Global
  Existence of Weak Solutions}.
\newblock arXiv:1504.00951v1.

\bibitem{SmTr} S. Smith, K. Trivisa.
\newblock {The stochastic Navier--Stokes equations for heat-conducting, compressible fluids: global existence of weak solutions.}
\newblock J. Evol. Equ. https://doi.org/10.1007/s00028-017-0407-1

\bibitem{MR1760377}
E.~Tornatore and H. ~F.~Yashima.
\newblock One-dimensional stochastic equations for a viscous barotropic gas.
\newblock {\em Ricerche Mat.}, 46(2), 255--283 (1998), 1997.

\bibitem{MR1807944}
E.~Tornatore.
\newblock Global solution of bi-dimensional stochastic equation for a viscous
  gas.
\newblock {\em NoDEA Nonlinear Differential Equations Appl.}, 7(4), 343--360,
  2000.
%\bibitem{simon} J. Simon, Compact sets in the space $L^p(0,T;B)$, Ann. Mat. Pura Appl. (4) 146 (1987), 65-96.
%\bibitem{Sa} M. Sango, Density dependent stochastic Navier--Stokes equations with non-Lipschitz random forcing, Rev. Math. Phys. 22 (2010) 669--697.
%\bibitem{TeYo} Y. Terasawa, N. Yoshida (2011): Stochastic power-law fluids: existence and uniqueness of weak solutions. The Annals of Applied Probability, Vol. 21, No. 5, 1827--1859.
%\bibitem{To} E. Tornatore (2000): Global solution of bi-dimensional stochastic equation for a viscous gas, NoDEA Nonlinear Differential Equations Appl. 7 (4), 343--360.
%\bibitem{Yo} N. Yoshida (2012): Stochastic Shear thickenning fluids: strong convergence of the  Galerkin approximation and the energy inequality. The Annals of Applied Probability 2012, Vol. 22, No. 3, 1215--1242
\end{thebibliography}
\end{document}